\documentclass[11pt]{article}
\usepackage[a4paper,top=2cm,bottom=2cm,left=1.5cm,right=1.5cm]{geometry}
\usepackage[utf8]{inputenc}
\usepackage{amsmath,amsfonts,amssymb,amsthm}
\usepackage{natbib,hyperref}
\usepackage{array}
\usepackage{xcolor}
\usepackage{tikz}
\usepackage{graphicx}
\usepackage{stmaryrd}
\usepackage{txfonts}
\usepackage{placeins}
\usepackage{float}
\usepackage{booktabs}
\graphicspath{{FiguresV2/}}

\newcommand{\Esp}{\mathbb{E}}

\newcommand{\Ncal}{\mathcal{N}}

\newcommand{\Var}{\mathbb{V}}
\newcommand{\prop}{f}
\newtheorem{lemma}{Lemma}
\newtheorem{theorem}{Theorem}
\newtheorem{corollary}{Corollary}
\theoremstyle{remark}
\newtheorem*{remark}{Remark}

\newcommand{\gray}[1]{\textcolor{gray}{#1}}
\newcommand{\purple}[1]{\textcolor{black}{#1}}

\newcommand{\SR}[2]{\gray{#1}\purple{#2}}

\title{Online and Offline Robust Multivariate Linear Regression}
\author{Antoine Godichon-Baggioni$^{(1)}$, Stéphane Robin$^{(1)}$, Laure Sansonnet$^{(2,1)}$ \\
${(1)}$ LPSM, UMR 8001, Sorbonne Université, France \\
${(2)}$ Université Paris-Saclay, AgroParisTech, INRAE, UMR MIA Paris-Saclay, France}
\date{}

\begin{document}

\maketitle

\begin{abstract}
We consider the robust estimation of the parameters of multivariate Gaussian linear regression models. To this aim we consider robust version of the usual (Mahalanobis) least-square criterion, with or without Ridge regularization.  {We introduce two methods: i) online  stochastic gradient descent algorithms and their averaged versions and (ii) offline fix-point algorithms, both used for the usual and Mahalanobis least-square criteria and their regularized verions.} Under weak assumptions, we prove the asymptotic normality of the resulting estimates. Because the variance matrix of the noise is usually unknown, we propose to plug a robust estimate of it in the Mahalanobis-based stochastic gradient descent algorithms. We show, on synthetic data, the dramatic gain in terms of robustness of the proposed estimates as compared to the classical least-square ones. Well also show the computational efficiency of the online versions of the proposed algorithms. All the proposed algorithms are implemented in the \texttt{R} package \texttt{RobRegression} available on CRAN.
\end{abstract}

\section{Introduction}

In this paper, we consider the classical multivariate Gaussian linear model where both the explanatory vector $X$ and the response $Y$ are multivariate. Namely, the model is defined for any couple of random vectors $(X,Y) \in \mathbb{R}^{p} \times \mathbb{R}^{q}$ (for some $p,q \in \mathbb{N}$, with $ {p,} q> 1$)   by
\begin{equation}\label{def::model} 
Y  =   \beta^{*}X    + \epsilon ,
\end{equation}
where $\beta^{*} \in \mathcal{M}_{q,p} (\mathbb{R})$ with $\mathcal{M}_{q,p} (\mathbb{R})$  the space of $q\times p$ real matrices. In addition, the noise vector $\epsilon $  and the explanatory vector $X$ are supposed to be independent and  {the noise vector is supposed to be Gaussian:} $\epsilon  \sim \mathcal{N}_q \left( 0 , \Sigma \right)$  {with positive definite covariance matrix $\Sigma \in \mathcal{M}_{q,q}(\mathbb{R})$}. 

\textbf{Least-square based estimation.}
A  usual way  {to estimate} the parameter $\beta^{*}$ is to see it as the solution of the following (convex) stochastic  optimization problem 
\begin{equation}\label{def::OLS} 
\beta^{*} = \arg\min_{\beta \in \mathcal{M}_{q,p}(\mathbb{R})} \mathbb{E}\left[ \left\| Y - \beta X \right\|^{2}   \right],
\end{equation}
where $\|\cdot\|$ stands for the $\ell_2$-norm. Then, several methods can be used to estimate it. A most usual one is to consider the empirical function generated by the sample before approximating its minimizer, \emph{i.e.}\,the $M$-estimate \citep{van2000asymptotic,MR2238141},  with the help of usual convex optimisation methods \citep{boyd2004convex}.
Although these methods are known to be efficient, they may encounter computational troubles when dealing with large samples of data lying in high dimensional spaces and arriving sequentially. 
To overcome {these limitations}, stochastic gradient algorithms \citep{robbins1951} have become hegemonic and are deeply studied \citep[see, among others, ][]{Duf97,pelletier1998almost,bach2013non,gadat2017optimal}. Indeed, they can deal sequentially with the data and so, with low computational costs per iteration. In addition, their averaged versions are known, under regularity assumptions, to be asymptotically efficient \citep{ruppert1988efficient,PolyakJud92,Pel00}.

 {Alternatively},   to account for the variance structure of the noise, one  {may} consider the Mahalanobis norm between $Y$ and its prediction. More precisely, one can see $\beta^{*}$ as
\begin{equation}\label{def::WLS} 
\beta^{*} = \arg\min_{\beta \in \mathcal{M}_{q,p}(\mathbb{R})} \mathbb{E}\left[ \left\| Y - \beta X \right\|_{\Sigma^{-1}}^{2}   \right],
\end{equation}
where $\| y \|_{\Sigma^{-1}} = \sqrt{ y^{T}\Sigma^{-1}y}$ denotes the Mahalanobis norm of $y\in \mathbb{R}^{q}$.
The main difficulty here is that $\Sigma$ is generally unknown.
 {$\Sigma^{-1}$ can be estimated by the inverse of the empirical covariance matrix of the residuals (when invertible). In practice, regularized estimates are often preferred, either assuming that $\Sigma$ has a particular structure  \citep{PERROTDOCKES2018,BlockCov2022}, using Tikhonov regularization \citep{HTF09}, or assuming sparsity using (an adaptation of) the Graphical Lasso \citep{FHT07,YuanLin2007,Banerjee2008,Rothman2010}.}

\textbf{{Robust estimation}.}
{Our work is motivated by the fact that} all these methods are very sensitive to the possible presence of outliers, {because} they minimize squared errors. To reduce this sensitivity, one can consider the minimization of the $\ell^{1}$-loss (see \cite{van2000asymptotic,MR2238141,pesme2020online} for the univariate case), \emph{i.e.}\,one can see $\beta^{*}$ as
\begin{equation}\label{def::robustLS} 
\beta^{*} = \arg\min_{\beta \in \mathcal{M}_{q,p}(\mathbb{R})} \mathbb{E}\left[ \left\| Y - \beta X \right\|   \right]  \quad \quad \text{or} \quad \quad  \beta^{*} = \arg\min_{\beta \in \mathcal{M}_{q,p}(\mathbb{R})} \mathbb{E}\left[ \left\| Y - \beta X \right\|_{\Sigma^{-1}}   \right] .
\end{equation}

\textbf{Objective.}
A first objective of this paper is to propose robust methods based on the minimization of these respective functions.  {We propose here both iterative and recursive estimates of the minimizers of these two functions (when $\Sigma$ is known). 
More specifically, we first propose an iterative method which consists in a fixed-point algorithm that presents similarities with the well-known Weiszfeld algorithm \citep{weiszfeld1937point,VZ00}. Then we propose new stochastic gradient algorithms and their averaged versions \citep{ruppert1988efficient,PolyakJud92}.
For both estimates, we establish convergence results such as rate of convergence, asymptotic normality and we give uniform bound of the mean quadratic error.
}

As the variance matrix $\Sigma$ is usually unknown, a second objective of this paper is to propose a methodology to build estimates based on the Mahalanobis norm. This consists in obtaining a first (iterative or recursive) robust estimate of $\beta^{*}$ (based on the Euclidean norm). This estimate {then enables} us to propose a robust estimate of  {$\Sigma$} based on the recent work of \cite{GBR2024}, before injecting it for the robust estimation of $\beta^{*}$ based on the Mahalanobis norm. {In addition}, we consider regularized versions of the previous estimates based on the Ridge regression estimator \citep{Ridge70},  {in order to deal with colinearities} in a high dimensional design matrix.


\textbf{Some related works. } 
{Robust inference in linear regression models has been studied in a wide range of works \citep{Rousseeuw2004,Ben2006,Agullo2008,HRVA2008,Chen2014,Li2022}.
Many approaches rely on rescaling techniques, such as S-estimators \citep{salibian2006fast,ollerer2016shooting}
having a simple high-breakdown regression estimator, which shares the flexibility and nice asymptotic properties of M-estimators
and MM-estimators \citep{yohai1987high,chen2002paper}
generalizing M-estimators with similar properties.
For a broad overview and numerous examples, we refer to the survey by \citet{filzmoser2021robust}.
However, these rescaling methods are often unsuitable in online settings, since they typically require an initial robust estimate to compute appropriate weights, among other limitations. In the univariate case (i.e., when the responses $Y$ are scalar), some online estimators do exist, based on M-estimation combined with stochastic gradient algorithms \citep{pesme2020online}.
Other approaches, such as trimming are not always appropriate for regression tasks. Because trimming is performed prior to regression, points with large residuals are removed \citep{rousseeuw2006computing}, and, to the best of our knowledge, such methods cannot be adapted to an online framework.
Although a few works have addressed the multivariate case \citep{Rousseeuw2004,AgY16}, no online methods are, to our knowledge, available in this setting. From a practical perspective, we have only identified one \texttt{R} package, \texttt{RRRR}, that specifically targets multivariate regression and with which we were able to compare our method.
}

\textbf{{Outline}.}
The paper is organized as follows.
In Section~\ref{sec::framework}, we {set  the framework for robust regression and define the criteria to be minimized}.
Section~\ref{sec::online} is dedicated to the online estimates of $\beta^*$ based on (weighted averaged) stochastic gradient algorithms denoted '\textit{Online OLS}' for the version based on the $\ell_2$-norm and '\textit{Online WLS}' for the version based on the Mahalanobis norm.
Section~\ref{sec::offline} is devoted to the offline estimates of $\beta^*$ based on fixed-point algorithms denoted '\textit{Offline OLS}' {(resp. '\textit{Offline WLS}') for the version based on the $\ell_2$-norm (resp. Mahalanobis norm)}.
For both online and offline estimates, we establish theoretical results {giving their} rate of convergence   and their asymptotic normality under weak assumptions.
In Section~\ref{sec::ridge}, we propose a regularized version of the previous estimates  using an adapted Ridge regression.
{Section~\ref{sec:comp} is dedicated to sum up the complexity and properties of robust regression methods.}
In Section~\ref{sec::mahasigma},  {we show how to use these estimates when $\Sigma$ is unknown}.
 {The simulation study presented in Section~\ref{sec::simu} supports} our theoretical results.
The proofs of our main theoretical results are postponed in Section~\ref{sec::proof}. All the
proposed methods are available in the \texttt{R} package \texttt{RobRegression} accessible on CRAN.

\section{Framework for robust regression}\label{sec::framework}



Let us first recall that the proposed robust alternative for estimating the parameter $\beta^{*}$ in the multivariate Gaussian linear model defined by~\eqref{def::model} is based on the fact that, as soon as the distribution of $\epsilon$ is symmetric, one can rewrite
\begin{equation}\label{def::G}
\beta^{*} =  \arg\min_{\beta \in \mathcal{M}_{q,p}(\mathbb{R})} G (\beta) \quad \quad \text{with} \quad G (\beta) := \mathbb{E}\left[ \left\| Y - \beta X\right\| {- \| Y \|} \right].
\end{equation}
Indeed, first note that $G: \mathcal{M}_{q,p}(\mathbb{R})\longrightarrow \mathbb{R}_{+} $ is a convex function.
Second, the function $G $ is differentiable and its gradient for all $\beta  \in \mathcal{M}_{q,p}(\mathbb{R})$ is defined by
\[
\nabla G (\beta ) = - \mathbb{E}\left[ \frac{Y - \beta X}{ \| Y - \beta  X\|} X^{T} \right].
\]
Since the distribution of $\epsilon$ is symmetric and since $\epsilon$ and $X$ are independent,
\[
\nabla G (\beta^{*} ) = - \mathbb{E}\left[  \frac{Y - \beta^{*}X}{ \| Y - \beta^{*}X \|} X^{T}  \right] = - \mathbb{E}\left[   \frac{\epsilon}{ \| \epsilon \|} X^{T} \right] = - \mathbb{E}\left[  \frac{\epsilon}{ \| \epsilon \|}  \right]  \mathbb{E}\left[ X^{T} \right] = 0,
\]
and $\beta^{*}$ is so a minimizer of the functional $G $.
{The following theorem ensures that $\beta^{*}$ admits a $50\%$ breakdown point with respect to the answer $Y$  (or equivalently with respect to the noise $\epsilon$).
\begin{theorem}\label{theo::rob}
    Let $F$ be the law of $X$. Suppose that $X$ admits a moment of order $1$ and that  for any sequence $\left( \beta_{n} \right) \in \mathcal{M}_{q,p}(\mathbb{R})$ such that $\|\beta_{n} \|_{F} \underset{n\to +\infty}{\longrightarrow} + \infty$, $\mathbb{E} \left[ \left\| \beta_{n}X \right\|\right]  \underset{n\to +\infty}{\longrightarrow} + \infty$. 
    Given $f \in [0,1)$ and considering  {the $f$ contamination of the noise} 
    \[
    \mathcal{L}_{f} = \left\lbrace (1-f) P_{0}+ f Q, \; \text{$Q$ a probability measure on $\mathbb{R}^{q}$} \right\rbrace ,
    \]
    where $P_{0}=\mathcal{N}(0, \Sigma) $ and
    let us denote by $\beta_{f,Q}^{*}$ the minimizer of the contaminated scenario, i.e.\,{}the minimizer of the functional
    \[
    G_{f,Q}: \beta \longmapsto (1-f)\iint  \left(\left\| \beta^{*}x  + \epsilon - \beta x \right\| - \| \beta^{*}x + \epsilon \|\right) \, dF(x)dP_{0}(\epsilon) +f\iint \left( \left\| \beta^{*}x  + \epsilon - \beta x \right\| - \| \beta^{*}x + \epsilon \|\right) \, dF(x)dQ(\epsilon) .
    \]
    Let us define the bias as
    \[
    B_{\beta^{*}} (f) = \sup \left\lbrace \left\| \beta^{*} -  \beta_{f,P}^{*} \right\| , P \in \mathcal{L}_{f} \right\rbrace .   
    \]
    Then 
\[
    f_{\beta^{*}}^{*} := \inf\left\lbrace f , B_{\beta^{*}} (f) = + \infty \right\rbrace \geq 0.5 .
    \]
\end{theorem}
}
{From an intuitive perspective, robustness with respect to the response 
variable $Y$ comes from the fact that the gradient is bounded in $Y$. 
Indeed, during a gradient step, even if a response is contaminated, 
the associated gradient remains bounded by $\| X \|_{F}$, so its impact is 
limited. In contrast, when considering the quadratic loss, the gradient 
takes the form $(Y - \beta^{*} X) X^{\top}$, which may blow up if $Y$ 
is contaminated. This explains why the {model}  is not robust in that case.
}

In addition, considering the use of the Mahalanobis norm leads to
\begin{equation}\label{def::GMaha}
\beta^{*} =  \arg\min_{\beta \in \mathcal{M}_{q,p}(\mathbb{R})} G_{\Sigma}(\beta) \quad \quad \text{with} \quad G_{\Sigma}(\beta):= \mathbb{E}\left[ \left\| Y - \beta  X\right\|_{\Sigma^{-1}} {-\| Y\|_{\Sigma^{-1}}} \right].
\end{equation}
Indeed, since the distribution of $\Sigma^{-1/2}\epsilon$ is symmetric and since $X$ and $\Sigma^{-1/2}\epsilon$ are independent, it comes that
\[
\nabla G_{\Sigma} (\beta^{*} ) = - \mathbb{E}\left[ \Sigma^{-1}  \frac{Y - \beta^{*}X}{ \| Y - \beta^{*}X \|_{\Sigma^{-1}}} X^{T} \right]  = - \mathbb{E}\left[  \Sigma^{-1}  \frac{\epsilon}{ \| \epsilon \|_{\Sigma^{-1}}} X^{T} \right] = - \Sigma^{-1/2} \mathbb{E}\left[  \frac{\Sigma^{-1/2}\epsilon}{ \| \Sigma^{-1/2} \epsilon \| }  \right] \mathbb{E}\left[ X^{T} \right] = 0,
\]
by observing that $\| \epsilon\|_{\Sigma^{-1}} = \left\| \Sigma^{-1/2} \epsilon \right\|$. {Observe that the statement of Theorem \ref{theo::rob} remains true in this case considering  the Mahalanobis distance instead of the Euclidean norm.}

In the following, we shall focus on the estimation of $\beta^{*}$ through the minimization of $G$ and $G_{\Sigma}$. Moreover, for both cases, we will introduce iterative algorithms (consisting in fixed-point algorithms) and online algorithms (consisting in averaged stochastic gradient algorithms). 

Furthermore, we will also consider regularized versions of the previous estimates by naturally considering a Ridge penalization adapted to the functions $G$ and $G_{\Sigma}$ (where the considered norm is not a squared norm). Then, we shall consider the following corresponding penalized functions:
\begin{equation}\label{def::Gridge}
G_{\lambda} (\beta) := \mathbb{E}\left[ \left\| Y - \beta X \right\| \right] + \lambda \left\| \beta \right\|_{F}
 \quad \quad \text{and} \quad \quad
G_{\Sigma , \lambda} (\beta) := \mathbb{E}\big[ \left\| Y - \beta X \right\|_{\Sigma^{-1}} \big] + \lambda \left\| \beta \right\|_{F},
\end{equation}
where $\lambda>0$ is a parameter to calibrate that controls the amount of shrinkage.
As soon as these functions are convex, continuous and diverge if $\| \beta \|_{F}$ goes to infinity, they admit a minimizer. In addition, supposing that this minimizer is different from $0$, it is unique and denoted $\beta_{\lambda}^{*}$.
Note that we have added a penalty term consistent with the seminal criterion to minimize, namely  $\lambda \|\beta\|_F$ (instead of $\lambda \|\beta\|^2_F$).

\section{Online estimation of $\beta^{*}$}\label{sec::online}

The objective is to apply usual  convergence results for stochastic algorithms \citep{robbins1951,PolyakJud92,pelletier1998almost,Pel00,godichon2016,godichon2016,gadat2017optimal,bach2013non}. In this aim, let us consider the following diffeomorphism:
\[
\begin{array}{rrcl}
\varphi : & 
\mathcal{M}_{q,p} (\mathbb{R}) & \longrightarrow & \mathbb{R}^{pq} \\
& \beta =  \left( \beta_{i,j} \right)_{1 \leq i \leq q, 1 \leq j \leq p} & \longmapsto\longrightarrow & \varphi (\beta ) = \left( \beta_{1,1} , \ldots , \beta_{1,p} , \ldots , \beta_{q,p} \right)^{T},
\end{array}
\]
which will be useful in the proof as well as for establishing the asymptotic normality of the estimates. 
\subsection{Online estimation of $\beta^{*}$ based on the minimization of $G$}
\label{sec::OLSonline}

Let us recall that the function $G$, defined by~\eqref{def::G}, to be minimized is defined for all $\beta \in \mathcal{M}_{q,p}(\mathbb{R})$ by $G(\beta) = \mathbb{E}\left[ g(X,Y,\beta) \right]$ where $g(X,Y,\beta) := \| Y - \beta X \|$. Then, as soon as $g(X,Y,.)$ is differentiable with $\nabla g(X,Y,\beta) = - \frac{Y-\beta X}{\| Y - \beta X\|}X^{T}$, the stochastic gradient algorithm is 
defined recursively for all $n\geq 0$ by
\begin{equation}\label{sgd}
\beta_{n+1}    = \beta_{n} + \gamma_{n+1}  \frac{Y_{n+1} - \beta_{n }X_{n+1}}{ \| Y_{n+1} - \beta_{n } X_{n+1}\|}  X_{n+1}^{T},
\end{equation}
where $\gamma_{n} = c_{\gamma}n^{-\gamma}$ with $c_{\gamma}> 0$ and $\gamma \in (1/2,1)$ \citep{robbins1951}. Observe that it is well known that stochastic gradient algorithms do not achieve the asymptotic efficiency. In order to accelerate the convergence, we should so consider its weighted averaged version (see \citep{ruppert1988efficient,PolyakJud92,mokkadem2011generalization,BGB2020} )  {named '\textit{Online OLS}' and} defined recursively, with $w\geq0$, for all $n\geq 0$ by 
\begin{equation}\label{asgd}
\overline{\beta}_{n+1 }   = \overline{\beta}_{n } + \frac{\log (n+2)^{w}}{\displaystyle \sum_{k=0}^{n+1} \log(k+1)^{w}} \left( \beta_{n+1 } - \overline{\beta}_{n } \right).
\end{equation}
In particular, note that
\[
\overline{\beta}_{n} = \frac{1}{\displaystyle \sum_{k=0}^{n} \log (k+1)^{w}} \sum_{k=0}^{n} \log (k+1)^{w} \beta_{k}
\]
and the classical averaging consists in taking $w = 0$. Taking $w > 0$ effectively decreases the weight of earlier iterates in comparison to the later ones and consequently it enables to avoid initialization issues.   {In practice we will set the averaging power to $w=2$.}
The following theorem gives the rate of convergence of the robust estimates defined by~\eqref{sgd} and \eqref{asgd} under weak assumptions.
\begin{theorem}\label{theo::without::sigma}
Let $\beta_{n}$ and $\overline{\beta}_n$ be defined respectively by~\eqref{sgd} and \eqref{asgd}.
Assume that $q \geq 3$, and that $X$ admits a moment of order $\max \lbrace 3,  2+2\eta \rbrace$ with $\eta > \frac{1}{\gamma} -1$.
Assume also that the matrix $\mathbb{E}\left[ X X^{T} \right]$ is positive definite.
Then,
\[
\left\| \beta_{n}  - \beta^{*}  \right\|_{F}^{2} = O \left( \frac{\ln n}{n^{\gamma}} \right) \quad a.s. \quad \quad \text{and} \quad \quad \left\| \overline{\beta}_{n}  - \beta^{*}  \right\|_{F}^{2} = O \left( \frac{\ln n}{n} \right) \quad a.s..
\]
Moreover,
\small
\[
\sqrt{n} \big( \varphi \left( \overline{\beta}_{n} \right) - \varphi \left( \beta^{*} \right) \big) \xrightarrow[n\to + \infty]{\mathcal{L}} \mathcal{N}_{pq} \left( 0 ,\mathbb{E}\left[\frac{1}{\left\| \epsilon \right\|}  \left(  I_{q}  - V_{\epsilon} \right)\right]^{-1}\mathbb{E}\left[ V_{\epsilon} \right]\mathbb{E}\left[\frac{1}{\left\| \epsilon \right\|}  \left(  I_{q}  - V_{\epsilon} \right)\right]^{-1} \otimes \left( \mathbb{E}\left[ XX^{T} \right] \right)^{-1} \right) ,
\]
\normalsize 
where 
$ V_{\epsilon}:= \frac{\epsilon}{\| \epsilon\|}\frac{ \epsilon^{T}}{\| \epsilon \|}$ and $\otimes$ denotes the Kronecker product.

If we suppose also that $X$ admits a moment of order $p'$ for any positive  $p'$, then there exist positive constants $C_{0}$ and $C_{1}$ such that for all $n \geq 1$
\[
\mathbb{E}\left[ \left\| \beta_{n} - \beta \right\|_{F}^{2} \right] \leq \frac{C_{0}}{n^{\gamma}} \quad \quad \text{and} \quad \quad \mathbb{E}\left[ \left\| \overline{\beta}_{n} - \beta \right\|_{F}^{2} \right] \leq \frac{C_{1}}{n}.
\]
\end{theorem}

\begin{remark}
In addition, observe that the $q\times q$ matrix 
\[\mathbb{E}\left[\frac{1}{\left\| \epsilon \right\|}  \left(  I_{q}  - V_{\epsilon} \right)\right]^{-1}\mathbb{E}\left[ V_{\epsilon} \right]\mathbb{E}\left[\frac{1}{\left\| \epsilon \right\|}  \left(  I_{q}  - V_{\epsilon} \right)\right]^{-1}\]
would be the asymptotic variance of recursive estimates of the median of the residual $\epsilon$ if we had access to the sequence $\left( \epsilon_{i} \right)$ (see \cite{HC}).
\end{remark}

\begin{remark}
It is important to note that if we had considered the quadratic risk defined by \eqref{def::OLS}, we could have focused on the empirical risk induced by the sample.
More precisely, we could have examined the $M$-estimator, i.e., the minimizer of this empirical risk. In this case, the $M$-estimator would have been explicit. However, as our simulations demonstrate, this approach lacks robustness.
Moreover, our online method achieves estimators with a computation time comparable to that required for computing the $M$-estimator (see Table \ref{tab::time}).
\end{remark}

\subsection{Online estimation of $\beta^{*}$ based on the minimization of $G_{\Sigma}$}
\label{sec::WLSonline}

Let us recall that the function $G_{\Sigma}$, defined by~\eqref{def::GMaha}, to be minimized is defined for all $\beta\in \mathcal{M}_{q,p}(\mathbb{R})$ by $G_\Sigma(\beta) = \mathbb{E}\left[ g_{\Sigma}(X,Y,\beta) \right]$ where $g_{\Sigma}(X,Y,\beta) := \| Y - \beta X \|_{\Sigma^{-1}}$. Then, as soon as $g_{\Sigma}(X,Y,.)$ is differentiable with $\nabla g_{\Sigma}(X,Y,\beta) = - \Sigma^{-1}\frac{Y-\beta X}{\| Y - \beta X\|}X^{T}$, the stochastic gradient algorithm and its weighted averaged version are  respectively and recursively defined, for all $n\geq 0$, by
\begin{align}
\label{sgd::Sigma}
\beta_{n+1,\Sigma}  & = \beta_{n,\Sigma} + \gamma_{n+1} \Sigma^{-1} \frac{Y_{n+1} - \beta_{n,\Sigma} X_{n+1}}{ \| Y_{n+1} - \beta_{n,\Sigma} X_{n+1} \|_{\Sigma^{-1}}}   X_{n+1}^{T},  \\
\label{asgd::Sigma}
\overline{\beta}_{n+1,\Sigma} & = \overline{\beta}_{n,\Sigma} + \frac{\log (n+1)^{w}}{\sum_{k=0}^{n} \log(k+1)^{w}} \left( \beta_{n+1,\Sigma} - \overline{\beta}_{n,\Sigma} \right),
\end{align}
where $\gamma_{n} = c_{\gamma}n^{-\gamma}$ with $c_{\gamma}> 0$ and $\gamma \in (1/2,1)$.
The following theorem gives the rates of convergence of these estimates. In the following, we refer 
to~\eqref{asgd::Sigma} as '\textit{Online WLS}'.
\begin{theorem}\label{theo::sigma}
Let $\beta_{n,\Sigma}$ and $\overline{\beta}_{n,\Sigma}$ be defined respectively by~\eqref{sgd::Sigma} and \eqref{asgd::Sigma}.
Assume that $q \geq 3$, and that $X$ admits a moment of order $\max \lbrace 3,  2+2\eta \rbrace$ with $\eta > \frac{1}{\gamma} -1$.
Assume also that the matrix $\mathbb{E}\left[ X X^{T} \right]$ is positive definite.
Then,
\[
\left\| \beta_{n,\Sigma}  - \beta^{*}  \right\|_F^{2} = O \left( \frac{\ln n}{n^{\gamma}} \right) \quad a.s. \quad \quad \text{and} \quad \quad \left\| \overline{\beta}_{n,\Sigma}  - \beta^{*}  \right\|_F^{2} = O \left( \frac{\ln n}{n} \right) \quad a.s.
\]
Moreover,
 {\[
\sqrt{n} \left( \varphi \left( \overline{\beta}_{n,\Sigma} \right) - \varphi \left( \beta^{*} \right) \right) \xrightarrow[n\to + \infty]{\mathcal{L}} \mathcal{N} \left( 0 , \frac{2q}{(q-1)^{2}} \frac{\Gamma\left(\frac{q}{2} \right)^2}{\Gamma\left(\frac{q-1}{2}\right)^2} \Sigma \otimes \left( \mathbb{E}\left[ XX^{T} \right] \right)^{-1} \right),
\]}
where $\Gamma$ is the Gamma function.
\normalsize 
If we suppose also that $X$ admits a moment of order $p'$ for any positive   $p'$, then there exist positive constants $C_{0,\Sigma}$ and $C_{1, \Sigma}$ such that for all $n \geq 1$
\[
\mathbb{E}\left[ \left\| \beta_{n,\Sigma} - \beta^* \right\|_F^{2} \right] \leq \frac{C_{0,\Sigma}}{n^{\gamma}} \quad \quad \text{and} \quad \quad \mathbb{E}\left[ \left\| \overline{\beta}_{n,\Sigma} - \beta^* \right\|_F^{2} \right] \leq \frac{C_{1,\Sigma}}{n} .
\]
\end{theorem}


\subsection{A comparison of these two approaches}

It is known that the naive OLS and naive WLS yields the same estimates: 
$
\arg\min_\beta \sum_{i=1}^n \|Y_i - \beta X_i\|^2 = \arg\min_\beta \sum_{i=1}^n \|Y_i - \beta X_i\|_{\Sigma^{-1}}^2
$ 
\citep[see][Chapter 5]{MKB79}, which means there is no statistical advantage to consider naive WLS rather than naive OLS. Still, this equality does not hold in presence of (Ridge) regularization.

As for the robust counterparts, the robust OLS and WLS estimates are different and yield different asymptotic variance, as shown by Theorems \ref{theo::without::sigma} and \ref{theo::sigma}. Unfortunately, because the asymptotic variance of the robust online OLS estimate given in Theorem \ref{theo::without::sigma} has no close form, we can not provide a systematic comparison of  the two asymptotic variances. Namely, we can not prove that  the difference between the OLS asymptotic variance matrix and the WLS asymptotic variance matrix is positive definite, which would imply a systematic interest of robust WLS as compared to robust OLS.

We can only report empirical observations. 
In parallel to the simulation study presented in Section \ref{sec::simu}, we simulated a large number of variance matrix $\Sigma$ and computed the asymptotic variances given by Theorems \ref{theo::without::sigma} and \ref{theo::sigma}. For all of them the difference between the OLS and the WLS variances turned out to be positive definite, \SR{}{indicating a better precision of the WLS estimator (see also Figure \ref{fig::compOW-n1000-noRidge})}.

Finally, it should be noted that in the case where the function $G_{\Sigma}$ is considered, the appearance of $\Sigma^{-1}$ in stochastic gradient algorithms can be, to some extent, linked to adaptive algorithms \citep{duchi2011adaptive,LP2020}. The latter are more effective when the problem is ill-conditioned. In our situation, this would correspond to the case where $\Sigma$ is strongly correlated.


\section{Offline estimation of $\beta^*$}\label{sec::offline}

\subsection{A fixed-point algorithm based on the minimization of $G$} \label{sec::OLSoffline}

In order to propose a new estimate  of $\beta^*$, we first show that finding the zero of the gradient of the function $G$ defined by~\eqref{def::G} leads to find a fixed point. Indeed, if for any $\beta \in \mathcal{M}_{q,p}(\mathbb{R})$, $\mathbb{E}\left[ \frac{XX^{T}}{\| Y - \beta X \|} \right] $ is positive definite, we observe that
\begin{align*}
\nabla G \left( \beta  \right) = 0 &  \Leftrightarrow \beta \mathbb{E}\left[ \frac{XX^{T}}{\| Y - \beta X \|} \right] = \mathbb{E}\left[ \frac{YX^{T}}{\| Y - \beta X \|}   \right] \Leftrightarrow \beta  = \mathbb{E}\left[ \frac{YX^{T}}{\| Y - \beta X \|} \right] \left( \mathbb{E}\left[ \frac{XX^{T}}{\| Y - \beta X \|} \right] \right)^{-1}.
\end{align*}
Then, given a sample $\mathcal{D}_n=\left( (X_{1} , Y_{1}) , \ldots ,(X_{n} , Y_{n}) \right) $ of $(X,Y)$ and assuming that for any $\beta \in \mathcal{M}_{q,p}(\mathbb{R})$, $\displaystyle\sum_{i=1}^{n} \frac{X_{i}X_{i}^{T}}{\left\| Y_{i} - \beta X_{i} \right\|}$ is positive definite,  the fixed-point algorithm is defined iteratively, for all $t \geq 0$, by
\begin{equation}\label{def::fix}
\beta_{n,t+1} = \sum_{i=1}^{n} \frac{Y_{i}X_{i}^{T}}{\left\| Y_{i} - \beta_{n,t} X_{i} \right\|} \left( \sum_{i=1}^{n} \frac{X_{i}X_{i}^{T}}{\left\| Y_{i} - \beta_{n,t} X_{i} \right\|} \right)^{-1} =: T_{n} \left( \beta_{n,t} \right) .
\end{equation} 
By defining $G_n$ as the empirical function generated by the sample $\mathcal{D}_n$ of $G$, \emph{i.e.}\,for all $\beta \in \mathcal{M}_{q,p}(\mathbb{R})$,
\begin{equation}
\label{def::Gn}
G_{n}(\beta) :=\frac{1}{n} \sum_{i=1}^{n} \left\| Y_{i} - \beta X_{i} \right\| ,
\end{equation} 
the fixed-point algorithm can be rewritten as a gradient descent as follows:
\begin{equation}
\label{def::fix::gradient}
\beta_{n,t+1} =  \beta_{n,t} + \frac{1}{n}\sum_{i=1}^{n} \frac{ \left( Y_{i} - \beta_{n,t}X_{i} \right)X_{i}^{T}}{\left\| Y_{i} - \beta_{n,t}X_{i} \right\|} \left( \frac{1}{n}\sum_{i=1}^{n} \frac{X_{i}X_{i}^{T}}{\left\| Y_{i} - \beta_{n,t} X_{i} \right\|} \right)^{-1}  = \beta_{n,t} - \nabla G_{n} \left( \beta_{n,t} \right) L_{n,t}^{-1},
\end{equation}
with a step sequence "on the right" $L_{n,t}^{-1}$ where $L_{n,t} :=  \displaystyle\frac{1}{n}\sum_{i=1}^{n} \frac{X_{i}X_{i}^{T}}{\left\| Y_{i} - \beta_{n,t} X_{i} \right\|} $. 
In the following, $\hat{\beta}_{n}$ denotes the minimizer of $G_{n}$.
\begin{theorem}\label{theo::fix}
Assume that $\sum_{i=1}^{n} X_{i}X_{i}^{T}$ is positive definite {almost surely}.
Then,
\[
\beta_{n,t} \xrightarrow[t\to + \infty]{}  \hat{\beta}_{n}.
\] 
\end{theorem}
The proof of Theorem~\ref{theo::fix} is given in Section~\ref{sec::proof}. 
The following result gives the asymptotic normality of the iterative estimate.

\begin{corollary}\label{cor::fix}
Assume that $X$ admits a moment of order $2$ and that $\mathbb{E}\left[ XX^{T} \right]$ is positive definite. Then,
\small
\[
\sqrt{n} \big( \varphi \left(   {\beta}_{n,t} \right) - \varphi \left( \beta^{*} \right) \big) \xrightarrow[n,t\to + \infty ]{\mathcal{L}} \mathcal{N}_{pq} \left( 0 , \mathbb{E}\left[\frac{1}{\left\| \epsilon \right\|}  \left(  I_{q}  - V_{\epsilon} \right)\right]^{-1}\mathbb{E}\left[ V_{\epsilon} \right]\mathbb{E}\left[\frac{1}{\left\| \epsilon \right\|}  \left(  I_{q}  - V_{\epsilon} \right)\right]^{-1} \otimes \left( \mathbb{E}\left[ XX^{T} \right] \right)^{-1} \right) .
\]
\normalsize 
\end{corollary}
Corollary~\ref{cor::fix} is a direct application of Theorem~\ref{theo::fix} coupled with Theorems 1 and 4 of \cite{niemiro1992asymptotics}. Observe that this corollary yields the same asymptotic normality as the recursive estimates, but with larger computational complexity. {This is not surprising, since it is well known that when the model is regular, both M-estimators and averaged stochastic gradient algorithms are asymptotically efficient \citep{PolyakJud92,Pel00}.}

\subsection{A fixed-point algorithm based on the minimization of $G_{\Sigma}$}
\label{sec::WLSoffline}

Similarly as in Section~\ref{sec::OLSoffline}, we first show that finding the zero of the gradient of $G_{\Sigma}$ defined by~\eqref{def::GMaha} leads to find a fixed point. Indeed, assuming that for any $\beta \in \mathcal{M}_{q,p}(\mathbb{R})$, $\mathbb{E}\left[ \frac{XX^{T}}{\| Y - \beta X \|_{\Sigma^{-1}}} \right]$ is positive definite, we observe that
\begin{align*}
\nabla G_{\Sigma}(\theta) = 0 \Leftrightarrow - \Sigma^{-1} \mathbb{E}\left[ \frac{Y - \beta X}{\left\| Y - \beta X \right\|_{\Sigma^{-1}}} X^{T} \right] = 0  \Leftrightarrow  \beta  = \mathbb{E}\left[ \frac{YX^{T}}{\| Y - \beta X \|_{\Sigma^{-1}}} \right] \left( \mathbb{E}\left[ \frac{XX^{T}}{\| Y - \beta X \|_{\Sigma^{-1}}} \right] \right)^{-1}.
\end{align*}
Hence, the fixed-point estimate is given iteratively, for all $t\geq 0$, by
\begin{equation}\label{def::fix::Sigma}
\beta_{n,t+1,\Sigma} = \sum_{i=1}^{n} \frac{Y_{i}X_{i}^{T}}{\left\| Y_{i} - \beta_{n,t,\Sigma} X_{i} \right\|_{\Sigma^{-1}}} \left( \sum_{i=1}^{n} \frac{X_{i}X_{i}^{T}}{\left\| Y_{i} - \beta_{n,t,\Sigma} X_{i} \right\|_{\Sigma^{-1}}} \right)^{-1} =: T_{n,\Sigma} \left( \beta_{n,t,\Sigma} \right),
\end{equation} 
which can also be written as a gradient algorithm with respect to the empirical function
\[
G_{n,\Sigma}(\beta ) := \frac{1}{n} \sum_{i=1}^{n} \left\| Y_{i} - \beta X_{i} \right\|_{\Sigma^{-1}}.
\]
In particular,
\begin{equation}\label{def::fix::gradient::Sigma}
\beta_{n,t+1, \Sigma} = \beta_{n,t,\Sigma} + \Sigma \frac{1}{n} \sum_{i=1}^{n} \frac{\Sigma^{-1} \left( Y_{i} - \beta_{n,t,\Sigma} X_{i} \right) }{\left\| Y_{i} - \beta_{n,t,\Sigma} X_{i} \right\|_{\Sigma^{-1}}} L_{n,t,\Sigma}^{-1} = \beta_{n,t,\Sigma} - \Sigma \nabla G_{n,\Sigma} \left( \beta_{n,t,\Sigma} \right) L_{n,t,\Sigma}^{-1},
\end{equation}
with  $\displaystyle L_{n,t,\Sigma} :=  \frac{1}{n}\sum_{i=1}^{n} \frac{X_{i}X_{i}^{T}}{\left\| Y_{i} - \beta_{n,t,\Sigma} X_{i} \right\|_{\Sigma^{-1}}} $.
In the following, $\hat{\beta}_{n,\Sigma}$ denotes the minimizer of $G_{n,\Sigma}$.
\begin{theorem}\label{theo::fix::Sigma}
Assume that $\sum_{i=1}^{n} X_{i}X_{i}^{T}$ is positive definite {almost surely}. 
Then,
\[
\beta_{n,t,\Sigma} \xrightarrow[t\to + \infty]{}  \hat{\beta}_{n,\Sigma}.
\] 
\end{theorem}
The proof of Theorem~\ref{theo::fix::Sigma} is given in Section~\ref{sec::proof}. 
The following result gives the asymptotic normality of the iterative estimate.\begin{corollary}\label{cor::fix::Sigma}
Assume that $X$ admits a moment of order $2$ and that $\mathbb{E}\left[ XX^{T} \right]$ is positive definite. Then,
\small
\[\sqrt{n} \big( \varphi \left( \overline{\beta}_{n,\Sigma} \right) - \varphi \left( \beta^{*} \right) \big) \xrightarrow[n,t\to + \infty]{\mathcal{L}} \mathcal{N} \left( 0 , \frac{2q}{(q-1)^{2}} \frac{\Gamma(\frac{q}{2})^2}{\Gamma(\frac{q-1}{2})^2} \Sigma \otimes \left( \mathbb{E}\left[ XX^{T} \right] \right)^{-1} \right).
\]
\normalsize
\end{corollary}
Corollary~\ref{cor::fix::Sigma} is a direct application of Theorem~\ref{theo::fix::Sigma} coupled with Theorems 1 and 4 of \cite{niemiro1992asymptotics}.

\section{The robust regularized regression} \label{sec::ridge}

We propose in this section the regularized version of the previous estimates. In particular, we only focus on the estimates built in Section~\ref{sec::WLSonline} and Section~\ref{sec::WLSoffline}. Note that similar results can be derived for the estimates proposed in Section~\ref{sec::OLSonline} and Section~\ref{sec::OLSoffline}. 

\subsection{Online estimation of $\beta^*$ based on the minimization of $G_{\Sigma,\lambda}$}

Recalling that $G_{\Sigma , \lambda}$ is defined by $G_{\Sigma, \lambda}(\beta) = \mathbb{E}\left[ g_{\Sigma , \lambda}(X,Y,\beta) \right]$ where $g_{\Sigma,\lambda}(X,Y,\beta) := \left\| Y - \beta X \right\|_{\Sigma^{-1}} + \lambda \left\| \beta \right\|_{F}$ (see \eqref{def::Gridge}). Then, as soon as $g_{\Sigma, \lambda}(X,Y,.)$ is differentiable with $\nabla g_{\Sigma, \lambda} ( X, Y , \beta) = - \Sigma^{-1} \frac{Y - \beta X }{\left\| Y - \beta X \right\|_{\Sigma^{-1}}} X^{T} + \lambda \frac{\beta}{\left\| \beta \right\|_{F}} \mathbf{1}_{\beta \neq 0}$ . Then, the stochastic gradient algorithm and its weighted averaged version are defined recursively for all $n \geq 0$ by
\begin{align}
\label{sgd::sigma::lambda}
\beta_{n+1,\Sigma, \lambda} & = \beta_{n,\Sigma, \lambda} + \gamma_{n+1} \left( \Sigma^{-1} \frac{Y_{n+1} - \beta_{n,\Sigma, \lambda} X_{n+1}}{\left\| Y_{n+1} - \beta_{n,\Sigma, \lambda}X_{n+1} \right\|_{\Sigma^{-1}}} X_{n+1}^{T} - \lambda \frac{\beta_{n,\Sigma, \lambda}}{\left\| \beta_{n,\Sigma, \lambda} \right\|_{F}} \mathbf{1}_{\beta_{n,\Sigma, \lambda} \neq 0} \right), \\
\label{asgd::sigma::lambda}\overline{\beta}_{n+1,\Sigma, \lambda} & = \overline{\beta}_{n,\Sigma, \lambda} + \frac{\log (n+1)^{w}}{\sum_{k=0}^{n} \log (k+1)^{w}} \left( \beta_{n+1,\Sigma, \lambda} - \overline{\beta}_{n,\Sigma, \lambda} \right) .
\end{align}
The following result gives the rate of convergence of the estimates.
\begin{theorem}\label{theo::sigma::lambda}
Suppose that $X$ admits a moment of order $\max \lbrace 3,  2+2\eta \rbrace$ with $\eta > \frac{1}{\gamma} -1$, and that $q \geq 3$. Suppose also that $\beta_{\lambda}^{*} \neq 0$, then
\[
\left\| \beta_{n,\Sigma,\lambda}  - \beta_{\lambda}^{*}  \right\|^{2} = O \left( \frac{\ln n}{n^{\gamma}} \right) \quad a.s. \quad \quad \text{and} \quad \quad \left\| \overline{\beta}_{n,\Sigma,\lambda}  - \beta_{\lambda}^{*}  \right\|^{2} = O \left( \frac{\ln n}{n} \right) \quad a.s.
\]
In addition, 
\[
\sqrt{n} \left( \varphi \left( \overline{\beta}_{n,\Sigma,\lambda} \right) - \varphi \left( \beta_{\lambda}^{*} \right) \right) \xrightarrow[n\to + \infty]{\mathcal{L}} \mathcal{N} \left( 0 , H_{\Sigma,\lambda}^{-1}V_{\Sigma,\lambda}H_{\Sigma,\lambda}^{-1} \right)
\]
where $H_{\Sigma,\lambda}:=\nabla^{2}\tilde{G}_{\Sigma, \lambda} \left( \varphi \left( \beta_{\lambda}^{*} \right) \right) $ and $
V_{\Sigma,\lambda}    = \mathbb{E}\left[ \nabla \tilde{g}_{\Sigma,\lambda} \left( X , Y , \beta_{\lambda}^{*} \right)\nabla \tilde{g}_{\Sigma,\lambda} \left( X , Y , \beta_{\lambda}^{*} \right)^{T} \right]  $.
\end{theorem}

\subsection{Fixed-point algorithm based on the minimization of $G_{\Sigma,\lambda}$}

Observe that in the penalized case, we did not succeed in writing the solution as an "easy" fix point. Any, following the scheme of proof of Theorems \ref{theo::fix} and \ref{theo::fix::Sigma}, this leads to an iterative conditioned gradient algorithm defined for all $t \geq 0$ by
\begin{align}
& \varphi\left( \beta_{n,t+1,\Sigma,\lambda} \right) =  \varphi\left(  \beta_{n,t,\Sigma,\lambda} \right) \nonumber\\
& + \left( \Sigma^{-1}\otimes \frac{1}{n}\sum_{i=1}^{n} \frac{X_{i}X_{i}^{T}}{\left\| Y_{i} - \beta_{n,t,\Sigma,\lambda} X_{i} \right\|_{\Sigma^{-1}}} + \frac{\lambda \mathbf{1}_{\beta_{n,t, \Sigma, \lambda} \neq 0} I_{pq}}{\left\|  \beta_{n,t,\Sigma,\lambda}\right\|_{F}}  \right)^{-1} \varphi\left( \frac{1}{n}\sum_{i=1}^{n} \frac{\left( Y_{i} -  \beta_{n,t,\Sigma,\lambda}X_{i}\right)X_{i}^{T}}{\left\| Y_{i} -  \beta_{n,t,\Sigma,\lambda} X_{i} \right\|_{\Sigma^{-1}}} \right) .\label{def::iterative::Sigma::lambda}
\end{align}
Indeed, here again, considering the functional 
\[
G_{n,\Sigma,\lambda} (\beta) := \frac{1}{n}\sum_{I=1}^{n} \left\| Y_{i} - \beta X_{i} \right\|_{\Sigma^{-1}} + \lambda \left\| \beta \right\|_{F} ,
\]
and denoting by $\nabla \tilde{G}_{n,\Sigma,\lambda}$ its gradient in the canonical basis of $\mathbb{R}^{pq}$, one can rewrite algorithm \eqref{def::iterative::Sigma::lambda} as
\begin{equation}
\label{def::gradient::Sigma::lambda} \varphi\left( \beta_{n,t+1,\Sigma,\lambda} \right) = \underbrace{ \varphi\left(  \beta_{n,t,\Sigma,\lambda} \right) + \left( \Sigma^{-1}\otimes \frac{1}{n}\sum_{i=1}^{n} \frac{X_{i}X_{i}^{T}}{\left\| Y_{i} - \beta_{n,t,\Sigma,\lambda} X_{i} \right\|_{\Sigma^{-1}}} + \frac{\lambda \mathbf{1}_{\beta_{n,t,\Sigma,\lambda} \neq 0} I_{pq}}{\left\|  \beta_{n,t,\Sigma,\lambda}\right\|_{F}}  \right)^{-1} \nabla \tilde{G}_{n,\Sigma,\lambda} \left( \varphi \left( \beta_{n,t,\Sigma,\lambda} \right) \right)}_{=: T_{n,\varphi,\Sigma,\lambda}\left( \beta_{n,t,\Sigma,\lambda} \right) }.
\end{equation}
In the sequel, let us denote by $\hat{\beta}_{n,\Sigma,\lambda}$ the minimizer of $G_{n,\Sigma,\lambda}$.
\begin{theorem}\label{theo::fix::Sigma::lambda}
One has
\[
\beta_{n,t,\Sigma,\lambda} \xrightarrow[n\to + \infty]{}\hat{\beta}_{n,\Sigma,\lambda}.
\]
In addition, if $X$ admits a moment of order $2$,
\[
\sqrt{n} \left( \varphi \left(  {\beta}_{n,t,\Sigma,\lambda} \right) - \varphi \left( \beta_{\lambda}^{*} \right) \right) \xrightarrow[n,t\to + \infty]{\mathcal{L}} \mathcal{N} \left( 0 , H_{\Sigma,\lambda}^{-1}V_{\Sigma,\lambda}H_{\Sigma,\lambda}^{-1} \right)
\]
\end{theorem}

\section{Complexity and comparison with existing methods}\label{sec:comp}

\subsection{Complexity}
 {
Table \ref{tab:complexity} summarizes the computational complexity of the proposed methods.
Here, $T$ denotes the total number of iterations performed by the iterative methods.
It is clear that the online methods are much less computationally demanding, in particular because they require only a single pass over the data (compared to $T$ passes for the iterative methods) and do not involve the inversion of a $p \times p$ matrix. However, even though the methods exhibit similar asymptotic behavior, simulations show that for small sample sizes, the online methods are less accurate than the offline ones. This can be summarized as follows: when the sample size is small, offline methods should be preferred, whereas for large sample sizes and/or when data arrive sequentially, online methods are more appropriate.
\begin{table}[H] 
\centering
\begin{tabular}{lcc}
\toprule
\textbf{Algorithm} & \textbf{{Equation}} & \textbf{Complexity} \\
\midrule
ASGD & \eqref{asgd} & $\mathcal{O}\!\left(npq\right)$ \\
ASGD with Mahalanobis Distance & \eqref{asgd::Sigma} & $\mathcal{O}\!\left(npq + nq^{2}\right)$ \\
ASGD with Mahalanobis Distance and Regularization & \eqref{asgd::sigma::lambda} & $\mathcal{O}\!\left(npq + nq^{2}\right)$ \\
Fixed-Point & \eqref{def::fix} & $\mathcal{O}\!\left(nTpq + Tp^{3}\right)$ \\
Fixed-Point with Mahalanobis Distance & \eqref{def::fix::Sigma} & $\mathcal{O}\!\left(nT(pq + q^{2} ) + Tp^{3}\right)$ \\
Fixed-Point with Mahalanobis Distance and Regularization & \eqref{def::iterative::Sigma::lambda} & $\mathcal{O}\!\left(nT(pq + q^{2}) + Tp^{3}\right)$ \\
\bottomrule
\end{tabular}
\caption{Computational complexity of the proposed algorithms.\label{tab:complexity}}
\end{table}
}

\subsection{Comparison with existing methods}
{
Table \ref{tab:robust_methods} summarizes the main approaches to robust linear regression.
It should be noted that, regarding theoretical guarantees, we only report what is explicitly stated in the cited papers.
The absence of a property in the table does not necessarily imply that it does not exist, but {rather}{} that it was not discussed in the referenced work.
}
\textcolor{red}{
\begin{table}[H]
\centering
\begin{tabular}{lcccc}
\hline
\textbf{Method} & \begin{tabular}[c]{@{}c@{}}Multivariate \\ Response\end{tabular} 
                &  {Online} 
                & \begin{tabular}[c]{@{}c@{}}Asymptotic \\ Normality\end{tabular} 
                & \begin{tabular}[c]{@{}c@{}}Mean Squared \\ Error\end{tabular} \\
\hline
S-estimator \cite{ollerer2016shooting}                               & No  & No  & ?   & ?        \\
MM-estimator \cite{chen2002paper}                                   & No  & No  & ?   & ?        \\
SGD \citep{pesme2020online}                                          & No  & Yes & ?   & $O(1/n)$ \\
MCD \citep{Rousseeuw2004}                                            & Yes & No  & ?   & ?        \\
RMRCDR \citep{santana2024robust}                                     & Yes & No  & ?   & ?        \\
Composite $\tau$-estimators \citep{AgY16}                            & Yes & No  & Yes & ?        \\
Fix point algorithms (Equations~\eqref{def::fix}, \eqref{def::fix::Sigma}, \eqref{def::iterative::Sigma::lambda}) & Yes & No  & Yes & ?        \\
ASGD algorithms (Equations~\eqref{asgd}, \eqref{asgd::Sigma}, \eqref{asgd::sigma::lambda})       & Yes & Yes & Yes & $O(1/n)$ \\
\hline
\end{tabular}
\caption{Comparison of robust regression methods.}
\label{tab:robust_methods}
\end{table}
}

\section{Estimation in practice when $\Sigma$ is unknown}\label{sec::mahasigma}

{In the previous sections, we have assumed that $\Sigma$ is known. We propose here to explain how we can use the proposed estimates in the case $\Sigma$ is unknown in practice.}
Given a sample $\left( X_{1},Y_{1} \right), \ldots , \left( X_{N}, Y_{N} \right)$, the  algorithm when $\Sigma$ is unknown consists to first estimate $\beta^{*}$ with the help of $\hat{\beta} =  \overline{\beta}_{N}$ or $\hat{\beta} = \beta_{T}$ (where $T$ is the number of iterations) given by \eqref{asgd} or \eqref{def::fix}. Then, one can estimate the Median Covariation Matrix of $\epsilon$ (see \cite{KrausPanaretos2012} and \cite{CG2015}). More precisely, the MCM of $\epsilon$ is defined by 
\[
V^{*} = \arg \min_{V} \mathbb{E}\left[ \left\|  \epsilon \epsilon^{T} - V \right\|_{F} \right] \quad \quad i.e \quad \quad V^{*} = \arg \min_{V} \mathbb{E}\left[ \left\|  \left( Y -   \beta^{*} X \right)  \left( Y - \beta^{*} X \right)^{T} - V \right\|_{F} \right] .
\]
As soon as $\epsilon$ is a Gaussian vector, it is well known \citep{KrausPanaretos2012} that the MCM and the variance $\Sigma$ have the same eigenvectors. Then, in order to construct a robust estimate of the variance, a first aim is to estimate the MCM of $\epsilon$. In this aim, let us consider the empirical function generated by the sample, and since $\beta^{*}$ is unknown, the aim is to minimize 
\[
G_{MCM,N}(V)=\frac{1}{N}\sum_{i=1}^{N}\left\|  \left( Y_{i} -  \hat{\beta}  X_{i}  \right)  \left( Y_{i} - \hat{\beta}  X_{i}\right)^{T} - V \right\|_{F},
\]
which leads to the Weiszfeld algorithm given by \cite{weiszfeld1937point} or an averaged stochastic gradient algorithm \citep{HC,CG2015}. These two algorithms are precisely described in Appendix. We will denote by $V_{N}$ the chosen estimate. In addition, one can calculate the eigenvalues $\delta_{N} = \left( \delta_{1,N} , \ldots \delta_{q,N} \right)$ (and the associated matrix of eigenvectors $P_{N}$) of $V_{N}$. Moreover, denoting by $\delta$ (resp. $\lambda$)  the eigenvalues of $V^{*}$ (resp. $\Sigma$), one has  \citep{KrausPanaretos2012}
\begin{equation}\label{relation}
\lambda = \delta   \Esp\left[ h(\delta, \lambda, U)\right] \oslash \Esp\left[U^{2} h(\delta, \lambda, U)\right], 
\end{equation}
where $\oslash$ is the division coordinate per coordinate, $U = \left( U_1 ,  \dots ,  U_q \right) \sim \Ncal(0, I_q)$ and 
$$
h(\delta, \lambda, U) = \left(\left(\sum_{j=1}^q (\delta_j - \lambda_j U_j^2\right)^2 + 2 \sum_{1 \leq j < k \leq q} \lambda_j \lambda_k U_j^2 U_k^2\right)^{-1/2}.
$$
This can also be written as 
\[
\delta \mathbb{E}\left[ h (\delta , \lambda , U ) \right] - \lambda  \centerdot \Esp\left[U^{2} h(\delta, \lambda, U)\right] = 0,
\]
where $\centerdot$ is the multiplication coordinate per coordinate.
Thus, one obtains a robust estimate $\delta_{N} = \left( \delta_{1,N} , \ldots  , \delta_{q,N} \right)$ of the eigenvalues of the variance $\Sigma$, using either a Fix-Point algorithm, a "gradient" algorithm or a Weighted Averaged Robbins Monro algorithm (see \cite{GBR2024}). Then, the robust estimate of the variance and its inverse are  given by 
\[
\Sigma_{N} = P_{N}  \delta_{N} P_{N}^{T} \quad \quad \text{and} \quad \quad \Sigma_{N}^{-1} = P_{N}  \delta_{N}^{-1}P_{N}^{T}
\]
where $P_{N}$ is the matrix composed of the normalized eigenvectors of the MCM and $\delta_{N}^{-1} = \left(  \delta_{1,N}^{-1} , \ldots  , \delta_{q,N}^{-1} \right)$.
Thus, one can replace $\Sigma^{-1}$ in \eqref{sgd::Sigma} or \eqref{def::fix::Sigma} by $ {\Sigma}_{N}^{-1}$  to obtain last estimates. This leads, to the weighted averaged stochastic gradient algorithm defined recursively for all $n \leq N-1$ by
 \begin{align}
\label{sgd::final}  \beta_{n+1,*} & =  \beta_{n,*} + \gamma_{n+1} \left( \Sigma_{N}^{-1} \frac{Y_{n+1} - \beta_{n,*} X_{n+1}}{\left\| Y_{n+1} - \beta_{n,*}X_{n+1} \right\|_{\Sigma_{N}^{-1}}} X_{n+1}^{T} - \lambda \frac{\beta_{n,*}}{\left\| \beta_{n,*} \right\|_{F}} \mathbf{1}_{\beta_{n,*} \neq 0} \right) \\
\label{asgd::final} \overline{\beta}_{n+1, *} & = \overline{\beta}_{n,*} + \frac{\log (n+1)^{w}}{\sum_{k=0}^{n} \log(k+1)^{w}} \left( \beta_{n+1,*} - \overline{\beta}_{n,*} \right) 
\end{align}
or the fix point algorithm defined iteratively for all $t \geq 0$ by
\[
\beta_{n,t+1,\Sigma_{N}} = T_{n,\Sigma_{N}} \left(  \beta_{n,t,\Sigma_{N}} \right) =    \sum_{i=1}^{N} \frac{Y_{i}X_{i}^{T}}{\left\| Y_{i} - \beta_{n,t,\Sigma_{N}} X_{i} \right\|_{\Sigma_{N}^{-1}}} \left( \sum_{i=1}^{N} \frac{X_{i}X_{i}^{T}}{\left\| Y_{i} - \beta_{n,t,\Sigma_{N}} X_{i} \right\|_{\Sigma_{N}^{-1}}} \right)^{-1} 
\]
 if $\lambda=0$ and the iterative algorithm
 \begin{align*}
& \varphi\left( \beta_{n,t+1,\Sigma_{N},\lambda} \right) = \varphi\left(  \beta_{n,t,\Sigma_{N},\lambda} \right) \\
& + \left( \Sigma_{N}^{-1}\otimes \frac{1}{N}\sum_{i=1}^{N} \frac{X_{i}X_{i}^{T}}{\left\| Y_{i} - \beta_{n,t,\Sigma_{N},\lambda} X_{i} \right\|_{\Sigma_{N}^{-1}}} + \frac{\lambda \mathbf{1}_{\beta_{n,t,\Sigma_{N},\lambda} \neq 0} I_{pq}}{\left\|  \beta_{n,t,\Sigma_{N},\lambda}\right\|_{F}}  \right)^{-1} \nabla \tilde{G}_{N,\Sigma_{N},\lambda} \left( \varphi \left( \beta_{n,t,\Sigma_{N},\lambda} \right) \right).
 \end{align*}
 if $\lambda > 0$.
 {Remark that the procedure presented here with Algorithm~\eqref{asgd::final} is not strictly online, since it requires two passes over the data. 
In practice, this is the version implemented in the R package \texttt{RobRegression}. 
However, in Section~\ref{sec:simu:online}, we will show how to process data in a truly online fashion while still relying on the Mahalanobis distance.}


\section{Simulation studies}\label{sec::simu}

\newcommand{\nn}{1000}
\newcommand{\sd}{1}
\renewcommand{\k}{1}

\newcommand{\parmReg}{k\k-sd\sd-seed1-df1}
\renewcommand{\parmReg}{k\k-sd\sd-seed1-df1}
\newcommand{\parmDisc}{K3-q20-scenario1-lin-seed1}


We now present a simulation study to assess the performances of the proposed estimation procedures. The main goal of the study is to evaluate the gain of using robust estimators as opposed to classical estimators, in presence of a varying fraction of outliers.
In this perspective, we considered two typical tasks that can be addressed with Model~\eqref{def::model}: multivariate regression and classification via a linear discriminant analysis. From a practical view-point, we have not been able to find packages addressing specifically multivariate regression, which prevented us to carry direct comparisons with other potentially robust methods.

\subsection{Multivariate regression} \label{sec::simulsReg}

\paragraph{Simulation design.} 
We fixed the number of variables to $p = 5$ for the explanatory variables in $X$ and to $q = 20$ for the response variables in $Y$. The $q \times p$ regression coefficients $\beta$ were sampled  as iid standard normal and fixed once for all. 
We also fixed the covariance matrix $\Sigma$ (see Appendix \ref{app::parms}). We considered data sets with $n = 100$, $1000$ and $10000$ observations. We also considered increasing proportions $\prop$ of outliers, namely: $\prop =  0$, $2$, $3$, $5$, $9$, $16$, $28$ and $50\%$. 

For each dimension $n$, we sampled the $n \times p$ entries of the matrix $X$ as iid normal and fixed them once for all. Then, for each fraction $\prop$, we sampled $B = 100$ $n \times q$ response matrices $Y$ according to Model~\eqref{def::model} in which a fraction $1-\prop$ of the residual terms $\epsilon_i$ were sampled from a multivariate normal $\Ncal_q(0, \Sigma)$, whereas the remaining fraction $\prop$ were sampled as iid Student with variance 1 and degree of freedom 1. 

We carried the inference on each simulated response matrix $Y$, using the following procedures: 
\begin{description}
  \item[\sl Naive OLS] refers to the classical ordinary least-square estimate of $\beta$: $\widehat{\beta} = \arg\min_\beta \sum_{i=1}^n \|Y_i - \beta X_i\|^2$;
  \item[\sl Offline OLS] refers to its robust counterpart, resulting from the fixpoint (offline) algorithm described in Section~\ref{sec::OLSoffline};
  \item[\sl Online OLS] refers to robust estimation resulting from the online algorithm described in Section~\ref{sec::OLSonline}.
  \item[\sl Naive, Offline and Online WLS] are defined similarly: '\textit{Naive WLS}' refers to the minimizer of $\sum_{i=1}^n \|Y_i - \beta X_i\|^2_{\widehat{\Sigma}^{-1}}$ (where $\widehat{\Sigma}$ is the empirical residual covariance matrix), '\textit{Offline WLS}' refers to the fixpoint algorithm from Section~\ref{sec::WLSoffline} and '\textit{Online WLS}' to the algorithm from Section~\ref{sec::WLSonline}.
  \item[\sl Ridge :] 
  For each procedure, we also considered a Ridge regularization of the parameter $\beta$, as described in Section~\ref{sec::ridge}. Namely, we added a penalty term consistent with each respective criterion, that is $\lambda \|\beta\|^2$ for the naive procedures and $\lambda \|\beta\|$ for the robust ones.
  In each case, the penalty coefficient $\lambda$ was determined via 10-fold cross-validation.
\end{description}

\SR{}{\paragraph{\sl Available alternatives.}
We found only two R packages available dedicated to robust multivariate linear regression. 
\begin{itemize}
  \item \texttt{robustvarComp} \citep{AgY16,AgY22} is dedicated to variance component models, for which it provides robust estimates of the variance components associated with a prescribed dependency structure, but not robust estimates of the regression coefficients. The objective of \texttt{robustvarComp} is therefore different from ours, which is why we did not include it in the comparative study. 
  \item \texttt{RRRR} \citep{YaZ20,YaZ23} addresses the same problem as us, in a slightly broader context since it takes into account dependent observations and the reduced-rank covariance matrix. We have included it in our study, imposing independent observations and a full-rank covariance matrix in order to make the comparison with our approach consistent. The package also offers an online version, which allows parameter estimates to be updated using a new additional dataset. This does not correspond to the classic definition of an online estimator, which may update estimates for each new piece of data, so we did not include this version in our comparison. No equivalent of the Ridge or Mahalanobis versions described in the preceding sections are available in \texttt{RRRR}, so we compare its results with our 'OLS' version.
\end{itemize}
}

\paragraph{Results.}
We now present the results obtained with the different procedures in terms of 
estimation accuracy, 
prediction accuracy, 
outlier detection 
and computational time. 
We use here the generic notation $\widehat{\beta}$ and $\widehat{\Sigma}$ for the estimates obtained with a given procedure.

\paragraph{\sl Estimation accuracy.}
Figure \ref{fig::mseBeta-n\nn} (top) displays the distribution of the mean squared error for the estimation of $\beta$ ($MSE(\beta) = \|\widehat{\beta} - \beta\|^2 / (pq)$) for a sample size of $n = \nn$. 
\SR{}{The numerical results are given in Table~\ref{tab::mseBeta-n\nn} in Appendix \ref{app::tables}} 
As expected, the $MSE(\beta)$ of the naive estimates dramatically increases as the proportion of outliers increases, whereas these of both the offline and online robust estimates remain stable. 
The Ridge regularization (same figure, bottom) tends to limit the effect of outliers on the naive estimator. 
The same pattern is observed for the WLS estimators (see Appendix~\ref{app::figures}, Figure~\ref{fig::mseBeta-n\nn-WLS}). 
\SR{}{Whatever the proportion of outliers, we see that the performances of RRRR procedure are similar to these of ou offline procedure.}

\begin{figure}[ht]
  \begin{center}
    \begin{tabular}{m{0.1\textwidth}m{0.7\textwidth}}
      OLS & \includegraphics[width=0.7\textwidth, trim=0 60 0 55, clip=]{MSEbeta-n\nn-fake-p5-q20-\parmReg-OLS} \\
      OLS+Ridge & \includegraphics[width=0.7\textwidth, trim=0 60 0 55, clip=]{MSEbeta-n\nn-fake-p5-q20-\parmReg-OLS+Ridge} 
    \end{tabular}
  \end{center}
  \caption{
  Mean squared error for the estimation of the regression coefficients $\beta$ ($MSE(\beta)$) for OLS estimates without or with Ridge regularization for a sample size of $n = \nn$ \SR{}{(with Student outliers)}. 
  The $y$-axis is log-scaled. 
  Legend:  \textcolor{black}{$\square$} = Naive, 
  \textcolor{red}{$\medcirc$} = Offline,
  \textcolor{green}{$\triangle$} = Online, 
  \textcolor{blue}{$+$} = RRRR. 
  Number in each box = fraction $\prop$ of outliers (in \%).  \label{fig::mseBeta-n\nn}} 
\end{figure}

In Appendix~\ref{app::figures}, Figure~\ref{fig::mseSigma-n\nn} displays the mean squared error for the estimation of $\Sigma$ ($MSE(\Sigma) = \|\widehat{\Sigma} - \Sigma\|^2 / q^2)$) for the same sample size. Again, the accuracy of naive estimates dramatically deteriorates as the proportion of outliers increases, whereas these of both the offline and online robust estimates remain stable. The Ridge regularization (which acts on $\beta$) has almost no effect on the quality of the respective estimates of $\Sigma$. 
\SR{}{The RRRR procedure appears to be less accurate than the proposed ones, except for very high outlier fractions.}

\paragraph{\sl Effect of the sample size.}
We also studied the effect of the sample size $n$ on the accuracy of the estimates. We only considered the regression coefficients $\beta$. Figure \ref{fig::mseBeta-prop28-OLS} 
\SR{}{(see also Table \ref{tab::mseBeta-prop28-OLS}, Appendix \ref{app::tables})} 
shows the gain resulting from the use of robust estimates as opposed to the naive one, for the OLS contrast. As expected, this gain increases when  the fraction of outliers is large ($\prop = 28\%$), but also increases when the sample size increases. 
Again, the same pattern is observed for the WLS contrast (see Appendix \ref{app::figures}, Figure \ref{fig::mseBeta-prop28-WLS}).
\SR{}{As in Figure \ref{fig::mseBeta-n\nn}, the RRRR procedure displays similar performances as our offline OLS procedure.}

\begin{figure}[ht]
  \begin{center}
    \begin{tabular}{m{0.1\textwidth}m{0.35\textwidth}m{0.35\textwidth}}
      & \multicolumn{1}{c}{$\prop = 5\%$} & \multicolumn{1}{c}{$\prop = 28\%$} \\
      OLS & 
      \includegraphics[width=0.35\textwidth, trim=0 60 0 55, clip=]{MSEbeta-prop5-fake-p5-q20-\parmReg-OLS} & 
      \includegraphics[width=0.35\textwidth, trim=0 60 0 55, clip=]{MSEbeta-prop28-fake-p5-q20-\parmReg-OLS} \\
      OLS+Ridge & 
      \includegraphics[width=0.35\textwidth, trim=0 60 0 55, clip=]{MSEbeta-prop5-fake-p5-q20-\parmReg-OLS+Ridge} & 
      \includegraphics[width=0.35\textwidth, trim=0 60 0 55, clip=]{MSEbeta-prop28-fake-p5-q20-\parmReg-OLS+Ridge} 
    \end{tabular}
  \end{center}
  \caption{
  Mean squared error for the estimation of the regression coefficients $\beta$ ($MSE(\beta)$) for OLS estimates with or without Ridge regularization for a fraction of outliers of $\prop = 5\%$ (left) and $\prop = 28\%$ (right) \SR{}{(with Student outliers)}. 
  The $y$-axis is log-scaled. 
  Legend:  \textcolor{black}{$\square$} = Naive, \textcolor{red}{$\medcirc$} = Offline, \textcolor{green}{$\triangle$} = Online, 
  \textcolor{blue}{$+$} = RRRR. Number in each box = sample size $n$.  
  \label{fig::mseBeta-prop28-OLS}} 
\end{figure}

\paragraph{\sl Prediction accuracy.}
We then considered the precision of the prediction of the response variables in $Y$, measured with $MSE(Y) = \|\widehat{\beta}X - Y\|^2 / (nq)$, and focusing on non-outlier observations. Overall, the prediction accuracy inherits this of the regression coefficients $\beta$, so the behavior of $MSE(Y)$ is similar to this of $MSE(\beta)$. Figure \ref{fig::predIn-n\nn} (devoted to Appendix \ref{app::figures}) illustrates this point.

\paragraph{\sl Outlier detection.}
Based on the prediction $\widehat{Y}_{ij}$ obtained for each observation $1 \leq i \leq n$ and response variable $1 \leq j \leq q$, we defined the estimated residual $\widehat{\epsilon}_{ij} = Y_{ij} - \widehat{Y}_{ij}$, based on which one may aim at deciding if observation $i$ is an outlier or not. 
More specifically, based on the estimate $\widehat{\Sigma}$, and denoting by $\widehat{\epsilon}_i =[\widehat{\epsilon}_{ij}]_{1 \leq j \leq p}$ the vector of residuals for observation $i$, 
we used its $\widehat{\Sigma}^{-1}$ $\ell_2$-norm $\|\widehat{\epsilon}_i\|^2_{\widehat{\Sigma}^{-1}}$ as a score, which is expected to be small for non-outliers, and large for outliers. 

Figure \ref{fig::auc-n\nn-OLS} gives the distribution of the area under the ROC curve (AUC) for the classification of outliers/non-outliers based on the proposed score for OLS estimates 
\SR{}{(see also Table \ref{tab::auc-n\nn-OLS}, Appendix \ref{app::tables}).} 
Again, the performances of the naive procedure decreases as the fraction of outlier increases, although staying quite good up to an intermediate fraction of outliers ($f \simeq 10-15\%$), whereas both robust procedures are not affected by their presence. Interestingly, the Ridge regularization does not improve the performances. 
\SR{}{To this respect, the RRRR procedure has a similar behavior as both our offline et online OLS procedures.}
Here also, OLS and WLS procedure behave similarly (see Appendix \ref{app::figures}, Figure \ref{fig::auc-n\nn-WLS}).

\begin{figure}[ht]
  \begin{center}
    \begin{tabular}{m{0.1\textwidth}m{0.7\textwidth}}
      OLS & \includegraphics[width=0.7\textwidth, trim=0 60 0 55, clip=]{AUC-n\nn-fake-p5-q20-\parmReg-OLS} \\
      OLS+Ridge & \includegraphics[width=0.7\textwidth, trim=0 60 0 55, clip=]{AUC-n\nn-fake-p5-q20-\parmReg-OLS+Ridge} 
    \end{tabular}
  \end{center}
  \caption{
  AUC for the detection of outliers for the OLS estimates with or without Ridge regularization for a sample size of $n = \nn$ \SR{}{(with Student outliers)}. 
  Same legend as Figure \ref{fig::mseBeta-n\nn}. The box $f = 0$ is empty as no outlier exists.  
  \label{fig::auc-n\nn-OLS}}
\end{figure}

\paragraph{\sl OLS versus WLS.} 
We then compared the performances of the robust OLS and WLS estimations, in terms of accuracy for the estimation of the regression coefficients $\beta$. 

We computed the asymptotic variance given in Theorems \ref{theo::without::sigma} and \ref{theo::sigma} for the simulated covariance matrix $X$ and the variance matrix $\Sigma$ described in Appendix \ref{app::figures} and deduced the ratio between the two asymptotic  variances $\Var(\widehat{\beta}_{jk, \Sigma}) / \Var(\widehat{\beta}_{jk})$ of each of the $p \times q = 100$ regression coefficients $\beta_{jk}$. 
These ratios turn ou to be very homogeneous, with a mean of $95.7\%$ and a standard deviation of $0.4\%$, indicating that using robust WLS rather than robust OLS induces a reduction of about $5\%$ of the variance. 
As for the difference between the two un-normalized covariance matrices $n(\Var(\widehat{\beta}) - \Var(\widehat{\beta}_{\Sigma}))$, all its eigenvalues are verified to be non negative, ranging from about 300 to about $10^4$. 

Figure \ref{fig::compOW-n\nn-noRidge} compares the $MSE(\beta)$ obtained with the OLS and WLS contrasts with the offline and online algorithms, without Ridge regularization (\SR{}{see also Table \ref{tab::compOW-n\nn-noRidge}, Appendix \ref{app::tables}}). 
We observe that resorting to WLS does not really improve the accuracy of the offline estimates, and may even deteriorate it when the fraction of outliers is high. On the contrary, the online algorithm takes advantage of the WLS contrast to provide more precise estimates of $\beta$. The gain gets even higher as outliers get more numerous (up to a fraction of $\prop = 50\%$). A similar pattern is observed when using the Ridge regularization, as shown in Appendix \ref{app::figures}, Figure \ref{fig::compOW-n\nn-Ridge}.

\begin{figure}[ht]
  \begin{center}
    \begin{tabular}{m{0.12\textwidth}m{0.7\textwidth}}
      Offline & \includegraphics[width=0.7\textwidth, trim=0 60 0 55, clip=]{MSEbetacompOW-NoRidge-n\nn-fake-p5-q20-\parmReg-Off} \\
      Online & \includegraphics[width=0.7\textwidth, trim=0 60 0 55, clip=]{MSEbetacompOW-NoRidge-n\nn-fake-p5-q20-\parmReg-On} 
    \end{tabular}
  \end{center}
  \caption{
  Comparison of the mean squared error for the estimation of the regression coefficients $\beta$ ($MSE(\beta)$) of OLS and WLS estimates for Offline and Online procedures without Ridge regularization for a sample size of $n = \nn$ \SR{}{(with Student outliers)}.
  Legend:  \textcolor{red}{$\medcirc$} = OLS, \textcolor{green}{$\triangle$} = WLS. Number in each box = fraction $\prop$ of outliers (in \%).     
  \label{fig::compOW-n\nn-noRidge}}
\end{figure}

We also compared the empirical variance of the robust OLS and WLS estimates of the regression coefficients with their asymptotic counterparts, given by Theorems \ref{theo::without::sigma} and \ref{theo::sigma}. Figure \ref{fig::compVarOW-asymp} presents this comparison for various sample sizes $n$ and both the offline and online procedures. We observe that the offline procedure almost the asymptotic variance already for $n = 100$, whereas the offline procedure gets closer to it, but does still not reach it for $n = 10000$.  {Thus, this figure clearly illustrates the asymptotic normalities stated in Theorems~\ref{theo::without::sigma}, \ref{theo::sigma}, \ref{theo::fix}, and \ref{theo::fix::Sigma}.} Observe that Figure \ref{fig::compVarOW-asymp} compares variances whereas Figure \ref{fig::compOW-n\nn-noRidge} presents mean square errors.

\begin{figure}[ht]
  \begin{center}
    \begin{tabular}{m{0.1\textwidth}m{0.2\textwidth}m{0.2\textwidth}m{0.2\textwidth}}
      & \multicolumn{1}{c}{$n=100$} & \multicolumn{1}{c}{$n=1000$} & \multicolumn{1}{c}{$n=10000$} \\
      OLS
      & \includegraphics[width=0.2\textwidth, trim=0 10 0 20, clip=]{CompVarRols-fake-p5-q20-\parmReg-n100-prop16} 
      & \includegraphics[width=0.2\textwidth, trim=0 10 0 20, clip=]{CompVarRols-fake-p5-q20-\parmReg-n1000-prop16} 
      & \includegraphics[width=0.2\textwidth, trim=0 10 0 20, clip=]{CompVarRols-fake-p5-q20-\parmReg-n10000-prop16} \\
      WLS
      & \includegraphics[width=0.2\textwidth, trim=0 10 0 20, clip=]{CompVarRwls-fake-p5-q20-\parmReg-n100-prop16} 
      & \includegraphics[width=0.2\textwidth, trim=0 10 0 20, clip=]{CompVarRwls-fake-p5-q20-\parmReg-n1000-prop16} 
      & \includegraphics[width=0.2\textwidth, trim=0 10 0 20, clip=]{CompVarRwls-fake-p5-q20-\parmReg-n10000-prop16} \\
    \end{tabular}
  \end{center}
  \caption{
   {
  Comparison of the empirical variances of the $p \times q = 100$ regression coefficients estimates $\widehat{\beta}_{jk}$ with the asymptotic variance, with a fraction of $16\%$ of outliers \SR{}{(with Student outliers)}. 
  Top: robust OLS procedure, bottom: Robust WLS procedure. 
  Left: $n = 100$, center: $n=1000$, right: $n=10000$.
  Legend:  \textcolor{red}{$\medcirc$} = offline, \textcolor{green}{$\triangle$} = online.
  }\label{fig::compVarOW-asymp}}
\end{figure}

\paragraph{\sl Computation time.} 
We recorded the mean CPU time for each procedure under each simulation configuration. The computational times does not significantly vary with the proportion of outliers (not shown) and is simply multiplied by a constant when using the Ridge regularization, due the fixed number of cross-validation steps. 

Table \ref{tab::time} reports the results for a fraction of $\prop=28\%$ outliers, without regularization. We observe the stability of the naive procedure, as opposed to the stochastic robust ones. More importantly,  {we observe the substantial gain of resorting to the online procedure, as opposed to offline (especially for the OLS contrast) when $n$ increases}.
\SR{}{Although the performances of the RRR procedure are comparable to the offline OLS we propose, both in terms of parameter estimation and outlier detection, it is clearly much slower (176 seconds as opposed to about 1 for the Offline and 0.1 for the Online) and its computational time may become prohibitive for very large datasets.}

\begin{table}[ht]
  \begin{center}
    \begin{tabular}{l|cccc|ccc}
      & \multicolumn{4}{c|}{OLS} & \multicolumn{3}{c}{WLS} \\
      & Naive & Offline & Online & RRRR & Naive & Offline & Online \\
      \hline
      $n=100$   & 0.021 & 0.012 & 0.002 & 0.481 & 0.024 & 0.094 & 0.047 \\
      $n=1000$  & 0.021 & 0.090 & 0.010 & 4.266 & 0.029 & 0.830 & 0.482 \\
      $n=10000$ & 0.029 & 0.952 & 0.098 & 176.5 & 0.094 & 7.571 & 4.775
    \end{tabular}
  \end{center}
  \caption{
  Mean CPU time in seconds for the OLS and WLS estimation procedures without Ridge regularization \SR{}{for a varying number of observations $n$ (with Student outliers)}. \label{tab::time}}
\end{table}

\FloatBarrier

\subsection{\SR{}{Extended simulation settings}}

\SR{}{To better explore  the performances of the proposed algorithm, we conducted two additional simulation studies: one in large dimension and one with outliers concentrated in one single value.}

\newcommand{\pLarge}{10}
\newcommand{\qLarge}{100}
\newcommand{\nLarge}{10000}
\newcommand{\dimLarge}{q\qLarge}
\newcommand{\parmLarge}{fake-p\pLarge-n\nLarge-k1-sd1-seed1}

\SR{}{
\subsubsection*{Large dimension} 
We first attempted to evaluate the performance of the proposed procedures in larger dimensions than in the preceding study. 
To this end, we conducted the same simulation study with $p = 10$ covariates, $q = 20, 50$ and $100$ response variables, and $n = 1000$ and $10000$ observations. 
The equivalent of Figures \ref{fig::mseBeta-n\nn}, \ref{fig::mseBeta-prop28-OLS}, \ref{fig::auc-n\nn-OLS} and \ref{fig::compOW-n\nn-noRidge} are given in Appendix \ref{app::large}, in Figures \ref{fig::mseBeta-n\nLarge-large}, \ref{fig::mseBeta-prop28-OLS-large}, \ref{fig::auc-n\nLarge-OLS-large} and \ref{fig::compOW-n\nLarge-noRidge-large}, respectively. 
The behavior of the different procedures is qualitatively the same as in the previous study.
We only report the computation times (see Table \ref{tab::timeLarge}), which show how well the algorithms we propose adapt to large dimensions.
}

\begin{table}[ht]
  \begin{center}
    $n = 1000$ \\ 
    \begin{tabular}{l|cccc|ccc}
      & \multicolumn{4}{c|}{OLS} & \multicolumn{3}{c}{WLS} \\
      & Naive & Offline & Online & RRRR & Naive & Offline & Online \\
      \hline
      $q=20$  & 0.032 & 0.094 & 0.014 & 6.734 & 0.032 & 1.059 & 0.649 \\
      $q=50$  & 0.065 & 0.093 & 0.020 & 27.65 & 0.077 & 2.619 & 2.346 \\
      $q=100$ & 0.120 & 0.151 & 0.029 & 87.96 & 0.200 & 10.04 & 9.779 \\
      \multicolumn{8}{c}{} \\
      \multicolumn{8}{c}{$n = 10000$} \\
      & \multicolumn{4}{c|}{OLS} & \multicolumn{3}{c}{WLS} \\
      & Naive & Offline & Online & RRRR & Naive & Offline & Online \\
      \hline
      $q=20$  & 0.044 & 1.057 & 0.130 & 206.6 & 0.140 & 7.984 & 4.866 \\
      $q=50$  & 0.111 & 0.990 & 0.184 & 494.6 & 0.335 & 19.57 & 17.48 \\
      $q=100$ & 0.206 & 1.191 & 0.275 & 1194. & 0.759 & 84.14 & 81.56
    \end{tabular}
  \end{center}
  \caption{
  \SR{}{
    Mean CPU time in seconds for the OLS and WLS estimation procedures without Ridge regularization for a varying number of response variables $q$. Top: $n = 1000$ observations, Bottom: $n = 10000$ observations. \label{tab::timeLarge}}
    }
\end{table}

\SR{}{We observe that the computational burden of the online procedure remains low, even for a large number of observations ($n = 10000$) and a large number of responses ($q = 100$). Fot the OLS approach, the computational time of the online is very low (less than 0.3 s), and remains reasonable for the offline procedure (1.2s), whereas this of the RRRR procedure becomes prohibitive (about 1200s). We remind that the RRR package is not focused on computational efficiency and provides a broad modelling framework.}

\FloatBarrier

\newcommand{\pDirac}{5}
\newcommand{\qDirac}{20}
\newcommand{\nDirac}{1000}
\newcommand{\muDirac}{3}
\newcommand{\parmDirac}{k1-sd1-mu\muDirac-seed1}

\SR{}{
\subsubsection*{Concentrated outliers} 
Upon suggestion of one of the reviewers, we considered concentrated outliers, that is outliers that all take the same value.
To mimic this situation, we used the same simulation design as in Section \ref{sec::simulsReg}, replacing the Student outlier distribution with a Dirac mass at the point with all coordinates equal to $\mu=\muDirac$.
}

\SR{}{The counterparts of Figures~\ref{fig::mseBeta-n\nDirac}, \ref{fig::mseBeta-prop28-OLS} and \ref{fig::predIn-n\nn} for this alternative simulation design are displayed in Appendix \ref{app::dirac}. 
We observe that in terms of both ($i$) precision of the regression coefficients estimates (Figure \ref{fig::mseBeta-n\nDirac-dirac}), ($ii$) effect of the sample size (Figure \ref{fig::mseBeta-prop28-OLS-dirac}) and ($iii$) prediction accuracy (Figure \ref{fig::predIn-n\nDirac-dirac}), Dirac outliers have  much less impact than Student outliers. For example, the mean squared error $MSE(\beta)$ remains below .5 in all configurations for all methods (including RRRR: see Figure \ref{fig::mseBeta-n\nDirac-dirac}).}

\SR{}{
The main difference appears for outlier detection, which displays a dramatic transition above a certain fraction of outliers. Figure \ref{fig::auc-n\nDirac-OLS-dirac} shows that the AUC is close to 1 for a small fraction $f$ (whereas it remains below 75\% with Student outliers), but drops as the fraction $f$ increases. The naive and RRRR algorithms are already significantly affected with 5\% outliers, and the performance of both the offline and the online algorithms deteriorates with a fraction of 16\%. The radical change to a given fraction for each method stems from the nature of the contamination: since all outliers are located in the same place, they are all reported together as outliers or normal values.
}

\begin{figure}[ht]
  \begin{center}
    \begin{tabular}{m{0.1\textwidth}m{0.7\textwidth}}
      OLS & \includegraphics[width=0.7\textwidth, trim=0 60 0 55, clip=]{AUC-n\nDirac-dirac-p5-q20-\parmDirac-OLS} \\
      OLS+Ridge & \includegraphics[width=0.7\textwidth, trim=0 60 0 55, clip=]{AUC-n\nDirac-dirac-p5-q20-\parmDirac-OLS+Ridge} 
    \end{tabular}
  \end{center}
  \caption{
  \SR{}{
  AUC for the detection of outliers for the OLS estimates with or without Ridge regularization for a sample size of $n = \nDirac$ (with Dirac outliers). 
  Same legend as Figure \ref{fig::mseBeta-n\nDirac}. 
  \label{fig::auc-n\nDirac-OLS-dirac}}}
\end{figure}

\FloatBarrier

\subsection{Linear discriminant analysis}

In the classification setting, the matrix $X$ corresponds to the 0/1 matrix classifying each of the $n$ observations into one the $q$ classes. Then, the regression coefficients $\beta$ simply give the mean of each response variable in each class and $\Sigma$ the variance matrix of the response variables. This setting corresponds to {\em linear} discriminant analysis (LDA) because $\Sigma$ is supposed to be the same in all classes, as opposed to {\em quadratic} discriminant analysis \citep[see][Chapters 11 and 12]{MKB79}.

\paragraph{Simulation design.} 
We used a design similar to this of \cite{GBR2024}. We considered $p = 3$ classes and the same dimension $q = 20$ for the response vector. We used a fixed covariance matrix $\Sigma$. All entries of the $n \times p$ matrix $X$ from Model \eqref{def::model} are 0's, but the first, second and third of the respective first, second and third column, respectively, which are filled with 1's. This amounts to consider three classes with equal size.  Denoting by $1_q$ (resp. $0_q$) the $q$-dimensional vector made of 1's (resp. 0's), we defined the mean vectors $\mu_1 = -\mu 1_q$, $\mu_2 = 0_q$ and $\mu_3 = +\mu 1_q$ in classes 1, 2, and 3 respectively. The matrix $\beta$ from Model \eqref{def::model} is hence defined as $\beta^\top = [\mu_1 \; \mu_2 \; \mu_3] = \mu [-1_q \; 0_q \; 1_q]$ and
the parameter $\mu$ controls the difficulty of the classification problem. As for the outliers, we replace the $\Ncal(0, \Sigma)$ distribution of the noise vector $\epsilon$ from Model \eqref{def::model} with a uniform distribution over the hypercube $(-20, 20)^q$.

\paragraph{Results.}

Based on the estimates $\widehat{\beta}$ (that is: $\widehat{\mu}_1, \dots \widehat{\mu}_p$) and $\widehat{\Sigma}$, one may classify any observation $i$ into to the class $k \in \llbracket p\rrbracket$ with nearest center, according to the $\ell_2$ norm $\|Y_i - \widehat{\mu}_k\|_{\widehat{\Sigma}^{-1}}$. We compared the classification performances resulting from the different OLS estimation procedure, in terms of adjusted rand index (ARI). We focused on the non-outliers, as the classification of outliers is meaningless. The results are gathered in Figure \ref{fig::mseBeta-n\nn}. 

Obviously, all procedures perform better when $ \mu$ increases, as the classification task becomes easier. Apart from this main effect, the behavior is somewhat similar to Figure \ref{fig::auc-n\nn-OLS}: the ARI of the robust methods remain stable, up to $\prop = 50\%$, whereas this of the naive procedure decreases when $\prop$ increases. Still, the ARI remains high for the naive procedure even for a large fraction $\prop$, when $\mu$ is large, that is when the class are well distinct.

\begin{figure}[ht]
  \begin{center}
    \begin{tabular}{m{0.08\textwidth}m{0.7\textwidth}}
      $\mu$ = 1 & \includegraphics[width=0.7\textwidth, trim=0 60 0 55, clip=]{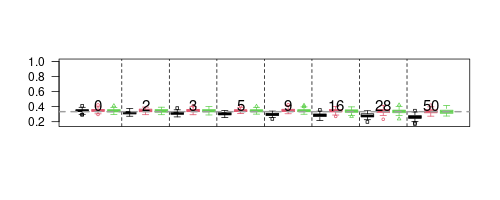} \\
      $\mu$ = 2 & \includegraphics[width=0.7\textwidth, trim=0 60 0 55, clip=]{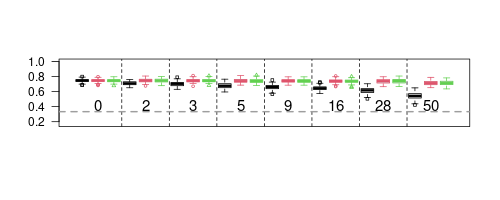} \\
      $\mu$ = 3 & \includegraphics[width=0.7\textwidth, trim=0 60 0 55, clip=]{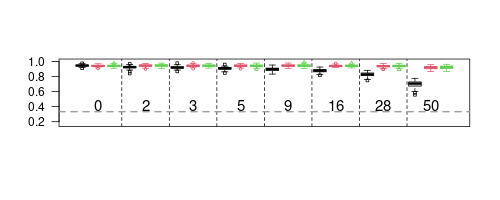} \\
      $\mu$ = 5 & \includegraphics[width=0.7\textwidth, trim=0 60 0 55, clip=]{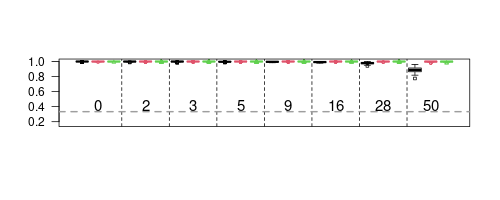} \\
    \end{tabular}
  \end{center}
  \caption{
  Adjusted rand index (ARI) for the OLS procedures. 
  Same legend as Figure \ref{fig::mseBeta-n\nn}. 
  Grey dashed line: $1/p = 1/3$ threshold.
 \label{fig::ari-n\nn}} 
\end{figure}

%

\subsection{Online estimation}\label{sec:simu:online}

{
We now focus on the online estimation setting and consider both OLS and WLS estimators. In the case of WLS, since the covariance matrix of the noise is unknown, we rely on the procedure described in Section~\ref{sec::mahasigma} to estimate it. Importantly, we place ourselves here in a streaming framework, where data arrive sequentially and are not stored in memory (in contrast with the setting considered in Section~\ref{sec::mahasigma}).

We thus consider two online methods in this context. The first one consists in running Algorithm~\ref{asgd} on the first $\alpha n$ observations, with $\alpha \in (0,1)$, which are stored temporarily. This yields an initial estimator of $\beta$, after which we apply the procedure of Section~\ref{sec::mahasigma} to obtain an estimate of the covariance matrix that is then used throughout the remainder of the analysis. These initial data are subsequently discarded, and the following observations are processed online with Algorithm~\ref{asgd::final}. We name this method "Initialized WLS".

The second online method uses the same initialization as the first. However, every $\alpha n$ steps, the covariance matrix is updated again using the procedure from Section~\ref{sec::mahasigma}. While this requires a costly operation on a regular basis, it yields a more accurate online estimate of the covariance matrix. We name this method "Full WLS". 

In Figure~\ref{fig:online_comparison}, we consider the previous setting with a fraction of $10\%$ of outliers, where outliers are generated from a Student distribution with one degree of freedom. 
As expected, the different online methods perform very well, exhibiting a linear decay of the mean squared error. 
However, each time the procedure described in Section~\ref{sec::mahasigma} is applied to update the estimator of $\Sigma$, it induces a substantial computational cost.
\begin{figure}[h!]
    \centering
    \includegraphics[width=0.4\linewidth]{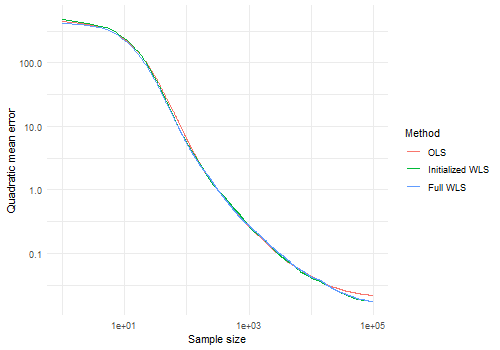}
\quad
    \includegraphics[width=0.4\linewidth]{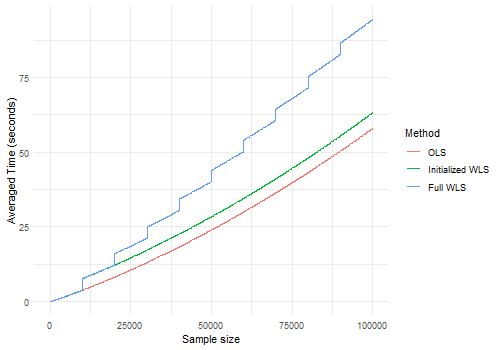}
\caption{On the left, evolution of the quadratic mean error of the different methods with respect to the sample size. On the right, evolution of the computation time of the different methods with respect to the sample size.}
    \label{fig:online_comparison}
\end{figure}
}

\FloatBarrier

\section{Proofs}\label{sec::proof}
\subsection{Proof of Theorem \ref{theo::rob}}
 {
    Let us suppose that the breakdown point is less than $0.5$. Then, there is $f < 0.5$ such that one can construct a sequence of contamination laws $Q_{n}$ such that $\lim_{n\to + \infty}\| \beta^{*} - \beta_{f,Q_n}^{*} \| = + \infty$, which means that 
    \[
 \| \beta_{f,Q_{n}}^{*} \|_{F} \xrightarrow[n\to + \infty]{} + \infty.
    \]
Moreover, one has for all $\beta \in \mathcal{M}_{q,p}(\mathbb{R})$
 \begin{align*}
        G_{f,Q_{n}}(\beta)  & =  (1-f)\iint  \left\| \beta^{*}x  + \epsilon - \beta x \right\| - \| \beta^{*}x + \epsilon \| dF(x)dP_{0}(\epsilon) +f\iint  \left\| \beta^{*}x  + \epsilon - \beta x \right\| - \| \beta^{*}x + \epsilon \| dF(x)dQ_{n}(\epsilon)  \\
        & \geq  (1-f)\iint  \left\| \beta^{*}x  + \epsilon - \beta x \right\| - \| \beta^{*}x + \epsilon \| dF(x)dP_{0}(\epsilon)  - f\iint  \left\|   \beta x \right\|  dF(x)dQ_{n}(\epsilon) \\
        & = (1-f)\iint  \left\| \beta^{*}x  + \epsilon - \beta x \right\| - \| \beta^{*}x + \epsilon \| dF(x)dP_{0}(\epsilon)  - f \mathbb{E} \left[ \left\| \beta X \right\| \right].
 \end{align*}
Observe that, taking $\beta=\beta_{f,Q_{n}}^{*}$, one has
\begin{align*}
\frac{G_{f,Q_{n}} \left( \beta_{f,Q_n}^{*} \right)}{ \mathbb{E} \left[ \left\| \beta_{f,Q_n}^{*}X \right\| \right]} \leq \frac{(1-2f) \mathbb{E} \left[ \left\| \beta_{f,Q_n}^{*} X \right\|\right]}{ \mathbb{E} \left[ \left\| \beta_{f,Q_n}^{*}X \right\| \right]} = 1-2f.
\end{align*}
In a same way, one has
\[
\frac{G_{f,Q_{n}} \left( \beta_{f,Q_n}^{*} \right)}{ \mathbb{E} \left[ \left\| \beta_{f,Q_n}^{*}X \right\| \right]} \geq \frac{(1-2f) \mathbb{E} \left[ \left\| \beta_{f,Q_n}^{*} \right\|\right] -2f \iint   \| \beta^{*}x + \epsilon \| dF(x)dP_{0}(\epsilon)}{\mathbb{E} \left[ \left\| \beta_{f,Q_n}^{*}X \right\| \right]} \xrightarrow[n\to +\infty]{} 1-2f,
\]
so that
\[
\frac{G_{f,Q_{n}} \left( \beta_{f,Q_n}^{*} \right)}{ \mathbb{E} \left[ \left\| \beta_{f,Q_n}^{*}X \right\| \right]}  \xrightarrow[n\to +\infty]{} 1-2f .
\]
Nevertheless, by definition of $\beta_{f,Q_n}^{*}$, one has
\[
G_{f,Q_{n}} \left( \beta_{f,Q_n}^{*} \right) \leq G_{f,Q_{n}} \left( 0\right) = 0,
\]
so that 
\[
1-2f = \lim_{n\to + \infty} \frac{G_{f,Q_{n}} \left( \beta_{f,Q_n}^{*} \right)}{ \mathbb{E} \left[ \left\| \beta_{f,Q_n}^{*}X \right\| \right]} \leq 0
\]
which is a contradiction with the fact that $f < 0.5$, that concludes the proof.
}

\subsection{Proof of Theorem \ref{theo::without::sigma}}

First, let us vectorizing the algorithm. In this aim, denoting by $Y[ j]$, $X[i]$ and $\beta[j,i]$ the coordinates, let us remark that one can rewrite rewrite $G$ as  

\[
G(\beta) = \mathbb{E}\left[ \sqrt{ \sum_{j=1}^{q}\left( Y[j] - \sum_{i=1}^{p} \beta[j,i]  X[i] \right)^{2}        } \right]
\]
and its gradient,  in the basis of $\mathbb{R}^{pq}$, is given by
\[
\nabla \tilde{G} ( \varphi ( \beta ) ) = -  \mathbb{E}\left[ \underbrace{ \frac{1}{\left\| Y -     \beta X \right\|} \begin{pmatrix}
\left( Y[1] - \sum_{i=1}^{p} \beta[1,i] X[i] \right) X \\
\vdots \\
\left( Y[q] - \sum_{i=1}^{p} \beta[q,i] X[i]\right) X
\end{pmatrix} }_{=:  \nabla \tilde{g}(X , Y , \varphi (\beta) )}  \right]
\]
In other word 
\[
\nabla \tilde{G}  (\varphi (\beta)) = - \mathbb{E}\left[ \frac{1}{\left\| Y -  \beta X \right\| } \left(  Y -   \beta X  \right) \odot X \right]
\]
where for all $z  \in \mathbb{R}^{q} $, 
\[
z \odot  X = \begin{pmatrix}
z[1] X \\
\vdots \\
z[q] X
\end{pmatrix}
\]

Observe that $\nabla \tilde{G}(\beta) = \varphi \left(  \nabla G (\beta) \right)$,
so that regarding the stochastic gradient algorithm given by \eqref{sgd} or the stochastic gradient algorithm given by $\nabla \tilde{G} (.)$ is equivalent; i.e considering the stochastic gradient algorithm given by
\[
 {\beta}_{n+1 , \varphi} =  {\beta}_{n, \varphi} + \gamma_{n+1} \nabla\tilde{g} \left( X_{n+1} , Y_{n+1} ,  {\beta}_{n , \varphi} \right) ,
\]
and its weighted averaged version denoted $\overline{\beta}_{n, \varphi}$, one has $ {\beta}_{n, \varphi} = \varphi \left( \beta_{n } \right) $ and $\overline{\beta}_{n, \varphi} = \varphi \left( \overline{\beta}_{n } \right)$.
Then, the aim is to verify that $\beta_{n,\varphi}$ satisfies assumptions in \cite{godichon2016} among others.

\textbf{Bounding the gradient: }
As soon as $X$ admits a moment of order $2a$ for some $a \geq 1$, one has for any $\beta \in \mathcal{M}_{q,p}(\mathbb{R})$,
\[
\mathbb{E}\left[ \left\| \nabla \tilde{g} \left( X , Y , \varphi (\beta) \right) \right\|^{2a} \right] \leq q^{a} \mathbb{E} \left[ \| X \|^{2a} \right]
\]
and Assumption \textbf{(A1)} in \cite{godichon2016} is satisfied if $X$ admits a moment of order $p'$ for any positive integer $p'$.

\textbf{Bounding the Hessian:} Thanks to Lemma \ref{lem::1}, the functional $G$ is twice differentiable and its  Hessian is of the form
\[
\nabla^{2} \tilde{G}(\varphi (\beta)) = \mathbb{E}\left[ \frac{1}{\left\| Y -  \beta X \right\|} \left( A  - B(\beta) \right) \right]
\]
with $A =\text{Diag} \left( X X^{T} , \ldots ,X X^{T} \right)$ and $B(\beta) = \left( B_{j,k}(\beta) \right)$ where $B_{j,k}(\beta)_{j,k=1, \ldots ,q} \in  \mathcal{M}_{q,q}(\mathbb{R})$ with 
\[
B_{j,k}(\beta) =\frac{ \left( Y[j] - \sum_{i=1}^{p} X[j]\beta[j,i] \right)}{\left\| Y -  \beta X \right\|} \frac{\left( Y[k] - \sum_{i=1}^{p} X[k]\beta[k,i] \right)}{\left\| Y -  \beta X \right\|} XX^{T} .
\]
In other word, one can rewrite the  Hessian as
\[
\nabla^{2} \tilde{G}(\varphi (\beta)) = \mathbb{E}\left[ \frac{1}{\left\| Y -  \beta X \right\|} \left( I_{q} - \frac{\left( Y -  \beta X \right) }{\left\| Y -  \beta X \right\|} \frac{\left( Y -   \beta X \right)^{T}}{\left\| Y -  \beta X \right\|} \right) \otimes X X^{T} \right] 
\]
where $\otimes$ is the Kronecker product.  Observe that the Hessian is so at least semi-definite positive. Thanks to Lemma \ref{lem::1}, one has
\begin{align*}
\left\| \nabla^{2} \tilde{G} \left( \varphi (\beta) \right) \right\|_{F} & \leq \left\| \mathbb{E}\left[ \mathbb{E}\left[ \frac{1}{\left\| Y -  \beta X \right\|} \left( I_{q} - \frac{\left( Y -  \beta X \right) }{\left\| Y -  \beta X \right\|} \frac{\left( Y -   \beta X \right)^{T}}{\left\| Y -  \beta X \right\|} \right) \otimes X X^{T} |X \right]  \right]  \right\|_{F} \\
& \leq \left( 1+ \frac{ \lambda_{\max}^{q/2}}{q2^{q/2 -1}\Gamma (q/2)}\frac{1}{q-1} \right) \mathbb{E}\left[ \left\| I_{q} \otimes XX^{T} \right\|_{F} \right] \\
&  = \left( 1+ \frac{ \lambda_{\max}^{q/2}}{\sqrt{q}2^{q/2 -1}\Gamma (q/2)}\frac{1}{q-1} \right) {\mathbb{E}\left[ \| X \|^{2} \right] }
\end{align*}
and the Hessian is so continuous, uniformly bounded,  and satisfy Assumption \textbf{(A3)} in \cite{godichon2016}.  In addition, one has  \[
\nabla^{2}\tilde{G} (\varphi (\beta^*)) = \mathbb{E}\left[ \frac{1}{\left\| \epsilon \right\|}  \left(  I_{q}  - \frac{\epsilon}{\| \epsilon\|} \left( \frac{\epsilon}{\| \epsilon \|}\right)^{T} \right) \otimes XX^{T} \right] .
\]
Since $\epsilon$ and $X$ are independent, 
\[\nabla^{2}\tilde{G} (\varphi (\beta^*)) = \mathbb{E}\left[ \frac{1}{\left\| \epsilon \right\|}  \left(  I_{q}  - \frac{\epsilon}{\| \epsilon\|} \left( \frac{\epsilon}{\| \epsilon \|}\right)^{T} \right)\right]  \otimes \mathbb{E}\left[ XX^{T} \right]\]
which is positive since $\Sigma $ and  $\mathbb{E} \left[ X X^{T} \right] $ are positive.  Assumption \textbf{(A2)} in \cite{godichon2016} is so satisfied. 
Let us now prove that the Hessian is Lipschitz. First notice that for all $\beta , \beta ' \in \mathcal{M}_{q,p} (\mathbb{R})$, one has
\begin{align*} 
\left\| \nabla^{2}\tilde{G} (\varphi (\beta )) - \nabla^{2}\tilde{G} (\varphi (\beta ')) \right\|_{F} & \leq \left. \mathbb{E}\left[ \left| \frac{1}{\|Y-\beta X\|} -   \frac{1}{\|Y-\beta ' X\|} \right| \left\| I_{q} \otimes XX^{T} \right\|_{F} \right] \right\rbrace =: R_{1}  \\
& +\underbrace{\mathbb{E}\left[ \frac{1}{\left\| Y - \beta '  X\right\|} \left\|  \frac{\left( Y -  \beta X \right) }{\left\| Y -  \beta X \right\|} \frac{\left( Y -   \beta X \right)^{T}}{\left\| Y -  \beta X \right\|} - \frac{\left( Y -  \beta ' X \right) }{\left\| Y -  \beta ' X \right\|} \frac{\left( Y -   \beta ' X \right)^{T}}{\left\| Y -  \beta ' X \right\|}  \right\|_{F} \ \left\| X \right\|^{2}  \right]}_{ =: R_{2}}
 \end{align*}
 First observe that the gradient of the function $\beta \longmapsto   \frac{1}{\left\| Y - \beta X \right\|} $ is given by $\frac{  Y - \beta X  }{\left\| Y - \beta X \right\|^{3}} \odot X$. Then, thanks to Cauchy-Schwarz inequality coupled with Lemma \ref{lem::1}, 
 \begin{align*}
 R_{1} \leq  \int_{0}^{1} \mathbb{E}\left[ \frac{1}{\left\| Y - ( \beta + t(\beta ' - \beta ))X \right\|^{2} } \sqrt{q}  \| X \|^{3} \right] dt \left\| \beta - \beta ' \right\|  & \leq  \left( 1+ \frac{\lambda_{\max}^{q/2}}{\sqrt{q}2^{q/2    }\Gamma (q/2)}\frac{1}{q-2} \right) \mathbb{E}\left[ \| X \|^{3} \right] \| \beta - \beta ' \| .
 \end{align*}
 Second, observe that (see \cite{CCM10} page 25)
 \begin{align*}
\left\|  \frac{Y - \beta X }{\left\| Y  - \beta X \right\| } - \frac{Y - \beta '  X }{\left\| Y  - \beta '  X \right\| } \right\| \leq \frac{\left\| \beta X - \beta ' X \right\| }{\left\| Y - \beta ' X \right\| }  \leq \frac{\| X \|}{\left\| Y  - \beta ' X \right\| } \left\| \beta - \beta ' \right\|_{F}.
\end{align*}  
Then, since $aa^{T}-bb^{T} = (a-b)a^{T} - b\left(b^{T} - a^{T} \right)$ and using Cauchy Schwarz inequality,
\begin{align*}
R_{2} & \leq \mathbb{E}\left[  \mathbb{E}\left[ \left( \frac{1}{\left\| Y - \beta ' X \right\|^{2} } + \frac{1}{\left\| Y - \beta X \right\|  \left\| Y  - \beta ' X \right\| } \right) |X\right] \| X  \|^{3}  \right] \left\| \beta - \beta ' \right\|_{F} \\
& \leq 2\left( 1+ \frac{\lambda_{\max}^{q/2}}{\sqrt{q}2^{q/2   }\Gamma (q/2)}\frac{1}{q-2} \right) \mathbb{E}\left[ \| X \|^{3} \right] \| \beta - \beta ' \|_{F}
\end{align*}
Then,
\[
\left\| \nabla^{2}\tilde{G} (\varphi (\beta )) - \nabla^{2}\tilde{G} (\varphi (\beta ')) \right\|_{F} \leq 3 \left( 1+ \frac{\sqrt{q}\lambda_{\max}^{q/2}}{2^{q/2  }\Gamma (q/2)}\frac{1}{q-2} \right) \mathbb{E}\left[ \| X \|^{3} \right] \| \beta - \beta ' \|_{F}
\]
and Assumption \textbf{(A5)} in \cite{godichon2016} is so satisfied.

\textbf{Convergence of the estimates: }
Assumptions in \cite{GB2017,godichon2016} are fulfilled and it comes
\begin{align*}
\left\| \beta_{n} - \beta^{*} \right\|_{F}^{2} & = \left\| \beta_{n,\varphi} - \varphi \left(\beta^{*} \right) \right\|^{2} = O \left( \frac{\ln n}{n^{\gamma}} \right) \quad a.s \\
\left\| \overline{\beta}_{n} - \beta^{*} \right\|_{F}^{2} & = \left\| \overline{\beta}_{n,\varphi} -\varphi \left(\beta^{*} \right) \right\|^{2} = O \left( \frac{\ln n}{n } \right) \quad a.s
\end{align*}
In addition, one has
\[
\sqrt{n} \left( \varphi \left( \overline{\beta}_{n} \right) - \varphi \left( \beta^{*} \right) \right) \xrightarrow[n\to + \infty]{\mathcal{L}} \mathcal{N} \left( 0 , H^{-1}V H^{-1} \right)
\]
where
$H = \mathbb{E}\left[\frac{1}{\left\| \epsilon \right\|}  \left(  I_{q}  - \frac{\epsilon}{\| \epsilon\|} \left( \frac{\epsilon}{\| \epsilon \|}\right)^{T} \right)\right] \otimes \mathbb{E}\left[ XX^{T} \right]$ and since $\epsilon $ and $X$ are independent,
\begin{align*}
V = \mathbb{E}\left[ \nabla \tilde{g} \left( X , Y , \beta^{*} \right)\nabla \tilde{g} \left( X , Y , \beta^{*} \right)^{T} \right] & = \mathbb{E}\left[ \frac{1}{\left\| Y - \beta^{*}X \right\|^{2}} \left( Y - \beta^{*}X \right)\left( Y - \beta^{*}X \right) ^{T} \otimes XX^{T} \right] \\
& = \mathbb{E}\left[ \frac{\epsilon}{\| \epsilon\|}\frac{\epsilon^{T}}{\| \epsilon \|} \right] \otimes \mathbb{E}\left[ XX^{T} \right]  .
\end{align*}
Then, since $(A\otimes B ) (C \otimes D) = (A  C) \otimes (BD)$, it comes \small
\[\sqrt{n} \left( \varphi \left( \overline{\beta}_{n} \right) - \varphi \left( \beta^{*} \right) \right) \xrightarrow[n\to + \infty]{\mathcal{L}} \mathcal{N} \left( 0 , \mathbb{E}\left[\frac{1}{\left\| \epsilon \right\|}  \left(  I_{q}  - \frac{\epsilon}{\| \epsilon\|}   \frac{\epsilon^{T}}{\| \epsilon \|}  \right)\right]^{-1}\mathbb{E}\left[  \frac{\epsilon}{\| \epsilon\|}   \frac{\epsilon^{T}}{\| \epsilon \|}  \right]\mathbb{E}\left[\frac{1}{\left\| \epsilon \right\|}  \left(  I_{q}  - \frac{\epsilon}{\| \epsilon\|}   \frac{\epsilon^{T}}{\| \epsilon \|}  \right)\right]^{-1} \otimes \left( \mathbb{E}\left[ XX^{T} \right] \right)^{-1} \right) .\]
\normalsize
Finally, thanks to \cite{godichon2016}, there are positive constants $C_{0} , C_{1}$ such that for all $n \geq 1$
\begin{align*}
\left\| \beta_{n} - \beta^{*} \right\|_{F}^{2} & = \left\| \beta_{n,\varphi} - \varphi (\beta) \right\|^{2}  \leq \frac{C_{0}}{n^{\gamma}} \\
\left\| \overline{\beta}_{n} - \beta^{*} \right\|_{F}^{2} & = \left\| \overline{\beta}_{n,\varphi} - \varphi (\beta) \right\|^{2} \leq \frac{C_{1}}{n} .
\end{align*}

\subsection{Proof of Theorem \ref{theo::sigma}}

\textbf{Writing in the basis of $\mathbb{R}^{pq}$: }
Vectorizing $\beta$,  the gradient in the basis of $\mathbb{R}^{pq} $ can be written as
\[
\nabla \tilde{G}_{\Sigma} (\varphi (\beta)) =  -    \mathbb{E}\left[ \underbrace{   \frac{1}{\left\| Y -  \beta X \right\|_{\Sigma^{-1}}} \left( \Sigma^{-1}(Y -   \beta X) \right) \odot X  }_{=:  \nabla \tilde{g}_{\Sigma}(X,Y,\beta)}\right] .
\]
Hence, here again, the stochastic gradient given by \eqref{sgd::Sigma} and the one derived from $\nabla \tilde{G}_{\Sigma}$ are equivalent. Indeed, considering the stochastic gradient algorithm defined recursively for all $n \geq 0$ by
\[
\beta_{n+1 , \Sigma ,\varphi} = \beta_{n ,\Sigma, \varphi} + \gamma_{n+1} \nabla \tilde{g}_{\Sigma} \left( X_{n+1} , Y_{n+1} , \beta_{n,\Sigma, \varphi} \right)
\]
and denoting by $\overline{\beta}_{n, \Sigma , \varphi}$ its weighted averaged version, one has $\beta_{n,\Sigma , \varphi} = \varphi \left( \beta_{n,\Sigma} \right)$ and $\overline{\beta}_{n, \Sigma , \varphi} = \varphi \left( \overline{\beta}_{n, \Sigma} \right)$. Then, as in the proof of Theorem \ref{theo::without::sigma}, the aim is to check that assumptions in \cite{godichon2016} (among others) are satisfied.

\textbf{Bounding the gradient: } As soon as $X$ admits a moment of order $2a$ for some $a \geq 1$, one has for any $\beta \in \mathcal{M}_{q,p} (\mathbb{R})$
\[
\mathbb{E}\left[ \left\| \nabla \tilde{g}_{\Sigma} \left( X , Y , \varphi (\beta) \right) \right\|^{2a} \right] \leq \frac{q}{\sqrt{\lambda_{\min} (\Sigma)}} \mathbb{E}\left[ \left\| X \right\|^{2a} \right]
\]
and Assumption \textbf{(A1)} in \cite{godichon2016} is so satisfied.

\textbf{Bounding the Hessian}
The Hessian is given by 
\[
\nabla^{2} \tilde{G}_{\Sigma} (\varphi (\beta)) = \mathbb{E}\left[ \frac{1}{\left\| Y -  \beta X \right\|_{\Sigma^{-1}} }   \left( \Sigma^{-1} -  \frac{\Sigma^{-1}  ( Y -  \beta X)}{\left\| Y - \beta X \right\|_{\Sigma^{-1}}}\frac{ ( Y -  \beta X )^{T}\Sigma^{-1} }{\left\| Y -  \beta X \right\|_{\Sigma^{-1}}} \right) \otimes XX^{T}\right] .
\]
Moreover, as in Theorem \ref{theo::without::sigma}, one has with the help of Lemma \ref{lem::2},
\begin{align*}
\left\| \nabla^{2} \tilde{G}_{\Sigma} (\varphi (\beta)) \right\|_{F} & \leq   \mathbb{E}\left[ \mathbb{E}\left[ \frac{1}{\left\| Y - \beta X \right\|_{\Sigma^{-1}}} | X \right] \left\| \Sigma^{-1} \right\|_{F} \| X \|^{2} \right] \\
& \leq \left( 1+ \frac{1}{q2^{q/2 -1}\Gamma (q/2)}\frac{1}{q-1} \right) \left\| \Sigma^{-1} \right\|_{F} \mathbb{E}\left[ \| X\|^{2} \right]
\end{align*}
and the Hessian is so continuous and satisfies Assumption \textbf{(A3)} in \cite{godichon2016}.
In addition,  denoting by $\tilde{\epsilon} = \Sigma^{-1/2} \epsilon$, observe that
\begin{align*}
H_{\Sigma}:=\nabla^{2}\tilde{G} (\varphi (\beta^*)) = \mathbb{E}\left( \frac{1}{\left\| \tilde{\epsilon} \right\|} \left( \Sigma^{-1} -  \Sigma^{-1/2} \frac{\tilde{\epsilon}}{\| \tilde{\epsilon} \|}\frac{\tilde{\epsilon}^{T}}{\| \tilde{\epsilon} \|} \Sigma^{-1/2} \right) \otimes XX^{T} \right] .
\end{align*}
Then, since $X$ is independent from $\tilde{\epsilon}$, and remarking that $\frac{\tilde{\epsilon}}{\| \tilde{\epsilon}\|}$ follows an uniform law on the sphere independent from $\| \tilde{\epsilon} \|$, it comes
\[
H_{\Sigma}:=\nabla^{2}\tilde{G} (\varphi (\beta^*)) = \mathbb{E}\left[ \frac{1}{\left\| \tilde{\epsilon} \right\|} \right]  \left( \Sigma^{-1} -  \Sigma^{-1/2}\mathbb{E}\left[  \frac{\tilde{\epsilon}}{\| \tilde{\epsilon} \|}\frac{\tilde{\epsilon}^{T}}{\| \tilde{\epsilon} \|} \right] \Sigma^{-1/2} \right) \otimes \mathbb{E}\left[ XX^{T} \right]
\]
leading to
\[
H_{\Sigma}  = \mathbb{E}\left[ \frac{1}{\left\| \mathcal{N}\left( 0 ,I_{q} \right) \right\|} \right] \frac{q-1}{q}  \Sigma^{-1} \otimes \mathbb{E} \left[ XX^{T} \right] 
\]
which is positive definite as soon as $\mathbb{E}\left[ XX^{T} \right]$ is. Then, Assumption \textbf{(A2)} in \cite{godichon2016} is so satisfied.

{Note that, for $q>1$,
if $Z\sim\mathcal{N}\left( 0 , I_{q} \right)$, $U=\|Z\|\sim\chi(q)$ (chi-distribution) with pdf: $f_q(u)=\frac{u^{q-1}e^{-u^2/2}}{2^{q/2-1}\Gamma(\frac{q}{2})}$ for $u\geq0$. Hence,
\[\mathbb{E}\left[ \frac{1}{\left\| \mathcal{N}\left( 0 , I_{q} \right) \right\|} \right] = \mathbb{E}\left(\frac{1}{U}\right) = \int_0^\infty \frac1u f_q(u) \,du
= \int_0^\infty \frac{2^{(q-1)/2-1}\Gamma(\frac{q-1}{2})}{2^{q/2-1}\Gamma(\frac{q}{2})} f_{q-1}(u) \,du = \frac{1}{\sqrt2}\frac{\Gamma(\frac{q-1}{2})}{\Gamma(\frac{q}{2})}.\]}

\textbf{The Hessian is Lipschitz: } 
Let us now prove that the Hessian is Lipschitz. Note that for all $\beta ,  \beta ' \in \mathcal{M}_{q,p}(\mathbb{R})$, 
\begin{align*}
&  \left\| \nabla^{2} \tilde{G}_{\Sigma} (\varphi (\beta)) - \nabla^{2} \tilde{G}_{\Sigma} (\varphi (\beta')) \right\|_{F}    \left. \leq \mathbb{E}\left[ \left| \frac{1}{\left\| Y - \beta X \right\|_{\Sigma^{-1}} } - \frac{1}{\left\| Y - \beta X \right\|_{\Sigma^{-1}} } \right| \left\| \Sigma^{-1} \right\|_{F}  \left\| X \right\|^{2} \right] \right\rbrace  =: R_{1}' \\
 & + \left. \mathbb{E}\left[ \frac{1}{\left\| Y - \beta ' X \right\|_{\Sigma^{-1}}} \left\| \frac{\Sigma^{-1}(Y - \beta X)}{\left\| Y - \beta X \right\|_{\Sigma^{-1}}}  \frac{(Y - \beta X)^{T}\Sigma^{-1}}{\left\| Y - \beta X \right\|_{\Sigma^{-1}}} - \frac{\Sigma^{-1}(Y - \beta' X)}{\left\| Y - \beta' X \right\|_{\Sigma^{-1}}}  \frac{ (Y - \beta ' X)^{T}\Sigma^{-1}}{\left\| Y - \beta ' X \right\|_{\Sigma^{-1}}} \right\|_{F} \right] \| X \|^{2} \right\rbrace =: R_{2}' .
\end{align*}
Analogously to the proof of Theorem \ref{theo::without::sigma}, the gradient of the functional $\beta \longrightarrow   \frac{1}{\left\| Y - \beta X \right\|_{\Sigma^{-1}}}  $ is given by $  \frac{\Sigma^{-1}(Y- \beta X)}{\left\| Y - \beta X \right\|_{\Sigma^{-1}}^{3}} \odot X$, leading to
\begin{align*}
R_{1}' & \leq \int_{0}^{1} \mathbb{E}\left[ \frac{\left\| \Sigma^{-1/2} \right\|_{F}}{\left\| Y - (\beta + t(\beta ' - \beta)) X \right\|_{\Sigma^{-1}}^{2} } \left\| \Sigma^{-1} \right\|_{F} \| X \|^{3}  \right] dt \\
&\leq  \left( 1+ \frac{1}{q2^{q/2  }\Gamma (q/2)}\frac{1}{q-2}  \right)\left\| \Sigma^{-1/2} \right\|_{F}\left\| \Sigma^{-1} \right\|_{F} \mathbb{E}\left[ \left\| X \right\|^{3} \right]
\end{align*}
In addition, since for all $\beta , \beta ' \in \mathcal{M}_{q,p}(\mathbb{R})$
\begin{align*}
\left\|  \frac{\Sigma^{-1}(Y - \beta X)}{\left\| Y - \beta X \right\|_{\Sigma^{-1}}} -  \frac{\Sigma^{-1}(Y - \beta ' X)}{\left\| Y - \beta ' X \right\|_{\Sigma^{-1}}} \right\| & \leq \left\|  \Sigma^{-1/2}\frac{\Sigma^{-1/2}(Y - \beta X)}{\left\| \Sigma^{-1/2}( Y - \beta X ) \right\|} -  \frac{\Sigma^{-1}(Y - \beta ' X)}{\left\| \Sigma^{-1/2}( Y - \beta ' X ) \right\| } \right\| \\
&  \leq \left\| \Sigma^{-1/2} \right\|_{F} \frac{ \left\| \Sigma^{-1/2} \left( \beta X - \beta ' X \right) \right\|}{\left\| Y - \beta ' X \right\|_{\Sigma^{-1}}} \leq \frac{ \left\| \Sigma^{-1/2} \right\|_{F}^{2} \| X \| }{\left\| Y - \beta ' X \right\|_{\Sigma^{-1}}}\| \beta   - \beta ' \|_{F}
\end{align*}
 
Then, since $aa^{T}-bb^{T} = (a-b)a^{T} - b\left(b^{T} - a^{T} \right)$ and using Cauchy Schwarz inequality,
\begin{align*}
R_{2}' & \leq \left\| \Sigma^{-1/2} \right\|_{F}^{3} \mathbb{E}\left[  \mathbb{E}\left[ \left( \frac{1}{\left\| Y - \beta ' X \right\|_{\Sigma^{-1}}^{2} } + \frac{1}{\left\| Y - \beta X \right\|_{\Sigma^{-1}} \left\| Y  - \beta ' X \right\|_{\Sigma^{-1 }}} \right) |X\right] \| X  \|^{3}  \right] \left\| \beta - \beta ' \right\|_{F} \\
& \leq  \left( 1+ \frac{1}{q2^{q/2  }\Gamma (q/2)}\frac{1}{q-2} \right) \left\| \Sigma^{-1/2} \right\|_{F}^{3} \mathbb{E}\left[ \| X \|^{3} \right] \left\| \beta - \beta ' \right\|_{F}
\end{align*}
Then, since $\left\| \Sigma^{-1/2} \right\|_{F} \left\| \Sigma^{-1} \right\|_{F} \leq \left\| \Sigma^{-1/2} \right\|^{3} $
\begin{align*}
\left\| \nabla^{2} \tilde{G}_{\Sigma} (\varphi (\beta)) - \nabla^{2} \tilde{G}_{\Sigma} (\varphi (\beta')) \right\|_{F}  \leq 3 \left( 1+ \frac{1}{q2^{q/2  }\Gamma (q/2)}\frac{1}{q-2} \right) \left\| \Sigma^{-1/2} \right\|_{F}^{3} \mathbb{E}\left[ \| X \|^{3} \right] \left\| \beta - \beta ' \right\|_{F}
\end{align*}
and Assumption \textbf{(A5)} in \cite{godichon2016} is so satisfied.

\textbf{Convergence of the estimates: }
Assumptions in \cite{GB2017,godichon2016} are fulfilled and it comes
\begin{align*}
\left\| \beta_{n,\Sigma} - \beta^{*} \right\|_{F}^{2} & = \left\| \beta_{n,\varphi} - \varphi \left(\beta^{*} \right) \right\|^{2} = O \left( \frac{\ln n}{n^{\gamma}} \right) \quad a.s \\
\left\| \overline{\beta}_{n,\Sigma} - \beta^{*} \right\|_{F}^{2} & = \left\| \overline{\beta}_{n,\varphi} -\varphi \left(\beta^{*} \right) \right\|^{2} = O \left( \frac{\ln n}{n } \right) \quad a.s
\end{align*}
In addition, one has
\[
\sqrt{n} \left( \varphi \left( \overline{\beta}_{n,\Sigma} \right) - \varphi \left( \beta^{*} \right) \right) \xrightarrow[n\to + \infty]{\mathcal{L}} \mathcal{N} \left( 0 , H_{\Sigma}^{-1}V_{\Sigma}H_{\Sigma}^{-1} \right)
\]
{where $H_{\Sigma}:=\nabla^{2}\tilde{G}_{\Sigma} (\varphi (\beta^*)) = \frac{1}{\sqrt2}\frac{(q-1)\Gamma(\frac{q-1}{2})}{q\Gamma(\frac{q}{2})} \Sigma^{-1} \otimes \mathbb{E} \left[ XX^{T} \right] $ and }

\begin{align*}
V_{\Sigma}    = \mathbb{E}\left[ \nabla \tilde{g}_{\Sigma} \left( X , Y , \beta^{*} \right)\nabla \tilde{g}_{\Sigma} \left( X , Y , \beta^{*} \right)^{T} \right] & = \mathbb{E}\left[ \frac{1}{\left\| Y - \beta^{*}X \right\|_{\Sigma^{-1}}^{2}} \Sigma^{-1} \left( Y - \beta^{*}X \right)\left( Y - \beta^{*}X \right)^{T} \Sigma^{-1} \otimes XX^{T} \right] \\
& = \mathbb{E}\left[ \frac{1}{\left\| \epsilon \right\|_{\Sigma^{-1}}^{2}} \Sigma^{-1}\epsilon \epsilon^{T} \Sigma^{-1} \otimes XX^{T} \right] .
\end{align*} 
Then, since $\epsilon$ and $X$ are independent, and since $\frac{\Sigma^{-1/2}\epsilon}{\left\| \epsilon \right\|_{\Sigma^{-1}}} \sim \mathcal{U} \left( \mathcal{S}^{q} \right)$ is independent from $\| \epsilon\|_{\Sigma^{-1}}$, it comes 
\begin{align*}
V_{\Sigma} = \Sigma^{-1/2} \mathbb{E}\left[ \frac{\Sigma^{-1/2}\epsilon}{\left\| \epsilon \right\|_{\Sigma^{-1}}}\left( \frac{\Sigma^{-1/2}\epsilon}{\left\| \epsilon \right\|_{\Sigma^{-1}}} \right)^{T} \right] \Sigma^{-1/2}\otimes \mathbb{E}\left[ XX^{T} \right] =  \frac{1}{q}\Sigma^{-1}\otimes \mathbb{E}\left[ XX^{T} \right] 
\end{align*}
Then, since $(A\otimes B ) (C \otimes D) = (A  C) \otimes (BD)$, it comes
{\[
\sqrt{n} \left( \varphi \left( \overline{\beta}_{n,\Sigma} \right) - \varphi \left( \beta^{*} \right) \right) \xrightarrow[n\to + \infty]{\mathcal{L}} \mathcal{N} \left( 0 , \frac{2q}{(q-1)^{2}} \frac{\Gamma(\frac{q}{2})^2}{\Gamma(\frac{q-1}{2})^2} \Sigma \otimes \left( \mathbb{E}\left[ XX^{T} \right] \right)^{-1} \right) .
\]}
Finally, thanks to \cite{godichon2016}, there are positive constants $C_{0,\Sigma} , C_{1,\Sigma}$ such that for all $n \geq 1$
\begin{align*}
\left\| \beta_{n,\Sigma} - \beta^{*} \right\|_{F}^{2} & = \left\| \beta_{n,\Sigma,\varphi} - \varphi (\beta) \right\|^{2}  \leq \frac{C_{0,\Sigma}}{n^{\gamma}} \\
\left\| \overline{\beta}_{n,\Sigma} - \beta^{*} \right\|_{F}^{2} & = \left\| \overline{\beta}_{n,\Sigma,\varphi} - \varphi (\beta) \right\|^{2} \leq \frac{C_{1,\Sigma}}{n} .
\end{align*}

\subsection{Proof of Theorem \ref{theo::fix}}
In all the sequel, for the sake of simplicity, we denote $\beta_{t}:= \beta_{n,t}$ for all $t \geq 0$.
Let us first give a useful result on the Hessian of $G_{n}$.
\begin{lemma}\label{useful::lemma}
Suppose that $\sum_{i=1}^{n}X_{i}X_{i}^{T}$ is positive, then for all   $\beta \in \mathcal{M}_{q,p}(\mathbb{R})$, the matrix
$\sum_{i=1}^{n} \frac{X_{i}X_{i}^{T}}{\left\| Y_{i} - \beta X_{i} \right\|}$ is positive.
\end{lemma}

\begin{proof}[of Lemma \ref{useful::lemma}]
Observe that $\sum_{i=1}^{n} X_{i}X_{i}^{T}$ positive implies necessary that there are $p$ indices $i_{1},\ldots ,i_{p}$ such that $ \left\lbrace X_{i_{1}},\ldots , X_{i_{p}} \right\rbrace$ are linearly independent. Then, denoting by $\succcurlyeq$ the order relation between symmetric matrices, one has
\[
\sum_{i=1}^{n} \frac{X_{i}X_{i}^{T}}{\left\| Y_{i} - \beta X_{i} \right\|} \succcurlyeq \sum_{j=1}^{p} \frac{X_{i_{j}}X_{i_{j}}^{T}}{\left\| Y_{i_{j}} - \beta X_{i_{j}} \right\|} \succcurlyeq \sum_{j=1}^{p} \frac{X_{i_{j}}X_{i_{j}}^{T}}{\left\| Y_{i_{j}} \right\| +  \|\beta\|_{F} \left\| X_{i_{j} } \right\| }  \succ 0 . 
\] 
\end{proof}

Let us now  follow the scheme of proof in \cite{beck2015weiszfeld}. In this aim, let us first prove the following Lemma (which can be seen as a translation of Lemma 5.1 in \cite{beck2015weiszfeld}).
\begin{lemma}\label{lem::1::bs}
For all $ \beta ' \in \mathcal{M}_{q,p}(\mathbb{R})$,
\begin{align*}
2 G_{n} \left( T_{n} (\beta ') \right) \leq&  2 G_{n} ( \beta ') + 2 \left\langle \nabla G_{n}(\beta ') , T_{n}(\beta ') - \beta ' \right\rangle_{F} \\
& +  \left(  \varphi(T_{n}(\beta')) - \varphi(\beta ') \right) ^{T} L_{n}(\beta ') \left(  \varphi( T_{n}(\beta ')) - \varphi(\beta ') \right) .
\end{align*}
\end{lemma}
\begin{proof}[of Lemma \ref{lem::1::bs}]
Let us consider the functional $h_{n} : \mathcal{M}_{q,p}(\mathbb{R}) \times \mathcal{M}_{q,p}(\mathbb{R}) \longrightarrow \mathbb{R}_{+}$ defined for all $\beta , \beta ' $ by
\[
h_{n} (\beta , \beta ' ) = \frac{1}{n}\sum_{i=1}^{n} \frac{\left\| Y_{i} - \beta X_{i} \right\|^{2}}{\left\| Y_{i} - \beta ' X_{i} \right\|} .
\] 
Then, the functional $h_{n}( . , \beta ')$ is quadratic, and writting it in the canonical basis of $\mathbb{R}^{pq}$, the Hessian is given by $2L_{n}(\beta ')$ with $L_{n}(\beta ') := I_{q} \otimes  \frac{1}{n}\sum_{i=1}^{n} \frac{X_{i}X_{i}^{T}}{\left\| Y_{i} - X_{i} \beta ' \right\|}$. Then, denoting by $\nabla_{\beta}h_{n}(.,.)$ the gradient of $h_{n}$ with respect to the first{argument}, and with the help of a Taylor's expansion, it comes
\[
h_{n} (\beta , \beta ' ) = h_{n}( \beta ' , \beta ' ) + \left\langle \nabla_{\beta}h_{n}(\beta ' , \beta ') , \beta - \beta ' \right\rangle + \left(  \varphi(\beta) - \varphi(\beta ') \right) ^{T} L_{n}(\beta ') \left(  \varphi(\beta) - \varphi(\beta ') \right)  
\] 
Observe that $h(\beta ' , \beta ' ) = G_{n}(\beta ')$ and that $\nabla_{\beta}h_{n}(\beta ' , \beta ' ) = 2 \nabla G_{n} (\beta ')$, so that
\[
h_{n} (\beta , \beta ' ) = G_{n}(\beta ') + 2 \left\langle \nabla G_{n}(\beta ') , \beta - \beta ' \right\rangle_{F} +  \left(  \varphi(\beta) - \varphi(\beta ') \right) ^{T} L_{n}(\beta ') \left(  \varphi(\beta) - \varphi(\beta ') \right)  
\]
Then, taking $\beta = T_{n}(\beta ')$, it comes
\[
h_{n} (T_{n}(\beta ') , \beta ' ) = G_{n}(\beta ') + 2 \left\langle \nabla G_{n}(\beta ') , T_{n}(\beta ') - \beta ' \right\rangle_{F} +  \left(  \varphi(T_{n}(\beta')) - \varphi(\beta ') \right) ^{T} L_{n}(\beta ') \left(  \varphi( T_{n}(\beta ')) - \varphi(\beta ') \right) .
\]
In addition, since  for all $a \in \mathbb{R}$ and $b > 0$ one has $\frac{a^{2}}{b} \geq 2a - b$, it comes that $h_{n} (\beta , \beta ') \geq 2 G_{n}(\beta) - G_{n}(\beta ')$. Then
\[
2 G_{n} \left( T_{n} (\beta ') \right) \leq 2 G_{n} ( \beta ') + 2 \left\langle \nabla G_{n}(\beta ') , T_{n}(\beta ') - \beta ' \right\rangle_{F} +  \left(  \varphi(T_{n}(\beta')) - \varphi(\beta ') \right) ^{T} L_{n}(\beta ') \left(  \varphi( T_{n}(\beta ')) - \varphi(\beta ') \right) .
\]
\end{proof}

Let us now prove the following Lemma (which can be seen as a translation of Lemma 5.2 in \cite{beck2015weiszfeld}):
\begin{lemma}\label{lem::2::bs}
For all $\beta \in \mathcal{M}_{q,p}(\mathbb{R})$,
\begin{align*}
G_{n} \left( \beta_{t+1} \right) & \leq G_{n} (\beta) + \frac{1}{2}\left( \varphi \left( \beta_{t} \right) - \varphi (\beta) \right)^{T} L_{n} \left( \beta_{t} \right)\left( \varphi \left( \beta_{t} \right) - \varphi (\beta) \right)\\
& - \frac{1}{2}\left( \varphi \left( \beta_{t+1} \right) - \varphi (\beta) \right)^{T} L_{n} \left( \beta_{t} \right)\left( \varphi \left( \beta_{t+1} \right) - \varphi (\beta) \right).
\end{align*}
\end{lemma}  
 
\begin{proof}[of Lemma \ref{lem::2::bs}]
With the help of Lemma \ref{lem::1::bs}, one has taking $\beta ' = \beta_{t}$,
\[
  G_{n} \left(\beta_{t+1} \right) \leq  G_{n} \left( \beta_{t} \right) +   \left\langle \nabla G_{n}\left(\beta_{t}\right) , \beta_{t+1} - \beta_{t} \right\rangle_{F} + \frac{1}{2}  \left(  \varphi\left( \beta_{t+1}\right) - \varphi\left( \beta_{t} \right)\right)^{T} L_{n}\left(\beta_{t} \right)  \left(  \varphi\left( \beta_{t+1}\right) - \varphi\left( \beta_{t} \right)\right) .
\]
In addition, since $G_{n}$ is convex, one has for any $\beta \in \mathcal{M}_{q,p}(\mathbb{R})$, $G_{n} \left( \beta_{t} \right) \leq G_{n}(\beta) + \left\langle \nabla G_{n}\left( \beta_{t} \right) ,\beta_{t} - \beta \right\rangle_{F}$, leading to
\begin{align*}
G_{n} \left( \beta_{t+1} \right) &  \leq G_{n}(\beta) + \left\langle \nabla G_{n}\left( \beta_{t} \right) , \beta_{t} - \beta \right\rangle_{F} +  \left\langle \nabla G_{n}\left(\beta_{t}\right) , \beta_{t+1} - \beta_{t} \right\rangle_{F} \\
&  + \frac{1}{2}  \left(  \varphi\left( \beta_{t+1}\right) - \varphi\left( \beta_{t} \right)\right)^{T} L_{n}\left(\beta_{t} \right)  \left(  \varphi\left( \beta_{t+1}\right) - \varphi\left( \beta_{t} \right)\right) \\
& =  G_{n}(\beta) + \left\langle \nabla G_{n}\left( \beta_{t} \right) , \beta_{t+1} - \beta \right\rangle_{F} + \frac{1}{2}  \left(  \varphi\left( \beta_{t+1}\right) - \varphi\left( \beta_{t} \right)\right)^{T} L_{n}\left(\beta_{t} \right)  \left(  \varphi\left( \beta_{t+1}\right) - \varphi\left( \beta_{t} \right)\right) .
\end{align*}
In addition, remark that equality \eqref{def::fix::gradient} can be written as
\[
\nabla G_{n}\left( \beta_{t} \right) = -  \left( \beta_{t+1} - \beta_{t} \right) \frac{1}{n}\sum_{i=1}^{n} \frac{X_{i}X_{i}^{T}}{\left\| Y_{i} - \beta_{t} X_{i} \right\|} 
\]
which can be written, in the basis of $\mathbb{R}^{pq}$ as
\[
   \nabla \tilde{G}_{n}\left(  \varphi\left(\beta_{t} \right) \right)  =  - L_{n}\left( \beta_{t} \right)\left( \varphi\left(  \beta_{t+1} \right) - \varphi \left( \beta_{t} \right) \right)
\]
Then, following \cite{beck2015weiszfeld},   
\begin{align*}
-  \langle \nabla G_{n}\left( \beta_{t} \right)& , \beta_{t+1} - \beta  \rangle_{F}   = \left( \varphi \left( \beta_{t+1} \right) - \varphi \left( \beta_{t} \right) \right)^{T} L_{n} \left( \beta_{t} \right)  \left( \varphi \left( \beta_{t+1} \right) - \varphi \left( \beta  \right) \right) \\
& =  \frac{1}{2} \left( \varphi \left( \beta_{t+1} \right) - \varphi \left( \beta \right) \right)^{T}L_{n}\left( \beta_{t} \right) \left( \varphi \left( \beta_{t+1} \right) - \varphi \left( \beta \right) \right) +  \frac{1}{2} \left( \varphi \left( \beta_{t+1} \right) - \varphi \left( \beta_{t } \right) \right)^{T} L_{n} \left( \beta_{t} \right)\left( \varphi \left( \beta_{t+1} \right) - \varphi \left( \beta_{t} \right) \right) \\
& - \frac{1}{2}\left( \varphi \left( \beta_{t} \right) - \varphi \left( \beta  \right) \right)^{T} L_{n} \left( \beta_{t} \right) \left( \varphi \left( \beta_{t} \right) - \varphi \left( \beta  \right) \right) .
\end{align*}
Then, for any $\beta \in \mathcal{M}_{q,p} (\mathbb{R})$,
\begin{align*}
G_{n} \left( \beta_{t+1} \right) & \leq G_{n} (\beta) + \frac{1}{2}\left( \varphi \left( \beta_{t} \right) - \varphi (\beta) \right)^{T} L_{n} \left( \beta_{t} \right)\left( \varphi \left( \beta_{t} \right) - \varphi (\beta) \right) - \frac{1}{2}\left( \varphi \left( \beta_{t+1} \right) - \varphi (\beta) \right)^{T} L_{n} \left( \beta_{t} \right)\left( \varphi \left( \beta_{t+1} \right) - \varphi (\beta) \right).
\end{align*}
\end{proof} 

We can now prove Theorem \ref{theo::fix}.

\begin{proof}[Proof of Theorem \ref{theo::fix}]
First, with the help of Lemma \ref{lem::2::bs}, it comes taking $\beta = \beta_{t}$
\begin{align}\label{ineqquimesauvelavie}
G_{n} \left( \beta_{t+1} \right) & \leq G_{n} \left(\beta_{t}\right)   - \frac{1}{2}\left( \varphi \left( \beta_{t+1} \right) - \varphi \left(\beta_{t}\right) \right)^{T} L_{n} \left( \beta_{t} \right)\left( \varphi \left( \beta_{t+1} \right) - \varphi \left(\beta_{t}\right) \right).
\end{align}
and since $L_{n} \left( \beta_{t} \right)$ is positive, it comes that the sequence $\left( G_{n} \left( \beta_{t} \right) \right)_{t \geq 0}$ is decreasing, and since $G_{n}$ is lower bounded by $0$, it means that it converges to a non negative constant $G_{\infty}$.{Now, since} the functional $G_{n}$ is strictly convex (as $\frac{1}{n}\sum_{i=1}^{n} X_{i}X_{i}^{T}$ is assumed to be positive), it means that there are a positive constant $r_{0}$ and a rank $t_{0}$ such that for all $t \geq t_{0}$, $\left\|  \beta_{t} \right\|_{F} \leq r_{0}$,  and in a particular case, the sequence $\left( \beta_{t} \right)_{t}$ is bounded, and let us denote by $\mathcal{K}$ a compact containing this sequence. Observe that on this compact, $L_{n}(.)$ is uniformly lower bounded by a positive matrix $L_{\mathcal{K}}$.  Let us now consider a converging subsequence $\beta_{u_{t}}$ of $\left( \beta_{ {t}} \right)$, and let us denote by $\beta_{u}$ its limit.  Then, since $G \left( \beta_{t} \right)$ converges to $G_{\infty}$, inequality \eqref{ineqquimesauvelavie} can be written as
\begin{align*}
0 & \frac12 \lambda_{\min} \left( L_{\mathcal{K}} \right) \left\| T_{n} \left( \beta_{u_{t}} \right) - \beta_{u_{t}} \right\|_{F}^{2}  \\
&  \leq \frac{1}{2} \left( \varphi \left( T_{n} \left( \beta_{u_{t}} \right) \right) - \varphi \left(\beta_{u_{t}}\right) \right)^{T} L_{n} \left( \beta_{t} \right) \left( \varphi \left( T_{n} \left( \beta_{u_{t}} \right) \right) - \varphi \left(\beta_{u_{t}}\right) \right) \leq G_{n}\left( \beta_{u_{t}} \right) - G_{n} \left( \beta_{u_{t+1}} \right) \xrightarrow[t\to + \infty]{} 0  
\end{align*}
Then, since $\lambda_{\min}\left( L_{\mathcal{K}} \right) > 0$, $\left\| T_{n} \left( \beta_{u_{t}} \right) - \beta_{u_{t}} \right\|$ converges to $0$, 
and $\beta_{u}$ is so the fixed point of $T_{n}$. In addition, 
\[
\lim_{t\to\infty}G_{n} \left( \beta_{t} \right) =\lim_{t\to\infty}G_{n} \left( \beta_{u_{t}} \right)  = G_{n} \left( \hat{\beta}_{n} \right)
\] 
where $ \hat{\beta}_{n}$ is the minimizer of $G_{n}$. Then, thanks to Lemma \ref{useful::lemma}, $G_{n}$ is locally strongly convex, which implies that $\beta_{t}$ converges to $\hat{\beta}_{n}$.
\end{proof}

\subsection{Proof of Theorem \ref{theo::fix::Sigma}}

In all the sequel, for the sake of simplicity, we denote $\beta_{t,\Sigma}:= \beta_{n,t,\Sigma}$ for all $t \geq 0$. The aim is to adapt the proof of Theorem \ref{theo::fix} in this context
First, observe that Lemma \ref{useful::lemma} can be directly derivated. 
We now prove the analogue of Lemma \ref{lem::1::bs} and \ref{lem::2::bs}. 
\begin{lemma}\label{lem::1::bs::Sigma}
For all $  \beta ' \in \mathcal{M}_{q,p}(\mathbb{R})$,
\begin{align*}
2 G_{n,\Sigma} \left( T_{n,\Sigma} (\beta ') \right) \leq & 2 G_{n,\Sigma} ( \beta ') + 2 \left\langle \nabla G_{n,\Sigma}(\beta ') , T_{n,\Sigma}(\beta ') - \beta ' \right\rangle_{F} \\
& +  \left(  \varphi(T_{n,\Sigma}(\beta')) - \varphi(\beta ') \right) ^{T} L_{n,\Sigma}(\beta ') \left(  \varphi( T_{n,\Sigma}(\beta ')) - \varphi(\beta ') \right) .
\end{align*}
\end{lemma}
\begin{proof}[of Lemma \ref{lem::1::bs::Sigma}]
Let us consider the functional $h_{n,\Sigma} : \mathcal{M}_{q,p}(\mathbb{R}) \times \mathcal{M}_{q,p}(\mathbb{R}) \longrightarrow \mathbb{R}_{+}$ defined for all $\beta , \beta ' $ by
\[
h_{n,\Sigma} (\beta , \beta ' ) = \frac{1}{n}\sum_{i=1}^{n} \frac{\left\| Y_{i} - \beta X_{i} \right\|_{\Sigma^{-1}}^{2}}{\left\| Y_{i} - \beta ' X_{i} \right\|_{\Sigma^{-1}}} .
\] 
Then, the functional $h_{n,\Sigma}( . , \beta ')$ is quadratic, and writing it in the canonical basis of $\mathbb{R}^{pq}$, the Hessian is given by $2L_{n,\Sigma}(\beta ')$ with $L_{n,\Sigma}(\beta ') := \Sigma^{-1} \otimes  \frac{1}{n}\sum_{i=1}^{n} \frac{X_{i}X_{i}^{T}}{\left\| Y_{i} - X_{i} \beta ' \right\|_{\Sigma^{-1}}}$. Then, with analogous developpement as in the proof of Lemma \ref{lem::1::bs}, for all $\beta , \beta ' \in \mathcal{M}_{q,p}(\mathbb{R})$,
\[
h_{n,\Sigma} (\beta , \beta ' ) = G_{n,\Sigma}(\beta ') + 2 \left\langle \nabla G_{n,\Sigma}(\beta ') , \beta - \beta ' \right\rangle_{F} +  \left(  \varphi(\beta) - \varphi(\beta ') \right) ^{T} L_{n,\Sigma}(\beta ') \left(  \varphi(\beta) - \varphi(\beta ') \right)  
\]
and taking $\beta = T_{n,\Sigma}(\beta ')$, it comes
\begin{align*}
2 G_{n,\Sigma} \left( T_{n,\Sigma} (\beta ') \right) \leq & 2 G_{n,\Sigma} ( \beta ') + 2 \left\langle \nabla G_{n,\Sigma}(\beta ') , T_{n,\Sigma}(\beta ') - \beta ' \right\rangle_{F} \\
& +  \left(  \varphi(T_{n,\Sigma}(\beta')) - \varphi(\beta ') \right) ^{T} L_{n,\Sigma}(\beta ') \left(  \varphi( T_{n,\Sigma}(\beta ')) - \varphi(\beta ') \right) .
\end{align*}
\end{proof}

We can now translate Lemma \ref{lem::2::bs} in this context.
\begin{lemma}\label{lem::2::bs::Sigma}
For any $\beta \in \mathcal{M}_{q,p}(\mathbb{R})$,
\begin{align*}
G_{n,\Sigma} \left( \beta_{t+1,\Sigma} \right) & \leq G_{n,\Sigma} (\beta) + \frac{1}{2}\left( \varphi \left( \beta_{t,\Sigma} \right) - \varphi (\beta) \right)^{T} L_{n,\Sigma} \left( \beta_{t,\Sigma} \right)\left( \varphi \left( \beta_{t,\Sigma} \right) - \varphi (\beta) \right) \\
&  - \frac{1}{2}\left( \varphi \left( \beta_{t+1,\Sigma} \right) - \varphi (\beta) \right)^{T} L_{n,\Sigma} \left( \beta_{t,\Sigma} \right)\left( \varphi \left( \beta_{t+1,\Sigma} \right) - \varphi (\beta) \right).
\end{align*}
\end{lemma}  
 
\begin{proof}[of Lemma \ref{lem::2::bs::Sigma}]
With the help of Lemma \ref{lem::1::bs} and as in the proof of Lemma \ref{lem::2::bs}, one has for all $\beta \in \mathcal{M}_{q,p}(\mathbb{R})$,

\begin{align*}
G_{n,\Sigma} \left( \beta_{t+1,\Sigma} \right) &   =  G_{n,\Sigma}(\beta) + \left\langle \nabla G_{n,\Sigma}\left( \beta_{t,\Sigma} \right) , \beta_{t+1,\Sigma} - \beta \right\rangle_{F}\\
& + \frac{1}{2}  \left(  \varphi\left( \beta_{t+1,\Sigma}\right) - \varphi\left( \beta_{t,\Sigma} \right)\right)^{T} L_{n,\Sigma}\left(\beta_{t,\Sigma} \right)  \left(  \varphi\left( \beta_{t+1,\Sigma}\right) - \varphi\left( \beta_{t,\Sigma} \right)\right) .
\end{align*}
In addition, remark that equality \eqref{def::fix::gradient::Sigma} can be written as
\[
\nabla G_{n,\Sigma}\left( \beta_{t,\Sigma} \right) = -  \Sigma^{-1} \left( \beta_{t+1,\Sigma} - \beta_{t,\Sigma} \right) \frac{1}{n}\sum_{i=1}^{n} \frac{X_{i}X_{i}^{T}}{\left\| Y_{i} - \beta_{t,\Sigma} X_{i} \right\|_{\Sigma^{-1}}} 
\]
which can be written, in the basis of $\mathbb{R}^{pq}$ as
\[
   \nabla \tilde{G}_{n,\Sigma}\left(  \varphi\left(\beta_{t,\Sigma} \right) \right)  =  - L_{n,\Sigma}\left( \beta_{t,\Sigma} \right)\left( \varphi\left(  \beta_{t+1,\Sigma} \right) - \varphi \left( \beta_{t,\Sigma} \right) \right)
\]
Then, for any $\beta \in \mathcal{M}_{q,p} (\mathbb{R})$,
\begin{align*}
G_{n,\Sigma} \left( \beta_{t+1,\Sigma} \right) & \leq G_{n,\Sigma} (\beta) + \frac{1}{2}\left( \varphi \left( \beta_{t,\Sigma} \right) - \varphi (\beta) \right)^{T} L_{n,\Sigma} \left( \beta_{t,\Sigma} \right)\left( \varphi \left( \beta_{t,\Sigma} \right) - \varphi (\beta) \right) \\
&  - \frac{1}{2}\left( \varphi \left( \beta_{t+1,\Sigma} \right) - \varphi (\beta) \right)^{T} L_{n,\Sigma} \left( \beta_{t,\Sigma} \right)\left( \varphi \left( \beta_{t+1,\Sigma} \right) - \varphi (\beta) \right).
\end{align*}
\end{proof} 
Thus, the proof of Theorem \ref{theo::fix::Sigma} is now the same as the one of Theorem \ref{theo::fix}.
 
\subsection{Proof of Theorem \ref{theo::sigma::lambda}}

\textbf{Writing in the basis of $\mathbb{R}^{pq}$: }
Vectorizing $\beta$,  the gradient in the basis of $\mathbb{R}^{pq} $ can be written as
\[
\nabla \tilde{G}_{\Sigma,\lambda} (\varphi (\beta)) =  -    \mathbb{E}\left[ \underbrace{   \frac{1}{\left\| Y -  \beta X \right\|_{\Sigma^{-1}}} \left( \Sigma^{-1}(Y -   \beta X) \right) \odot X  - \lambda\frac{\beta}{\| \beta\|_{F}}\mathbf{1}_{\beta \neq 0} }_{=:  \nabla \tilde{g}_{\Sigma,\lambda}(X,Y,\beta)}\right] .
\]
Hence, here again, the stochastic gradient
 and the one derived from $\nabla \tilde{G}_{\Sigma,\lambda}$ are equivalent. Indeed, considering the stochastic gradient algorithm defined recursively for all $n \geq 0$ by
\[
\beta_{n+1 , \Sigma,\lambda ,\varphi} = \beta_{n ,\Sigma, \lambda \varphi} + \gamma_{n+1} \left(  \nabla \tilde{g}_{\Sigma} \left( X_{n+1} , Y_{n+1} , \beta_{n,\Sigma,\lambda \varphi} \right) - \lambda \frac{ \beta_{n,\Sigma,\lambda \varphi}}{\left\|  \beta_{n,\Sigma,\lambda \varphi} \right\|_{F}} \mathbf{1}_{ \beta_{n,\Sigma,\lambda \varphi} \neq 0} \right)
\]
and denoting by $\overline{\beta}_{n, \Sigma,\lambda , \varphi}$ its weighted averaged version, one has $\beta_{n,\Sigma , \lambda,\varphi} = \varphi \left( \beta_{n,\Sigma,\lambda} \right)$ and $\overline{\beta}_{n, \Sigma , \lambda, \varphi} = \varphi \left( \overline{\beta}_{n, \Sigma, \lambda} \right)$. Be careful that here, the functional $G_{\Sigma, \lambda}$  is not twice continuously differentiable, and the Hessian, when it exists, is  not uniformly bounded. We so refer to \cite{pelletier1998almost} and \cite{Pel00} for asymptotic results.

\textbf{Bounding the gradient: } As soon as $X$ admits a moment of order $2a$ for some $a \geq 1$, one has for any $\beta \in \mathcal{M}_{q,p} (\mathbb{R})$
\[
\mathbb{E}\left[ \left\| \nabla \tilde{g}_{\Sigma ,\lambda} \left( X , Y , \varphi (\beta) \right) \right\|^{2a} \right] \leq 2^{2a-1}\frac{q}{\sqrt{\lambda_{\min} (\Sigma)}} \mathbb{E}\left[ \left\| X \right\|^{2a} \right] + 2^{2a-1} .
\]

\textbf{Bounding the Hessian on a  neighborhood of $\beta_{\lambda}^{*}$.}
The Hessian is given for all $\beta \neq 0$ by 
\begin{equation}\label{def::hess::lambda}
\nabla^{2} \tilde{G}_{\Sigma} (\varphi (\beta)) = \mathbb{E}\left[ \frac{1}{\left\| Y -  \beta X \right\|_{\Sigma^{-1}} }   \left( \Sigma^{-1} -  \frac{\Sigma^{-1}  ( Y -  \beta X)}{\left\| Y - \beta X \right\|_{\Sigma^{-1}}}\frac{ ( Y -  \beta X )^{T}\Sigma^{-1} }{\left\| Y -  \beta X \right\|_{\Sigma^{-1}}} \right) \otimes XX^{T}\right] + \frac{1}{\left\| \beta \right\|_{F}} \left( I_{pq} - \frac{\varphi (\beta) \varphi(\beta)^{T}}{\left\| \beta \right\|_{F}^{2}} \right) .
\end{equation}
Moreover, as in Theorem \ref{theo::without::sigma}, one has with the help of Lemma \ref{lem::2},
\begin{align*}
\left\| \nabla^{2} \tilde{G}_{\Sigma,\lambda} (\varphi (\beta)) \right\|_{F} & \leq   \mathbb{E}\left[ \mathbb{E}\left[ \frac{1}{\left\| Y - \beta X \right\|_{\Sigma^{-1}}} | X \right] \left\| \Sigma^{-1} \right\|_{F} \| X \|^{2} \right] + \frac{pq}{\| \beta \|_{F}} \\
& \leq \left( 1+ \frac{1}{q2^{q/2 -1}\Gamma (q/2)}\frac{1}{q-1} \right) \left\| \Sigma^{-1} \right\|_{F}  \mathbb{E}\left[ \| X\|^{2} \right] + \frac{pq}{\| \beta \|_{F}}
\end{align*}
and the Hessian is so continuous and bounded on a neighborhood of $\beta^{*} \neq 0$.
Observe that it is unfortunately not possible here to get a nice writing of $H_{\Sigma, \lambda}  = \nabla^{2}\tilde{G}_{\Sigma, \lambda} \left( \beta_{\lambda}^{*} \right)$.

\textbf{The Hessian is Lipschitz on a neighborhood of $\beta_{\lambda}^{*}$: } 
Let us now prove that the Hessian is locally Lipschitz. Let $\left\| \beta_{\lambda}^{*} \right\|_{F} > \epsilon > 0$ and  $V_{\epsilon}$   a neighborhood of $\beta_{\lambda}^{*}$ such that for all $\beta \in V_{\epsilon}$, $\| \beta\|_{F} \geq \epsilon /2$. Then, for all $\beta, \beta ' \in V_{\epsilon}$,
\begin{align*}
&  \left\| \nabla^{2} \tilde{G}_{\Sigma,\lambda} (\varphi (\beta)) - \nabla^{2} \tilde{G}_{\Sigma,\lambda} (\varphi (\beta')) \right\|_{F}    \leq \left\|   \nabla^{2} \tilde{G}_{\Sigma} (\varphi (\beta)) - \nabla^{2} \tilde{G}_{\Sigma} (\varphi (\beta')) \right\|_{F} + \left| \frac{1}{\left\| \beta \right\|_{F}} - \frac{1}{\left\| \beta ' \right\|_{F}} \right| + \left| \frac{ \beta }{\| \beta \|_{F}} - \frac{\beta ' }{\| \beta ' \|_{F}} \right| .
\end{align*}
Furthermore,
\[
\left| \frac{1}{\left\| \beta \right\|_{F}} - \frac{1}{\left\| \beta ' \right\|_{F}} \right| \leq \frac{\left\| \beta - \beta ' \right\|_{F}}{\left\| \beta \right\|_{F}\|\beta ' \|_{F}} \leq \frac{4}{\epsilon^{2}} \left\| \beta - \beta ' \right\|_{F}
\]
In a same way
\[
\left| \frac{ \beta }{\| \beta \|_{F}} - \frac{\beta ' }{\| \beta ' \|_{F}} \right| \leq \frac{\| \beta - \beta ' \|}{\| \beta ' \|} \leq \frac{\epsilon}{2} \| \beta - \beta ' \|_{F} .
\]
Then, the Hessian is $3 \left( 1+ \frac{1}{q2^{q/2  }\Gamma (q/2)}\frac{1}{q-2} \right) \left\| \Sigma^{-1/2} \right\|_{F}^{3} \mathbb{E}\left[ \| X \|^{3} \right]  + \frac{4}{\epsilon^{2}} + \frac{2}{\epsilon}$ Lispchitz on $V_{\epsilon}$.

\textbf{Convergence of the estimates: }
Assumptions in \cite{GB2017,godichon2016} are fulfilled and it comes
\begin{align*}
\left\| \beta_{n,\Sigma} - \beta^{*} \right\|_{F}^{2} & = \left\| \beta_{n,\varphi} - \varphi \left(\beta^{*} \right) \right\|^{2} = O \left( \frac{\ln n}{n^{\gamma}} \right) \quad a.s \\
\left\| \overline{\beta}_{n,\Sigma} - \beta^{*} \right\|_{F}^{2} & = \left\| \overline{\beta}_{n,\varphi} -\varphi \left(\beta^{*} \right) \right\|^{2} = O \left( \frac{\ln n}{n } \right) \quad a.s
\end{align*}
In addition, one has
\[
\sqrt{n} \left( \varphi \left( \overline{\beta}_{n,\Sigma,\lambda} \right) - \varphi \left( \beta_{\lambda}^{*} \right) \right) \xrightarrow[n\to + \infty]{\mathcal{L}} \mathcal{N} \left( 0 , H_{\Sigma,\lambda}^{-1}V_{\Sigma,\lambda}H_{\Sigma,\lambda}^{-1} \right)
\]
where $H_{\Sigma,\lambda}:=\nabla^{2}\tilde{G}_{\Sigma, \lambda} \left( \varphi \left( \beta_{\lambda}^{*} \right) \right) $ and $
V_{\Sigma,\lambda}    = \mathbb{E}\left[ \nabla \tilde{g}_{\Sigma,\lambda} \left( X , Y , \beta^{*} \right)\nabla \tilde{g}_{\Sigma,\lambda} \left( X , Y , \beta^{*} \right)^{T} \right]  $.

\subsection{Proof of Theorem \ref{theo::fix::Sigma::lambda}}
In all the sequel, for the sake of simplicity, we denote $\beta_{t,\Sigma,\lambda}:= \beta_{n,t,\Sigma,\lambda}$ for all $t \geq 0$.  The aim is to adapt the proof of Theorem \ref{theo::fix} in this context, i.e to translate and prove the   lemma. First, observe that Lemma \ref{useful::lemma} can be directly derivated. 
\begin{lemma}\label{lem::1::bs::Sigma::lambda}
For all $  \beta ' \in \mathcal{M}_{q,p}(\mathbb{R}) \backslash \lbrace 0\rbrace $,
\begin{align*}
2 G_{n,\Sigma,\lambda} \left( T_{n,\varphi,\Sigma,\lambda} (\beta ') \right) & \leq 2 G_{n,\Sigma,\lambda} ( \beta ') + 2 \left\langle \nabla \tilde{G}_{n,\Sigma,\lambda}(\beta ') ,T_{n,\varphi,\Sigma,\lambda}(\beta ') - \varphi( \beta ') \right\rangle   \\
& +  \left(   T_{n,\varphi,\Sigma,\lambda}(\beta')  - \varphi(\beta ') \right) ^{T} L_{n,\Sigma,\lambda}(\beta ') \left(    T_{n,\varphi,\Sigma,\lambda}(\beta ')) - \varphi(\beta ') \right) .
\end{align*}
with $L_{n,\Sigma,\lambda}(\beta ') := \Sigma^{-1} \otimes  \frac{1}{n}\sum_{i=1}^{n} \frac{X_{i}X_{i}^{T}}{\left\| Y_{i} - X_{i} \beta ' \right\|_{\Sigma^{-1}}} + \frac{\lambda}{\beta '} I_{pq}$.
\end{lemma}

\begin{proof}[of Lemma \ref{lem::1::bs::Sigma::lambda}]
Let us consider the functional $h_{n,\Sigma} : \mathcal{M}_{q,p}(\mathbb{R}) \times \mathcal{M}_{q,p}(\mathbb{R}) \longrightarrow \mathbb{R}_{+}$ defined for all $\beta , \beta ' $ with $\beta ' \neq 0$ by
\[
h_{n,\Sigma,\lambda} (\beta , \beta ' ) = \frac{1}{n}\sum_{i=1}^{n} \frac{\left\| Y_{i} - \beta X_{i} \right\|_{\Sigma^{-1}}^{2}}{\left\| Y_{i} - \beta ' X_{i} \right\|_{\Sigma^{-1}}} + \lambda \frac{\| \beta \|_{F}}{\| \beta ' \|_{F}} .
\] 
Then, the functional $h_{n,\Sigma,\lambda}( . , \beta ')$ is quadratic, and writting it in the canonical basis of $\mathbb{R}^{pq}$, the Hessian is given by $2L_{n,\Sigma,\lambda}(\beta ')$ with $L_{n,\Sigma,\lambda}(\beta ') := \Sigma^{-1} \otimes  \frac{1}{n}\sum_{i=1}^{n} \frac{X_{i}X_{i}^{T}}{\left\| Y_{i} - X_{i} \beta ' \right\|_{\Sigma^{-1}}} + \frac{\lambda}{\beta '} I_{pq}$. Then, with analogous development as in the proof of Lemma \ref{lem::1::bs}, for all $\beta , \beta ' \in \mathcal{M}_{q,p}(\mathbb{R})$ with $\beta ' \neq 0$,
\[
h_{n,\Sigma,\lambda} (\beta , \beta ' ) = G_{n,\Sigma,\lambda}(\beta ') + 2 \left\langle \nabla G_{n,\Sigma,\lambda}(\beta ') , \beta - \beta ' \right\rangle_{F} +  \left(  \varphi(\beta) - \varphi(\beta ') \right) ^{T} L_{n,\Sigma,\lambda}(\beta ') \left(  \varphi(\beta) - \varphi(\beta ') \right)  
\]
and taking $\beta =  T_{n,\varphi,\Sigma,\lambda}\left( \beta_{t,\Sigma,\lambda} \right)$, it comes
\begin{align*}
2 G_{n,\Sigma,\lambda} \left( T_{n,\varphi,\Sigma,\lambda} (\beta ') \right) & \leq 2 G_{n,\Sigma,\lambda} ( \beta ') + 2 \left\langle \nabla \tilde{G}_{n,\Sigma,\lambda}(\beta ') ,T_{n,\varphi,\Sigma,\lambda}(\beta ') - \varphi( \beta ') \right\rangle   \\
& +  \left(   T_{n,\varphi,\Sigma,\lambda}(\beta')  - \varphi(\beta ') \right) ^{T} L_{n,\Sigma,\lambda}(\beta ') \left(    T_{n,\varphi,\Sigma,\lambda}(\beta ')) - \varphi(\beta ') \right) .
\end{align*}
\end{proof}

We can now translate Lemma \ref{lem::2::bs} in this context.
\begin{lemma}\label{lem::2::bs::Sigma::lambda}
For any $\beta \in \mathcal{M}_{q,p}(\mathbb{R})$,
\begin{align*}
G_{n,\Sigma} \left( \beta_{t+1,\Sigma,\lambda} \right) & \leq G_{n,\Sigma,\lambda} (\beta) + \frac{1}{2}\left( \varphi \left( \beta_{t,\Sigma,\lambda} \right) - \varphi (\beta) \right)^{T} L_{n,\Sigma,\lambda} \left( \beta_{t,\Sigma,\lambda} \right)\left( \varphi \left( \beta_{t,\Sigma,\lambda} \right) - \varphi (\beta) \right) \\
&  - \frac{1}{2}\left( \varphi \left( \beta_{t+1,\Sigma,\lambda} \right) - \varphi (\beta) \right)^{T} L_{n,\Sigma,\lambda} \left( \beta_{t,\Sigma,\lambda} \right)\left( \varphi \left( \beta_{t+1,\Sigma,\lambda} \right) - \varphi (\beta) \right).
\end{align*}
\end{lemma}  
 
\begin{proof}[of Lemma \ref{lem::2::bs::Sigma::lambda}]
With the help of Lemma \ref{lem::1::bs} and as in the proof of Lemma \ref{lem::2::bs}, one has for all $\beta \in \mathcal{M}_{q,p}(\mathbb{R})$,
\begin{align*}
 G_{n,\Sigma} \left( \beta_{t+1,\Sigma} \right)    = & G_{n,\Sigma}(\beta) + \left\langle \nabla G_{n,\Sigma}\left( \beta_{t,\Sigma} \right) , \beta_{t+1,\Sigma} - \beta \right\rangle_{F} \\
& + \frac{1}{2}  \left(  \varphi\left( \beta_{t+1,\Sigma}\right) - \varphi\left( \beta_{t,\Sigma} \right)\right)^{T} L_{n,\Sigma}\left(\beta_{t,\Sigma} \right)  \left(  \varphi\left( \beta_{t+1,\Sigma}\right) - \varphi\left( \beta_{t,\Sigma} \right)\right) .
\end{align*}
In addition, remark that equality \eqref{def::gradient::Sigma::lambda} can be written as
\[
   \nabla \tilde{G}_{n,\Sigma,\lambda}\left(  \varphi\left(\beta_{t,\Sigma} \right) \right)  =    L_{n,\Sigma,\lambda}\left( \beta_{t,\Sigma,\lambda} \right)\left( \varphi\left(  \beta_{t+1,\Sigma,\lambda} \right) - \varphi \left( \beta_{t,\Sigma,\lambda} \right) \right)
\]
Then, for any $\beta \in \mathcal{M}_{q,p} (\mathbb{R})$,
\begin{align*}
G_{n,\Sigma,\lambda} \left( \beta_{t+1,\Sigma,\lambda} \right) & \leq G_{n,\Sigma,\lambda} (\beta) + \frac{1}{2}\left( \varphi \left( \beta_{t,\Sigma,\lambda} \right) - \varphi (\beta) \right)^{T} L_{n,\Sigma,\lambda} \left( \beta_{t,\Sigma,\lambda} \right)\left( \varphi \left( \beta_{t,\Sigma,\lambda} \right) - \varphi (\beta) \right) \\
&  - \frac{1}{2}\left( \varphi \left( \beta_{t+1,\Sigma,\lambda} \right) - \varphi (\beta) \right)^{T} L_{n,\Sigma,\lambda} \left( \beta_{t,\Sigma,\lambda} \right)\left( \varphi \left( \beta_{t+1,\Sigma,\lambda} \right) - \varphi (\beta) \right).
\end{align*}
\end{proof} 
Thus, as in the proof of Theorem \ref{theo::fix}, taking $\beta = \beta_{t,\Sigma,\lambda}$, it comes
\begin{align}
G_{n,\Sigma,\lambda} \left( \beta_{t+1,\Sigma,\lambda} \right)  \leq &G_{n,\Sigma,\lambda} \left( \beta_{t ,\Sigma,\lambda}  \right) \nonumber\\
& - \frac{1}{2}\left( \varphi \left( \beta_{t+1,\Sigma,\lambda} \right) - \varphi \left( \beta_{t ,\Sigma,\lambda}  \right) \right)^{T} L_{n,\Sigma,\lambda} \left( \beta_{t,\Sigma,\lambda} \right)\left( \varphi \left( \beta_{t+1,\Sigma,\lambda} \right) - \varphi \left( \beta_{t ,\Sigma,\lambda}  \right)\right)\label{eq::pleinlecul}
\end{align}
Then, $\left( G_{n,\Sigma,\lambda} \left( \beta_{t+1,\Sigma,\lambda} \right) \right)_{t}$ is a decreasing sequence and converges to a non negative constant $G_{\infty,\Sigma,\lambda}$. Here again, $G_{n,\Sigma,\lambda}$ is strictly convex and this means that $\beta_{t,\Sigma,\lambda}$ belongs to a compact $\mathcal{K}$. In addition, since $\epsilon $ is Gaussian, observe that for all $t$, $\beta_{t,\Sigma,\lambda} \neq 0 $ almost surely. Then, for any $t$, $L_{n,\Sigma,\lambda} \left( \beta_{t,\Sigma,\lambda} \right)$ is positive, and $\lambda_{\min} \left( L_{n,\Sigma,\lambda} \left( \beta_{t,\Sigma,\lambda} \right) \right) \geq \frac{\lambda}{\left\| \beta_{t} \right\|}$
 Then, denoting by $D_{\mathcal{K}} = \sup \left\lbrace \| \beta \| , \beta \in \mathcal{K} \right\rbrace$, it comes that for any $t$, $\lambda_{\min} \left( L_{n,\Sigma,\lambda} \left( \beta_{t,\Sigma,\lambda} \right) \right) \geq \frac{\lambda}{D_{\mathcal{K}}}$. Then, one can conclude as in the proof of Theorem \ref{theo::fix}.

\subsection{Two useful lemmas}


The following lemma ensures that the functional $G$ is twice continuously differentiable.
\begin{lemma}\label{lem::1}
Suppose $q \geq 3$, then for any $a< q$,  and  for all $\beta \in \mathcal{M}_{q,p}(\mathbb{R})$
\[
\mathbb{E}\left[ \frac{1}{\left\| Y -  \beta X \right\|^{a}} | X \right] \leq 1+ \frac{\lambda_{\max}^{q/2}}{q2^{q/2 -1}\Gamma (q/2)}\frac{a}{q-a} .
\] 
where $\lambda_{\max} = \lambda_{\max} \left( \Sigma \right)$.
\end{lemma}
In order to ensure that the functional $G_{\Sigma}$ is twice differentiable, let us translate Lemma \ref{lem::1}  as follows:
\begin{lemma}\label{lem::2}
Suppose $q \geq 3$, then for any $a< q$,  and  for all $\beta \in \mathcal{M}_{q,p}(\mathbb{R})$
\[
\mathbb{E}\left[ \frac{1}{\left\| Y -X \beta \right\|_{\Sigma^{-1}}^{a}} | X \right] \leq 1+ \frac{1}{q2^{q/2 -1}\Gamma (q/2)}\frac{a}{q-a} .
\] 
\end{lemma}
\begin{proof}[of Lemmas \ref{lem::1} and \ref{lem::2}]
For any $\beta \in \mathcal{M}_{q,p}(\mathbb{R})$, one has since $q \geq 3$ and $a< q$ and since $\epsilon\sim \mathcal{N}\left( 0 ,\Sigma \right)$,
\begin{align*}
\mathbb{E}\left[ \frac{1}{\left\| \epsilon -  \beta X + \beta^{*}  X\right\|_{\Sigma^{-1}}^{a}} | X \right] & = \int_{0}^{+\infty} \mathbb{P}\left[  \left\| \epsilon -  \beta X +  \beta^{*} X \right\|_{\Sigma^{-1}}^{a}  \leq \frac{1}{t} |X \right] dt \\
& \leq \int_{0}^{+\infty} \mathbb{P}\left[  \left\| \epsilon   \right\|_{\Sigma^{-1}}^{a}  \leq \frac{1}{t} |X \right] dt \\
& \leq \int_{0}^{+\infty} \mathbb{P}\left[  \chi_{q}^{2}  \leq \frac{1}{t^{2/a}} |X \right] dt \\
& \leq   1+ \frac{1}{\Gamma (q/2)}\int_{1}^{+\infty} \gamma \left( q/2,t^{-2/a} /2 \right) dt \end{align*}
where $\gamma (.,.)$ is the incomplete Gamma function. Then,

\begin{align*}
\mathbb{E}\left[ \frac{1}{\left\| \epsilon -  \beta X + \beta^{*}  X\right\|_{\Sigma^{-1}}^{a}} | X \right]& =  1+ \frac{1}{\Gamma (q/2)}\int_{1}^{+\infty} \int_{0}^{\frac{1}{2t^{2/a}}} u^{\frac{q}{2}-1}e^{-u}du \\
& \leq 1+ \frac{1}{\Gamma (q/2)}\int_{1}^{+\infty} \int_{0}^{\frac{1}{2t^{2/a}}} u^{\frac{q}{2}-1} du \\
& =  1+ \frac{1}{q2^{q/2 -1}\Gamma (q/2)}\int_{1}^{+\infty}  \frac{1}{t^{q/a}} dt \\
& \leq 1+ \frac{1}{2^{q/2 -1}\Gamma (q/2)}\frac{a}{q-a}
\end{align*}
which concludes the proof of Lemma \ref{lem::2}.
Analogously, since $\frac{1}{\lambda_{\min}} \left\| h \right\|_{\Sigma} \geq t  \| h\| \geq  \frac{1}{\lambda_{\max}} \left\| h \right\|_{\Sigma}$, 
\begin{align*}
\mathbb{E}\left[ \frac{1}{\left\| \epsilon -  \beta X +   \beta^{*} X \right\|^{a}} | X \right] & = \int_{0}^{+\infty} \mathbb{P}\left[  \left\| \epsilon -  \beta X +   \beta^{*} X \right\|^{a}  \leq \frac{1}{t} |X \right] dt \\
& \leq \int_{0}^{+\infty} \mathbb{P}\left[  \left\| \epsilon -  \beta X +   \beta^{*} X \right\|_{\Sigma^{-1}}^{a}  \leq \frac{\lambda_{\max}^{a}}{t} |X \right] dt \\
& \leq \int_{0}^{+\infty} \mathbb{P}\left[  \left\| \epsilon   \right\|_{\Sigma^{-1}}^{a}   \leq \frac{\lambda_{\max}^{a}}{t} |X \right] dt \\
& \leq   1+ \frac{1}{\Gamma (q/2)}\int_{1}^{+\infty} \gamma \left( q/2,\lambda_{\max}t^{-2/a} /2 \right) dt \\
& =  1+ \frac{1}{\Gamma (q/2)}\int_{1}^{+\infty} \int_{0}^{\frac{\lambda_{\max}}{2t^{2/a}}} u^{\frac{q}{2}-1}e^{-u}du 
\end{align*}
where $\gamma (.,.)$ is the incomplete Gamma function. Then,

\begin{align*}
\mathbb{E}\left[ \frac{1}{\left\| \epsilon +  \beta X -   \beta^{*} X \right\|^{a}} | X \right] & \leq 1+ \frac{1}{\Gamma (q/2)}\int_{1}^{+\infty} \int_{0}^{\frac{\lambda_{\max}}{t^{2/a}}} u^{\frac{q}{2}-1} du \\
& =  1+ \frac{\lambda_{\max}^{q/2}}{q2^{q/2 -1}\Gamma (q/2)}\int_{1}^{+\infty}  \frac{1}{t^{q/a}} dt \\
& \leq 1+ \frac{\lambda_{\max}^{q/2}}{q2^{q/2 -1}\Gamma (q/2)}\frac{a}{q-a}
\end{align*}
\end{proof}

\bibliographystyle{apalike}
\bibliography{references}

\appendix 

\section{Estimating the Median Covariation Matrix}

\subsection{Without estimating the Median}
Let us recall that the aim is to estimate 
\[
G_{MCM,N}(V)=\frac{1}{N}\sum_{i=1}^{n}\left\|  \left( Y_{i} - \overline{\beta}_{N} X_{i}   \right)\left( Y_{i} - \overline{\beta}_{N} X_{i}\right)^{T}  - V \right\|_{F},
\]
\textbf{Weiszfeld algorithm: } For all $t \geq 0$
\[
V_{t+1} = \frac{\sum_{i=1}^{N} \frac{\left( Y_{i} -  \overline{\beta}_{N} X_{i}  \right) \left( Y_{i} - \overline{\beta}_{N}X_{i}\right)^{T}}{\left\| \left( Y_{i} -  \overline{\beta}_{N}X_{i}  \right) \left( Y_{i} - \overline{\beta}_{N}X_{i}\right)^{T} - V_{t} \right\|_{F}}}{\sum_{i=1}^{N} \frac{1}{\left\| \left( Y_{i} -  \overline{\beta}_{N} X_{i} \right) \left( Y_{i} - \overline{\beta}_{N}X_{i}\right)^{T} - V_{t} \right\|_{F}} }
\]
\textbf{Averaged stochastic gradient algorithm :} For all $n \leq N-1$,
\begin{align*}
V_{n+1} &  = V_{n} + \gamma_{n+1} \frac{\left( Y_{n+1} -  \overline{\beta}_{N} X_{n+1}  \right) \left( Y_{n+1} - \overline{\beta}_{N} X_{n+1}\right)^{T} - V_{n}}{\left\| \left( Y_{n+1} -\overline{\beta}_{N}   X_{n+1}  \right) \left( Y_{n+1} - \overline{\beta}_{N}X_{n+1}\right)^{T} - V_{n} \right\|_{F} } \\
\overline{V}_{n+1} &  = \overline{V}_{n} + \frac{1}{n+1} \left( V_{n+1} - \overline{V}_{n} \right)
\end{align*} 

\subsection{With the estimates of the Median}
Observe that previous estimates consider that $\overline{\beta}_{N}$ is an unbiased estimate of $\beta^{*}$ in the sense that $\text{med} \left( Y - X \beta \overline{\beta}_{N} \right)$, which cannot be verified theoretically and which is probably wrong. Then, one can consider these following alternatives \citep{CG2015}.

\medskip

\noindent\textbf{Wesizfeld algorithms (estimating the median): } For all $t \geq 0$, 
\[
m_{t+1} = \frac{\sum_{i=1}^{N} \frac{ Y_{i} -  \overline{\beta}_{N}X_{i}  }{\left\| Y_{i} -  \overline{\beta}_{N}X_{i} - m_{t} \right\|}}{\sum_{i=1}^{N} \frac{1}{\left\| Y_{i} - \overline{\beta}_{N}X_{i}  - m_{t} \right\|}}
\]

\medskip

\noindent\textbf{Weiszfeld algorithm (estimating the MCM): } Considering the estimate of the median $m_{T}$ for some integer $T$, for all $t \geq 0$
\[
V_{t+1} = \frac{\sum_{i=1}^{N} \frac{\left( Y_{i} -  \overline{\beta}_{N}X_{i} - m_{T}  \right)^{T} \left( Y_{i} - \overline{\beta}_{N}X_{i}- m_{T} \right)}{\left\| \left( Y_{i} -  \overline{\beta}_{N}X_{i}  - m_{T} \right) \left( Y_{i} - \overline{\beta}_{N}X_{i}- m_{T} \right)^{T} - V_{t} \right\|_{F}}}{\sum_{i=1}^{N} \frac{1}{\left\| \left( Y_{i} - \overline{\beta}_{N} X_{i}  - m_{T} \right) \left( Y_{i} -\overline{\beta}_{N} X_{i}- m_{T} \right)^{T} - V_{t} \right\|_{F}} }
\]

\medskip

\noindent \textbf{Averaged stochastic gradient algorithms: } For all $n \leq N-1$,

\begin{align*}
m_{n+1} & = m_{n} + \gamma_{n+1} \frac{Y_{n+1} - \overline{\beta}_{N}X_{n+1} - m_{n}}{\left\| Y_{n+1} - \overline{\beta}_{N}X_{n+1} - m_{n} \right\|} \\
\overline{m}_{n+1} & = \overline{m}_{n} + \frac{1}{n+1} \left( m_{n+1} - \overline{m}_{n} \right) \\
V_{n+1} &  = V_{n} + \gamma_{n+1} \frac{\left( Y_{n+1} - \overline{\beta}_{N} X_{n+1}  \right) \left( Y_{n+1} - \overline{\beta}_{N}X_{n+1}\right)^{T} - V_{n}}{\left\| \left( Y_{n+1} - \overline{\beta}_{N}X_{n+1}   \right) \left( Y_{n+1} - \overline{\beta}_{N} X_{n+1}\right)^{T} - V_{n} \right\|_{F} } \\
\overline{V}_{n+1} &  = \overline{V}_{n} + \frac{1}{n+1} \left( V_{n+1} - \overline{V}_{n} \right)
\end{align*}

\FloatBarrier

\section{Additional simulation results} 

\subsection{Simulation parameters} \label{app::parms}


\paragraph{\sl Covariance matrix.} ~

The $q \times q$ covariance matrix $\Sigma$ was build as $\Sigma = u u^\top + \sigma^2 I_q$, where $u$ is a $q$-dimensional vector with iid $\Ncal(0, 1)$ coordinates and $\sigma = \sd$. This structure was chosen so to induce a fairly strong correlation between the responses, to highlight the difference between OLS and WLS estimates.

\subsection{Additional simulation figures}  \label{app::figures}

\paragraph{\sl Estimation accuracy.} ~

\begin{figure}[ht]
  \begin{center}
    \begin{tabular}{m{0.1\textwidth}m{0.7\textwidth}}
      WLS & \includegraphics[width=0.7\textwidth, trim=0 60 0 55, clip=]{MSEbeta-n\nn-fake-p5-q20-\parmReg-WLS} \\
      WLS+Ridge & \includegraphics[width=0.7\textwidth, trim=0 60 0 55, clip=]{MSEbeta-n\nn-fake-p5-q20-\parmReg-WLS+Ridge} 
    \end{tabular}
  \end{center}
  \caption{Mean squared error for the estimation of the regression coefficients $\beta$ for WLS estimates without or with Ridge regularization. Legend:  \textcolor{black}{$\square$} = Naive, \textcolor{red}{$\medcirc$} = Offline, \textcolor{green}{$\triangle$} = Online. Top of each panel = fraction $\prop$ of outliers (in \%). 
  Same legend as Figure~\ref{fig::mseBeta-n\nn}.
  \label{fig::mseBeta-n\nn-WLS}}
\end{figure}

\begin{figure}[ht]
  \begin{center}
    \begin{tabular}{m{0.1\textwidth}m{0.7\textwidth}}
      OLS & \includegraphics[width=0.7\textwidth, trim=0 60 0 55, clip=]{MSEsigma-n\nn-fake-p5-q20-\parmReg-noRidge} \\
      OLS+Ridge & \includegraphics[width=0.7\textwidth, trim=0 60 0 55, clip=]{MSEsigma-n\nn-fake-p5-q20-\parmReg-Ridge} 
    \end{tabular}
  \end{center}
  \caption{
  Mean squared error for the estimation of the covariance matrix $\Sigma$ ($MSE(\Sigma  )$) without or with Ridge regularization. 
  The $y$-axis is log-scaled. 
  Same legend as Figure~\ref{fig::mseBeta-n\nn}.
\label{fig::mseSigma-n\nn}} 
\end{figure}

\FloatBarrier


\FloatBarrier

\paragraph{\sl Effect of the sample size.} ~

\begin{figure}[ht]
  \begin{center}
    \begin{tabular}{m{0.1\textwidth}m{0.35\textwidth}m{0.35\textwidth}}
      & \multicolumn{1}{c}{$\prop = 5\%$} & \multicolumn{1}{c}{$\prop = 28\%$} \\
      WLS & 
      \includegraphics[width=0.35\textwidth, trim=0 60 0 55, clip=]{MSEbeta-prop5-fake-p5-q20-\parmReg-WLS} &
      \includegraphics[width=0.35\textwidth, trim=0 60 0 55, clip=]{MSEbeta-prop28-fake-p5-q20-\parmReg-WLS} \\
      WLS+Ridge & 
      \includegraphics[width=0.35\textwidth, trim=0 60 0 55, clip=]{MSEbeta-prop5-fake-p5-q20-\parmReg-WLS+Ridge} & 
      \includegraphics[width=0.35\textwidth, trim=0 60 0 55, clip=]{MSEbeta-prop28-fake-p5-q20-\parmReg-WLS+Ridge} 
    \end{tabular}  
  \end{center}
  \caption{
  Mean squared error for the estimation of the regression coefficients $\beta$ ($MSE(\beta)$) for WLS estimates with or without Ridge regularization for a fraction of outliers of $\prop = 5\%$ (left) and $\prop = 28\%$ (right). 
  The $y$-axis is log-scaled.
  Same legend as Figure \ref{fig::mseBeta-prop28-OLS}.
  \label{fig::mseBeta-prop28-WLS}} 
\end{figure}

\FloatBarrier

\paragraph{\sl Prediction accuracy (OLS vs WLS).} ~

\begin{figure}[ht]
  \begin{center}
    \begin{tabular}{m{0.1\textwidth}m{0.7\textwidth}}
      OLS & \includegraphics[width=0.7\textwidth, trim=0 60 0 55, clip=]{PredIn-n\nn-fake-p5-q20-\parmReg-OLS} \\
      OLS+Ridge & \includegraphics[width=0.7\textwidth, trim=0 60 0 55, clip=]{PredIn-n\nn-fake-p5-q20-\parmReg-OLS+Ridge} \\
      WLS  & \includegraphics[width=0.7\textwidth, trim=0 60 0 55, clip=]{PredIn-n\nn-fake-p5-q20-\parmReg-WLS} \\
      WLS+Ridge & \includegraphics[width=0.7\textwidth, trim=0 60 0 55, clip=]{PredIn-n\nn-fake-p5-q20-\parmReg-WLS+Ridge} 
    \end{tabular}
  \end{center}
  \caption{
  Mean squared error for the predicted value $\widehat{\beta} X$ as compared to the true response $Y$ ($MSE(Y)$) for the OLS and WLS estimates with or without Ridge regularization for a sample size of $n = \nn$. 
  Same legend as Figure \ref{fig::mseBeta-n\nn}.  \label{fig::predIn-n\nn}}
\end{figure}

\FloatBarrier

\paragraph{\sl Outlier detection.} ~

\begin{figure}[ht]
  \begin{center}
    \begin{tabular}{m{0.1\textwidth}m{0.7\textwidth}}
      WLS  & \includegraphics[width=0.7\textwidth, trim=0 60 0 55, clip=]{AUC-n\nn-fake-p5-q20-\parmReg-WLS} \\
      WLS+Ridge & \includegraphics[width=0.7\textwidth, trim=0 60 0 55, clip=]{AUC-n\nn-fake-p5-q20-\parmReg-WLS+Ridge} 
    \end{tabular}
  \end{center}
  \caption{
  AUC for the detection of outliers for the WLS estimates with or without Ridge regularization for a sample size of $n = \nn$. 
  Same legend as Figure \ref{fig::mseBeta-n\nn}.  \label{fig::auc-n\nn-WLS}}
\end{figure}

\FloatBarrier

\paragraph{\sl OLS versus WLS.} ~

\begin{figure}[ht]
  \begin{center}
    \begin{tabular}{m{0.12\textwidth}m{0.7\textwidth}}
      Offline+Ridge & \includegraphics[width=0.7\textwidth, trim=0 60 0 55, clip=]{MSEbetacompOW-Ridge-n\nn-fake-p5-q20-\parmReg-Off} \\
      Online+Ridge & \includegraphics[width=0.7\textwidth, trim=0 60 0 55, clip=]{MSEbetacompOW-Ridge-n\nn-fake-p5-q20-\parmReg-On} \\
    \end{tabular}
  \end{center}
  \caption{
  Comparison of the mean squared error for the estimation of the regression coefficients $\beta$ ($MSE(\beta)$) of OLS and WLS estimates for Offlin and Online procedures with Ridge regularization for a sample size of $n = \nn$.
  Legend:  \textcolor{red}{$\medcirc$} = Offline, \textcolor{green}{$\triangle$} = Online. Number in each box = fraction $\prop$ of outliers (in \%).     \label{fig::compOW-n\nn-Ridge}}
\end{figure}

\FloatBarrier

\subsection{\SR{}{Simulations in large dimension}} \label{app::large}

\SR{}{We present here the equivalent of Figures \ref{fig::mseBeta-n\nn}, \ref{fig::mseBeta-prop28-OLS}, \ref{fig::auc-n\nn-OLS} and \ref{fig::compOW-n\nn-noRidge} in a large dimension setting with $n = \nLarge$ observations, $p=\pLarge$ covariates and $q=\qLarge$ responses.}

\paragraph{\SR{}{\sl Estimation accuracy.}} ~

\begin{figure}[ht]
  \begin{center}
    \begin{tabular}{m{0.1\textwidth}m{0.7\textwidth}}
      OLS & \includegraphics[width=0.7\textwidth, trim=0 60 0 55, clip=]{MSEbeta-\dimLarge-\parmLarge-OLS} \\
      OLS+Ridge & \includegraphics[width=0.7\textwidth, trim=0 60 0 55, clip=]{MSEbeta-\dimLarge-\parmLarge-OLS+Ridge} 
    \end{tabular}
  \end{center}
  \caption{\SR{}{
  Mean squared error for the estimation of the regression coefficients $\beta$ ($MSE(\beta)$) for OLS estimates without or with Ridge regularization for a sample size of $n = \nLarge$, with $p=\pLarge$ covariates and $q=\qLarge$ responses (with Student outliers). 
  Same legend as Figure \ref{fig::mseBeta-n\nn}.
  \label{fig::mseBeta-n\nLarge-large}}}
\end{figure}

\FloatBarrier

\paragraph{\SR{}{\sl Effect of the number of responses.}} ~

\begin{figure}[ht]
  \begin{center}
    \begin{tabular}{m{0.1\textwidth}m{0.35\textwidth}m{0.35\textwidth}}
      & \multicolumn{1}{c}{$\prop = 5\%$} & \multicolumn{1}{c}{$\prop = 28\%$} \\
      OLS & 
      \includegraphics[width=0.35\textwidth, trim=0 60 0 55, clip=]{MSEbeta-prop5-\parmLarge-OLS} & 
      \includegraphics[width=0.35\textwidth, trim=0 60 0 55, clip=]{MSEbeta-prop28-\parmLarge-OLS} \\
      OLS+Ridge & 
      \includegraphics[width=0.35\textwidth, trim=0 60 0 55, clip=]{MSEbeta-prop5-\parmLarge-OLS+Ridge} & 
      \includegraphics[width=0.35\textwidth, trim=0 60 0 55, clip=]{MSEbeta-prop28-\parmLarge-OLS+Ridge} 
    \end{tabular}
  \end{center}
  \caption{\SR{}{
  Mean squared error for the estimation of the regression coefficients $\beta$ ($MSE(\beta)$) for OLS estimates with or without Ridge regularization for a fraction of outliers of $\prop = 5\%$ (left) and $\prop = 28\%$ (right) for $n=\nLarge$ observations a,d $p=\pLarge$ covariates (with Student outliers). 
  Number in each box = number of responses $q$.  
  Same legend as Figure \ref{fig::mseBeta-prop28-OLS}.
  \label{fig::mseBeta-prop28-OLS-large}}}
\end{figure}

\FloatBarrier

%

\paragraph{\SR{}{\sl Outlier detection.}} ~

\begin{figure}[ht]
  \begin{center}
    \begin{tabular}{m{0.1\textwidth}m{0.7\textwidth}}
      OLS & \includegraphics[width=0.7\textwidth, trim=0 60 0 55, clip=]{AUC-\dimLarge-\parmLarge-OLS} \\
      OLS+Ridge & \includegraphics[width=0.7\textwidth, trim=0 60 0 55, clip=]{AUC-\dimLarge-\parmLarge-OLS+Ridge} 
    \end{tabular}
  \end{center}
  \caption{\SR{}{
  AUC for the detection of outliers for the OLS estimates with or without Ridge regularization for a sample size of $n = \nLarge$ (with Student outliers). 
  Same legend as Figure \ref{fig::mseBeta-n\nn}. 
  \label{fig::auc-n\nLarge-OLS-large}}}
\end{figure}

\FloatBarrier

\paragraph{\SR{}{\sl OLS versus WLS.}} ~

\begin{figure}[ht]
  \begin{center}
    \begin{tabular}{m{0.12\textwidth}m{0.7\textwidth}}
      Offline & \includegraphics[width=0.7\textwidth, trim=0 60 0 55, clip=]{MSEbetacompOW-NoRidge-\dimLarge-\parmLarge-Off} \\
      Online & \includegraphics[width=0.7\textwidth, trim=0 60 0 55, clip=]{MSEbetacompOW-NoRidge-\dimLarge-\parmLarge-On} 
    \end{tabular}
  \end{center}
  \caption{\SR{}{
  Comparison of the mean squared error for the estimation of the regression coefficients $\beta$ ($MSE(\beta)$) of OLS and WLS estimates for Offline and Online procedures without Ridge regularization for a sample size of $n = \nLarge$, with $p=\pLarge$ covariates and $q=\qLarge$ responses (with Student outliers).
  Same legend as in Figure \ref{fig::compOW-n\nn-noRidge}.
  \label{fig::compOW-n\nLarge-noRidge-large}}}
\end{figure}

\FloatBarrier

\subsection{\SR{}{Simulations with Dirac outliers}} \label{app::dirac}

\paragraph{\SR{}{\sl Estimation accuracy.}} ~

\begin{figure}[ht]
  \begin{center}
    \begin{tabular}{m{0.1\textwidth}m{0.7\textwidth}}
      OLS & \includegraphics[width=0.7\textwidth, trim=0 60 0 55, clip=]{MSEbeta-n\nDirac-dirac-p5-q20-\parmDirac-OLS} \\
      OLS+Ridge & \includegraphics[width=0.7\textwidth, trim=0 60 0 55, clip=]{MSEbeta-n\nDirac-dirac-p5-q20-\parmDirac-OLS+Ridge} 
    \end{tabular}
  \end{center}
  \caption{\SR{}{
  Mean squared error for the estimation of the regression coefficients $\beta$ ($MSE(\beta)$) for OLS estimates without or with Ridge regularization for a sample size of $n = \nDirac$ (with Dirac outliers). 
  The $y$-axis is log-scaled. 
  Same legend as Figure \ref{fig::mseBeta-n\nn}.  
  \label{fig::mseBeta-n\nDirac-dirac}}} 
\end{figure}

\FloatBarrier

\paragraph{\SR{}{\sl Effect of the sample size.}}~

\begin{figure}[ht]
  \begin{center}
    \begin{tabular}{m{0.1\textwidth}m{0.35\textwidth}m{0.35\textwidth}}
      & \multicolumn{1}{c}{$\prop = 5\%$} & \multicolumn{1}{c}{$\prop = 28\%$} \\
      OLS & 
      \includegraphics[width=0.35\textwidth, trim=0 60 0 55, clip=]{MSEbeta-prop5-dirac-p5-q20-\parmDirac-OLS} & 
      \includegraphics[width=0.35\textwidth, trim=0 60 0 55, clip=]{MSEbeta-prop28-dirac-p5-q20-\parmDirac-OLS} \\
      OLS+Ridge & 
      \includegraphics[width=0.35\textwidth, trim=0 60 0 55, clip=]{MSEbeta-prop5-dirac-p5-q20-\parmDirac-OLS+Ridge} & 
      \includegraphics[width=0.35\textwidth, trim=0 60 0 55, clip=]{MSEbeta-prop28-dirac-p5-q20-\parmDirac-OLS+Ridge} 
    \end{tabular}
  \end{center}
  \caption{\SR{}{
  Mean squared error for the estimation of the regression coefficients $\beta$ ($MSE(\beta)$) for OLS estimates with or without Ridge regularization for a fraction of outliers of $\prop = 5\%$ (left) and $\prop = 28\%$ (right) (with Dirac outliers). 
  The $y$-axis is log-scaled. 
  Same legend as Figure \ref{fig::mseBeta-n\nn}.  
  \label{fig::mseBeta-prop28-OLS-dirac}}} 
\end{figure}

\paragraph{\SR{}{\sl Prediction accuracy.}}
\begin{figure}[ht]
  \begin{center}
    \begin{tabular}{m{0.1\textwidth}m{0.7\textwidth}}
      OLS & \includegraphics[width=0.7\textwidth, trim=0 60 0 55, clip=]{PredIn-n\nDirac-dirac-p5-q20-\parmDirac-OLS} \\
      OLS+Ridge & \includegraphics[width=0.7\textwidth, trim=0 60 0 55, clip=]{PredIn-n\nDirac-dirac-p5-q20-\parmDirac-OLS+Ridge} \\
      WLS  & \includegraphics[width=0.7\textwidth, trim=0 60 0 55, clip=]{PredIn-n\nDirac-dirac-p5-q20-\parmDirac-WLS} \\
      WLS+Ridge & \includegraphics[width=0.7\textwidth, trim=0 60 0 55, clip=]{PredIn-n\nDirac-dirac-p5-q20-\parmDirac-WLS+Ridge} 
    \end{tabular}
  \end{center}
  \caption{\SR{}{
  Mean squared error for the predicted value $\widehat{\beta} X$ as compared to the true response $Y$ ($MSE(Y)$) for the OLS and WLS estimates with or without Ridge regularization for a sample size of $n = \nDirac$ (with Dirac outliers). 
  Same legend as Figure \ref{fig::mseBeta-n\nn}.  
  \label{fig::predIn-n\nDirac-dirac}}}
\end{figure}

%
%

\FloatBarrier

\subsection{\SR{}{Additional simulation tables}} \label{app::tables}

\newcommand{\tablePath}{TablesV2}

\SR{}{We provide here the numerical results corresponding to Figures \ref{fig::mseBeta-n\nn}, \ref{fig::mseBeta-prop28-OLS}, \ref{fig::auc-n\nn-OLS} and \ref{fig::compOW-n\nn-noRidge} in Tables \ref{tab::mseBeta-n\nn}, \ref{tab::mseBeta-prop28-OLS}, \ref{tab::auc-n\nn-OLS} and \ref{tab::compOW-n\nn-noRidge}, respectively.}

\paragraph{\SR{}{\sl Estimation accuracy.}} ~

\begin{table}[ht]
  \begin{center}
    OLS \\
    \begin{tabular}{r|rrrrrrrr} 
& \multicolumn{2}{c}{ Naive } & \multicolumn{2}{c}{ Offline } & \multicolumn{2}{c}{ Online } & \multicolumn{2}{c}{ RRRR } \\ 
\hline 
0  &0.0018  & ( 0.00058 ) & 0.0019  & ( 0.00066 ) & 0.0067  & ( 0.0033 ) & 0.002  & ( 0.00068 ) \\ 
2  &0.14  & ( 0.6 ) & 0.0019  & ( 0.00053 ) & 0.0068  & ( 0.0034 ) & 0.0019  & ( 0.00056 ) \\ 
3  &0.06  & ( 0.38 ) & 0.0019  & ( 0.00055 ) & 0.0077  & ( 0.008 ) & 0.002  & ( 0.00057 ) \\ 
5  &0.16  & ( 0.81 ) & 0.0019  & ( 0.00053 ) & 0.0073  & ( 0.0045 ) & 0.0019  & ( 0.00052 ) \\ 
9  &14  & ( 140 ) & 0.0019  & ( 0.00052 ) & 0.0071  & ( 0.0036 ) & 0.0019  & ( 0.00054 ) \\ 
16  &1.6  & ( 7.9 ) & 0.0018  & ( 0.00051 ) & 0.0074  & ( 0.0041 ) & 0.0018  & ( 5e-04 ) \\ 
28  &5.2  & ( 30 ) & 0.0018  & ( 0.00046 ) & 0.0075  & ( 0.0034 ) & 0.0017  & ( 0.00051 ) \\ 
50  &3.3  & ( 17 ) & 0.0016  & ( 0.00035 ) & 0.0079  & ( 0.0035 ) & 0.0014  & ( 0.00034 ) \\ 
\end{tabular}

    ~ \\
    OLS + Ridge \\
    \begin{tabular}{r|rrrrrr} 
& \multicolumn{2}{c}{ Naive } & \multicolumn{2}{c}{ Offline } & \multicolumn{2}{c}{ Online } \\ 
\hline 
0  &0.0018  & ( 0.00058 ) & 0.0019  & ( 0.00066 ) & 0.0059  & ( 0.002 ) \\ 
2  &0.05  & ( 0.16 ) & 0.002  & ( 0.00091 ) & 0.0062  & ( 0.0025 ) \\ 
3  &0.025  & ( 0.1 ) & 0.0019  & ( 0.00058 ) & 0.0063  & ( 0.0029 ) \\ 
5  &0.043  & ( 0.14 ) & 0.0028  & ( 0.0096 ) & 0.011  & ( 0.044 ) \\ 
9  &0.11  & ( 0.2 ) & 0.018  & ( 0.11 ) & 0.022  & ( 0.1 ) \\ 
16  &0.16  & ( 0.23 ) & 0.0027  & ( 0.004 ) & 0.011  & ( 0.032 ) \\ 
28  &0.19  & ( 0.23 ) & 0.019  & ( 0.11 ) & 0.022  & ( 0.098 ) \\ 
50  &0.35  & ( 0.25 ) & 0.0048  & ( 0.009 ) & 0.01  & ( 0.0081 ) \\ 
\end{tabular}

  \end{center}
  \caption{\SR{}{
  Mean squared error for the estimation of the regression coefficients $\beta$ ($MSE(\beta)$) for OLS estimates without or with Ridge regularization for a sample size of $n = \nn$. 
  Rows: proportion of outliers (\%). 
  Values: mean (sd) over $B = 100$ simulations. 
  (Numeric version of Figure \ref{fig::mseBeta-n\nn}.)
  \label{tab::mseBeta-n\nn}}}
\end{table}

\FloatBarrier

\paragraph{\SR{}{\sl Effect of the sample size.}} ~

\begin{table}[ht]
  \begin{center}
    OLS, $\prop = 5\%$  \\
    \begin{tabular}{r|rrrrrrrr} 
& \multicolumn{2}{c}{ Naive } & \multicolumn{2}{c}{ Offline } & \multicolumn{2}{c}{ Online } & \multicolumn{2}{c}{ RRRR } \\ 
\hline 
100  &0.17  & ( 0.85 ) & 0.02  & ( 0.0062 ) & 1.6  & ( 3.1 ) & 0.023  & ( 0.009 ) \\ 
1000  &0.16  & ( 0.81 ) & 0.0019  & ( 0.00053 ) & 0.0073  & ( 0.0045 ) & 0.0019  & ( 0.00052 ) \\ 
10000  &0.22  & ( 1.1 ) & 0.00018  & ( 5.6e-05 ) & 0.00025  & ( 6.5e-05 ) & 0.00018  & ( 5.7e-05 ) \\ 
\end{tabular}

    ~ \\
    OLS, $\prop = 28\%$  \\
    \begin{tabular}{r|rrrrrrrr} 
& \multicolumn{2}{c}{ Naive } & \multicolumn{2}{c}{ Offline } & \multicolumn{2}{c}{ Online } & \multicolumn{2}{c}{ RRRR } \\ 
\hline 
100  &85  & ( 660 ) & 0.019  & ( 0.0054 ) & 2.4  & ( 4.5 ) & 0.021  & ( 0.0075 ) \\ 
1000  &5.2  & ( 30 ) & 0.0018  & ( 0.00046 ) & 0.0075  & ( 0.0034 ) & 0.0017  & ( 0.00051 ) \\ 
10000  &19  & ( 130 ) & 0.00017  & ( 3.5e-05 ) & 0.00024  & ( 4.9e-05 ) & 0.00016  & ( 3.7e-05 ) \\ 
\end{tabular}

    ~ \\
    OLS+Ridge, $\prop = 5\%$ \\
    \begin{tabular}{r|rrrrrr} 
& \multicolumn{2}{c}{ Naive } & \multicolumn{2}{c}{ Offline } & \multicolumn{2}{c}{ Online } \\ 
\hline 
100  &0.054  & ( 0.12 ) & 0.02  & ( 0.0072 ) & 0.18  & ( 0.072 ) \\ 
1000  &0.043  & ( 0.14 ) & 0.0028  & ( 0.0096 ) & 0.011  & ( 0.044 ) \\ 
10000  &0.051  & ( 0.15 ) & 0.00066  & ( 0.004 ) & 0.00076  & ( 0.0044 ) \\ 
\end{tabular}

    ~ \\
    OLS+Ridge, $\prop = 5\%$  \\  
    \begin{tabular}{r|rrrrrr} 
& \multicolumn{2}{c}{ Naive } & \multicolumn{2}{c}{ Offline } & \multicolumn{2}{c}{ Online } \\ 
\hline 
100  &7  & ( 57 ) & 0.055  & ( 0.15 ) & 0.45  & ( 2.1 ) \\ 
1000  &0.19  & ( 0.23 ) & 0.019  & ( 0.11 ) & 0.022  & ( 0.098 ) \\ 
10000  &0.28  & ( 0.28 ) & 0.013  & ( 0.083 ) & 0.47  & ( 4.6 ) \\ 
\end{tabular}

  \end{center}
  \caption{\SR{}{
  Mean squared error for the estimation of the regression coefficients $\beta$ ($MSE(\beta)$) for OLS estimates with or without Ridge regularization for a fraction of outliers of $\prop = 5\%$ and $\prop = 28\%$. 
  Same legend as Table \ref{tab::mseBeta-n\nn}
  (Numeric version of Figure \ref{fig::mseBeta-prop28-OLS}.)
  \label{tab::mseBeta-prop28-OLS}} }
\end{table}

\FloatBarrier

\paragraph{\SR{}{\sl Outlier detection.}} ~

\begin{table}[ht]
  \begin{center}
    OLS \\
    \begin{tabular}{r|rrrrrrrr} 
& \multicolumn{2}{c}{ Naive } & \multicolumn{2}{c}{ Offline } & \multicolumn{2}{c}{ Online } & \multicolumn{2}{c}{ RRRR } \\ 
\hline 
2  &0.67  & ( 0.089 ) & 0.67  & ( 0.091 ) & 0.67  & ( 0.091 ) & 0.67  & ( 0.091 ) \\ 
3  &0.67  & ( 0.081 ) & 0.67  & ( 0.083 ) & 0.67  & ( 0.083 ) & 0.67  & ( 0.083 ) \\ 
5  &0.66  & ( 0.065 ) & 0.67  & ( 0.064 ) & 0.67  & ( 0.065 ) & 0.67  & ( 0.064 ) \\ 
9  &0.66  & ( 0.044 ) & 0.66  & ( 0.046 ) & 0.67  & ( 0.046 ) & 0.67  & ( 0.046 ) \\ 
16  &0.64  & ( 0.032 ) & 0.66  & ( 0.033 ) & 0.66  & ( 0.033 ) & 0.66  & ( 0.033 ) \\ 
28  &0.63  & ( 0.03 ) & 0.66  & ( 0.028 ) & 0.66  & ( 0.028 ) & 0.66  & ( 0.028 ) \\ 
50  &0.6  & ( 0.025 ) & 0.65  & ( 0.019 ) & 0.65  & ( 0.019 ) & 0.66  & ( 0.019 ) \\ 
\end{tabular}

    ~ \\
    OLS+Ridge \\
    \begin{tabular}{r|rrrrrr} 
& \multicolumn{2}{c}{ Naive } & \multicolumn{2}{c}{ Offline } & \multicolumn{2}{c}{ Online } \\ 
\hline 
2  &0.67  & ( 0.088 ) & 0.67  & ( 0.091 ) & 0.67  & ( 0.09 ) \\ 
3  &0.67  & ( 0.08 ) & 0.67  & ( 0.083 ) & 0.67  & ( 0.083 ) \\ 
5  &0.66  & ( 0.065 ) & 0.67  & ( 0.065 ) & 0.67  & ( 0.065 ) \\ 
9  &0.66  & ( 0.043 ) & 0.66  & ( 0.045 ) & 0.67  & ( 0.045 ) \\ 
16  &0.65  & ( 0.031 ) & 0.66  & ( 0.033 ) & 0.66  & ( 0.033 ) \\ 
28  &0.63  & ( 0.031 ) & 0.66  & ( 0.027 ) & 0.66  & ( 0.028 ) \\ 
50  &0.62  & ( 0.029 ) & 0.65  & ( 0.019 ) & 0.65  & ( 0.019 ) \\ 
\end{tabular}

  \end{center}
  \caption{\SR{}{
  AUC for the detection of outliers for the OLS estimates with or without Ridge regularization for a sample size of $n = \nn$. 
  Same legend as Table \ref{tab::mseBeta-n\nn}
  (Numeric version of Figure \ref{fig::auc-n\nn-OLS}.)
  \label{tab::auc-n\nn-OLS}}}
\end{table}

\FloatBarrier

\paragraph{\SR{}{\sl Prediction accuracy.}} ~

\begin{table}[ht]
  \begin{tabular}{p{0.5\textwidth}p{0.5\textwidth}}
    \multicolumn{1}{c}{Offline} & \multicolumn{1}{c}{Online} \\
    \begin{tabular}{r|rrrr} 
& \multicolumn{2}{c}{ OLS } & \multicolumn{2}{c}{ WLS } \\ 
\hline 
0  &0.0019  & ( 0.00066 ) & 0.0018  & ( 0.00061 ) \\ 
2  &0.0019  & ( 0.00053 ) & 0.0018  & ( 0.00053 ) \\ 
3  &0.0019  & ( 0.00055 ) & 0.0018  & ( 0.00055 ) \\ 
5  &0.0019  & ( 0.00053 ) & 0.0018  & ( 5e-04 ) \\ 
9  &0.0019  & ( 0.00052 ) & 0.0018  & ( 0.00054 ) \\ 
16  &0.0018  & ( 0.00051 ) & 0.0018  & ( 0.00053 ) \\ 
28  &0.0018  & ( 0.00046 ) & 0.0018  & ( 0.00052 ) \\ 
50  &0.0016  & ( 0.00035 ) & 0.0017  & ( 0.00039 ) \\ 
\end{tabular}

    &
    \begin{tabular}{r|rrrr} 
& \multicolumn{2}{c}{ OLS } & \multicolumn{2}{c}{ WLS } \\ 
\hline 
0  &0.0067  & ( 0.0033 ) & 0.003  & ( 0.0011 ) \\ 
2  &0.0068  & ( 0.0034 ) & 0.0031  & ( 0.0014 ) \\ 
3  &0.0077  & ( 0.008 ) & 0.0032  & ( 0.0017 ) \\ 
5  &0.0073  & ( 0.0045 ) & 0.0029  & ( 0.0011 ) \\ 
9  &0.0071  & ( 0.0036 ) & 0.0032  & ( 0.0014 ) \\ 
16  &0.0074  & ( 0.0041 ) & 0.0029  & ( 0.0011 ) \\ 
28  &0.0075  & ( 0.0034 ) & 0.0031  & ( 0.0014 ) \\ 
50  &0.0079  & ( 0.0035 ) & 0.0029  & ( 0.00096 ) \\ 
\end{tabular}

  \end{tabular}
  \caption{\SR{}{
  Comparison of the mean squared error for the estimation of the regression coefficients $\beta$ ($MSE(\beta)$) of OLS and WLS estimates for Offline and Online procedures without Ridge regularization for a sample size of $n = \nn$.
  Legend:  \textcolor{red}{$\medcirc$} = OLS, \textcolor{green}{$\triangle$} = WLS. Number in each box = fraction $\prop$ of outliers (in \%).     
  (Numeric version of Figure \ref{fig::compOW-n\nn-noRidge}.)
  \label{tab::compOW-n\nn-noRidge}}}
\end{table}

\FloatBarrier

\end{document}